\documentclass[a4paper,oneside,american]{amsart}
\usepackage[LGR,T1]{fontenc}
\usepackage{textcomp}
\usepackage{bm}
\usepackage{amstext}
\usepackage{amsthm}
\usepackage{amssymb}
\usepackage{graphicx}
\usepackage[all]{xy}

\makeatletter

\DeclareRobustCommand{\greektext}{%
  \fontencoding{LGR}\selectfont\def\encodingdefault{LGR}}
\DeclareRobustCommand{\textgreek}[1]{\leavevmode{\greektext #1}}

\numberwithin{equation}{section}
\theoremstyle{plain}
\newtheorem{thm}{\protect\theoremname}[section]
\theoremstyle{definition}
\newtheorem{defn}[thm]{\protect\definitionname}
\theoremstyle{remark}
\newtheorem{rem}[thm]{\protect\remarkname}
\theoremstyle{plain}
\newtheorem{prop}[thm]{\protect\propositionname}
\theoremstyle{plain}
\newtheorem{lem}[thm]{\protect\lemmaname}
\theoremstyle{plain}
\newtheorem{conjecture}[thm]{\protect\conjecturename}

\newcounter{myparagraph}[subsection]

\renewcommand{\themyparagraph}{\bf {\arabic{section}.\arabic{subsection}.\alph{myparagraph}}}

\newenvironment{myitem}{\begin{list}{}{
\setlength{\leftmargin}{0.8cm}
\setlength{\itemindent}{-0.5cm}
\setlength{\itemsep}{2pt}
}}{\end{list}}

\def\hsp#1{{\hspace{ #1 pt}}}

\def\*{\hsp{-3p}*\hsp{-3pt}}

\usepackage[all]{xy}

\usepackage{color}
\usepackage{amsthm}
\usepackage{bm}
\usepackage{graphicx, color, ulem}

\usepackage[labelformat=empty]{caption}

\makeatother

\usepackage{babel}
\providecommand{\conjecturename}{Conjecture}
\providecommand{\definitionname}{Definition}
\providecommand{\lemmaname}{Lemma}
\providecommand{\propositionname}{Proposition}
\providecommand{\remarkname}{Remark}
\providecommand{\theoremname}{Theorem}

\begin{document}
\def\m{\hsp{-4}}

\def\n{\hsp{-2}}

\global\long\def\tx{\mathtt{x}}%
\global\long\def\ty{\mathtt{y}}%
\global\long\def\tz{\mathtt{z}}%
\global\long\def\tw{\mathtt{w}}%

\global\long\def\tA{\mathtt{A}}%

\global\long\def\s{\mathtt{s}}%

\global\long\def\e{\bm{e}}%

~
\begin{flushright}
~%Sep.1.2025
\par\end{flushright}
\title{families of Calabi-Yau manifolds and mirror symmetry}
\dedicatory{Dedicated to Professor S.-T. Yau on the occasion of his 75th birthday}
\author{Shinobu Hosono}
\begin{abstract}
We survey mirror symmetry of Calabi-Yau manifolds from the perspective
of families of Calabi-Yau manifolds and their period integrals. Special
emphasis is laid on distinguished properties of the hypergeometric
series of Gel'fand, Kapranov, and Zelevinsky that appear in mirror
symmetry. After defining mirror symmetry in terms of families of Calabi-Yau
manifolds, we summarize a general construction of moduli spaces of
Calabi-Yau hypersurfaces (complete intersections) in toric varieties.
We review the central charge formula, and assuming it, we show mirror
symmetry for the pairs of Calabi-Yau manifolds associated with reflexive
polytopes. By describing the moduli spaces globally, we present interesting
examples of Calabi-Yau manifolds where birational geometry and geometry
of Fourier-Mukai partners of a Calabi-Yau manifold arise from the
study of mirror symmetry. 
\end{abstract}

\maketitle
\tableofcontents

\section{Introduction}

Mirror symmetry of Calabi-Yau manifolds has been one of the central
topics in theoretical physics and mathematics since its discovery
at the beginning of the 1990s. It has been providing us with a strong
motivation to study moduli spaces of Calabi-Yau manifolds. In particular,
the pioneering work on period integrals and its surprising application
to enumerative geometry by Candelas et al \cite{Cand} indicates that
there is a rich and nontrivial structure encoded in the moduli spaces
of Calabi-Yau manifolds. 

Toward a mathematical understanding of mirror symmetry, investigations
over the last three decades have led us to two formulations of mirror
symmetry. One is based on the geometry of Calabi-Yau manifolds, called
geometric mirror symmetry due to Strominger-Yau-Zaslow \cite{SYZ}
and also Gross-Siebert \cite{GS1,GS2}, and the other focuses on equivalences
of two different categories: the derived category of coherent sheaves
and the derived Fukaya category in symplectic geometry for a Calabi-Yau
manifold. The latter is called homological mirror symmetry \cite{Ko}.
Both approaches provide general frameworks to understand the symmetry.
On the other hand, to fully understand the symmetry, it seems still
important to have many examples of Calabi-Yau manifolds where we observe
the symmetry in a variety of different forms. 

In this article, combining the calculations in \cite{Cand-II,Cand-IIs}
and \cite{HKTY1,HKTY2} with the homological mirror symmetry, we define
mirror symmetry of $X$ and $\check{X}$ as a certain isomorphism
between the $A$-structure of $X$ and the $B$-structure of $\check{X}$.
The $A$-structure of $X$ is described simply by the hard Lefschetz
theorem. On the other hand, to describe the $B$-structure of $\check{X}$,
we require a deformation family of Calabi-Yau manifolds of $\check{X}$
over some parameter space $\mathcal{M}_{\check{X}}^{0}$ and a special
degeneration point on a compactification $\mathcal{M}_{\check{X}}$
of $\mathcal{M}_{\check{X}}^{0}$. The latter requirement is exactly
the problem which was studied for a few concrete examples at the early
stage of mirror symmetry in \cite{Cand-II,Cand-IIs} and \cite{HKTY1,HKTY2}.
In particular in \cite{HKTY1,HKTY2}, this problem has been connected
to the problem of constructing local solutions of Gel'fand, Kapranov,
and Zelevensky (GKZ) hypergeometric system \cite{GKZ}. We will survey
this problem from the viewpoint of making suitable moduli spaces for
families of (complete intersection) Calabi-Yau manifolds. After the
construction of moduli spaces, we will review the \textit{central
charge formula} (conjecture) introduced in \cite{Hos}. Assuming this
conjecture, we show mirror symmetry in general for Calabi-Yau complete
intersections in toric varieties due to Batyrev and Borisov \cite{Bat,BB}.
We also survey interesting observations made in \cite{HTmov,HTAbCY},
where birational geometry and/or Fourier-Mukai partners of Calabi-Yau
manifolds arise from the global study of moduli spaces of mirror Calabi-Yau
manifolds. We will restrict our attention mostly to Calabi-Yau threefolds.
However, most of the definitions and arguments apply to elliptic curves
and K3 surfaces with straightforward modifications. 

The construction of this article is as follows: In Section \ref{sec:Families-of-Calabi-Yau},
we will introduce $A$-structure and $B$-structure of Calabi-Yau
manifolds, and define mirror symmetry in terms of these structures.
In Section \ref{sec:Sec3}, we summarize the constructions of pairs
of Calabi-Yau manifolds in toric varieties by the so-called reflexive
polytopes following the works \cite{Bat,BB}. We define moduli spaces
of general Calabi-Yau hypersurfaces (complete intersections) in terms
of the GIT quotients by natural torus actions on the parameters in
defining equations. We define period integrals of families of Calabi-Yau
manifolds over the moduli spaces and, following \cite{BC}, we note
that period integrals are solutions of the Gel'fand-Kapranov-Zelevinsky
(GKZ) hypergeometric system defined over the toric varieties of the
secondary fan. The torus actions on the parameters of these Calabi-Yau
manifolds come from the automorphism groups of the ambient toric varieties.
The automorphism groups of toric varieties are larger than the algebraic
tori in general. In Subsect.\ref{subsec:Non-reductive-gr}, we present
two examples where the toric varieties have larger automorphism groups
than tori. In Subsect.\ref{subsec:K3lambda}, we provide examples
where we need to preserve certain special forms of defining equations
for some symmetry reasons. Examples in these two subsections are given
by elliptic curves and K3 surfaces. However, in these subsections,
we present a detailed study of the relevant moduli spaces. In Sect.\ref{sec:Sec4},
we review the central charge formula (conjecture) introduced in \cite{Hos}.
Assuming the central charge formula, we show that, for a topological
mirror pair $(X,\check{X}$) due to Batyrev-Borisov, the $B$-structure
of the Calabi-Yau manifold $\check{X}$ is isomorphic to the $A$-structure
of $X$ and vice versa. In Sect.\ref{sec:Birat-and-FM}, we will present
briefly two interesting examples where birational geometry and/or
Fourier-Mukai partners of $X$ arise from studying the moduli space
of a mirror Calabi-Yau manifold $\check{X}$ globally. Some related
subjects are discussed in Sect.\ref{sec:Summary}. In the Appendix,
a derivation of the canonical form of integral and symplectic basis
of period integrals (or local solutions of the GKZ system) is sketched. 

\vskip0.3cm\noindent \textbf{Acknowledgments}: This survey article\footnote{This is a substantial extension of a short report \cite{HT-matrix}
submitted to the MATRIX Annals.} is based on the author\textquoteright s collaborations with many
people on topics related to the title. The first one goes back to
a collaboration with Professor S.-T. Yau, Albrecht Klemm, and Stefan
Theisen {[}57{]}, which was carried out while he was a postdoctoral
fellow with Professor Yau at Harvard University (1992--1993). He
is deeply grateful to Professor Yau for the opportunity to initiate
research on this interesting and stimulating topic, and for his nearly
three decades of subsequent collaborations. He would like to thank
Albrecht Klemm, Bong H. Lian, Keiji Oguiso, Hiromichi Takagi, and
Stefan Theisen for their collaborations on the subjects surveyed in
this article.

This article was completed during the author's visit to the Max-Planck-Institut
f\"ur Mathematik in Bonn for his sabbatical. He would like to thank
the institute for its excellent research environment and the warm
hospitality during his stay. This work is also supported in part by
Grants-in-Aid for Scientific Research C 24K06743. 

~

~

\section{\protect\label{sec:Families-of-Calabi-Yau}Families of Calabi-Yau
manifolds and mirror symmetry}

\subsection{Calabi-Yau manifolds}

Complex manifolds which admit Ricci-flat K\"{a}hler metrics are called
Calabi-Yau manifolds in general. Abelian varieties or complex tori
are simple examples of such Calabi-Yau manifolds. Here in this article,
however, we adopt the following definition of Calabi-Yau manifolds
in a narrow sense, where abelian varieties and complex tori are excluded. 
\begin{defn}
We call a non-singular projective variety $X$ of dimension $d$ Calabi-Yau
manifold if it satisfies $c_{1}(X)=0$ and $H^{i}(X,\mathcal{O}_{X})=0\,(1\leq i\leq d-1)$. 
\end{defn}

Given a Calabi-Yau manifold $X$, we can describe the deformation
of its complex structures by the Kodaira-Spencer theory. In particular,
due to the theorem by Bogomolov-Tian-Todorov, all infinitesimal deformations
are unobstructed and given by the cohomology $H^{1}(X,TX)$. We denote
by $\mathcal{M}$ the deformation space, and write by $X_{p}$ the
Calabi-Yau manifold represented by $p\in\mathcal{M}$. According to
the deformation theory, we have a linear map from the tangent space 

\[
KS_{p}:T_{p}\mathcal{M}\rightarrow H^{1}(X_{p},TX),
\]
which is called Kodaira-Spencer map. 
\begin{defn}
Let $X$ be a Calabi-Yau manifold and $\mathfrak{X},\mathcal{M}$
be complex manifolds. If there exists a surjective morphism $\pi:\mathfrak{X}\rightarrow\text{\ensuremath{\mathcal{M}}}$
satisfying the following conditions: (0) $\pi^{-1}(p_{0})\simeq X$,
$(1)$ differential $\pi_{*}$ is submersive, (2) the fiber $\pi^{-1}(p)$
is compact for all $p\in\mathcal{M}$, (3) the Kodaira-Spencer maps
are isomorphic for all $p\in\mathcal{M}$, then we call $\mathfrak{X}$
a deformation family of $X$. 
\end{defn}

It is often required that the morphism $\pi:\mathfrak{X}\rightarrow\text{\ensuremath{\mathcal{M}}}$
is proper, however, we do not require this property following \cite[Thm. 2.8]{Kod}.

\subsection{$A$-structure and $B$-structure}

For a given Calabi-Yau manifold, we introduce two different but similar
structures. In what follows, we assume Calabi-Yau manifolds are of
dimension three.

\subsubsection{\protect\label{subsec:A-structure}A-structure}

Since a Calabi-Yau manifold $X$ is projective by definition, we can
take a positive line bundle $L$ with its first Chern class $c_{1}(L)\in H^{1,1}(X,\mathbb{Z})$
being primitive. We take a K\"{a}hler form $\kappa$ satisfying $[\kappa]=[c_{1}(L)]$.
The K\"{a}hler form $\kappa$ is a $(1,1)$ form, and determines the
nilpotent linear map $L_{\kappa}:H^{p,q}(X)\rightarrow H^{p+1,q+1}(X)$
in the Hard Lefschetz theorem. For Calabi-Yau manifolds, since it
acts trivially on $H^{3}(X)=\oplus_{p+q=3}H^{p,q}(X)$ , we can restrict
it on $H^{even}(X)=\oplus_{p+q:\text{even}}H^{p,q}(X)$ and obtain
\begin{equation}
L_{\kappa}:H^{even}(X,\mathbb{Q})\rightarrow H^{even}(X,\mathbb{Q}),\label{eq:A-str}
\end{equation}
where we use the fact that $\kappa$ is a real $(1,1)$ form.

For coherent sheaves $\mathcal{E},\mathcal{F}$ on $X$, we denote
the Riemann-Roch (RR) pairing by 
\[
\chi(\mathcal{E},\mathcal{F})=\sum_{i=0}^{3}(-1)^{i}\dim Ext^{i}(\mathcal{E},\mathcal{F}).
\]
This RR pairing is a pairing on the Grothendieck group $K(X)$ of
coherent sheaves on $X$. The numerical $K$ group $K_{num}(X):=K(X)/\equiv$
is defined by the quotient by the radical of RR pairing. $K_{num}(X)$
defines a free abelian group. For Calabi-Yau manifolds, due to the
Serre duality, RR pairing introduces a skew symmetric form on $K_{num}(X)$,
which we denote by $(K_{num}(X),\chi(-,-))$. Moreover, by the Lefshetz
(1,1) theorem and Hard-Lefschetz theorem, we can see that the Chern
character homomorphism $ch:K_{num}(X)\rightarrow H^{even}(X,\mathbb{Q})$
gives rise to an isomorphism $ch(K_{num}(X))\otimes\mathbb{Q}\simeq$$H^{even}(X,\mathbb{Q})$
for Calabi-Yau threefolds. Namely, by the homomorphism $ch:K_{num}(X)\rightarrow H^{even}(X,\mathbb{Q})$,
we have a natural integral structure with a skew symmetric form on
$H^{even}(X,\mathbb{Q})$ coming from the structure $(K_{num}(X),\chi(-,-))$. 
\begin{defn}
We define $A$-structure of a Calabi-Yau manifold $X$ by the nilpotent
linear map $L_{\kappa}:H^{even}(X,\mathbb{Q})\rightarrow H^{even}(X,\mathbb{Q})$
together with the integral structure on $H^{even}(X,\mathbb{Q})$
coming from $(K_{num}(X),\chi(-,-))$. 
\end{defn}

\subsubsection{\protect\label{subsec:B-str.}B-structure }

For a Calabi-Yau manifold $X$, we can define a similar structure
to the $A$-structure defined above, but we need to require some additional
assumptions on $X$. The first assumption is that $X$ has a deformation
family $\pi:\mathfrak{X}\rightarrow\mathcal{M}^{0}$ over a parameter
space which is quasi-projective. We assume that there is a compactification
$\mathcal{M}$ of $\mathcal{M}^{0}$ such that the fiber $X_{p}=\pi^{-1}(p)$
degenerates to a (singular) Calabi-Yau variety $X_{q}$ when we take
a limit $p\rightarrow q\in\mathcal{M}\setminus\mathcal{M}^{0}$. The
compactification $\mathcal{M}^{0}$ is assumed to be a (singular)
projective variety whose singular loci are contained in $\mathcal{M}\setminus\mathcal{M}^{0}$.
Also, we assume $\mathcal{M}\setminus\mathcal{M}^{0}$ is given by
a (non-normal crossiong) divisor $D$ in $\mathcal{M}$. Let $\bar{o}$
be a singular point on the divisor $D$ and take the following steps:

1. Let $U_{\bar{o}}$ to be an affine neighborhood of $\bar{o}$.
By successive blowing-ups starting at $\bar{o}$, we assume the divisor
$U_{\bar{o}}\cap D$ is transformed to 
\[
\hat{U}_{\bar{o}}\setminus U_{\bar{o}}^{sm}=\cup_{i}D_{i},
\]
in terms of normal crossing boundary (exceptional) divisors, where
$U_{\bar{o}}^{sm}$ represents $U_{\bar{o}}\setminus D$. We focus
on a boundary point given by $o=D_{i_{1}}\cap\cdots\cap D_{i_{r}}$($r=\dim\mathcal{M}$).

2. Recall that for a family of Calabi-Yau manifolds $\pi:\mathfrak{X}\rightarrow\mathcal{M}^{0}$,
we naturally have a locally constant sheaf $R^{3}\pi_{*}\mathbb{C}_{\mathfrak{X}}$
over $\mathcal{M}^{0}$. This sheaf corresponds to a holomorphic vector
bundle whose fibers are $H^{3}(X_{m},\mathbb{C})$ ($m\in\mathcal{M}^{0}$)
with the holomorphic flat connection (Gauss-Manin connection). Each
fiber has (topological) integral structure $H^{3}(X_{m},\mathbb{Z})$
and also a skew symmetric (symplectic) form defined by $(\alpha,\beta):=\int_{X_{m}}\alpha\wedge\beta$.
The Gauss-Manin connection is naturally compatible with this symplectic
structure.

3. Take a boundary point $o\in\hat{U}_{\bar{o}}$ as described above,
and write it $o=D_{1}\cap\cdots\cap D_{r}$$(r=\dim\mathcal{M})$
by reordering the indices of the (exceptional) divisors. Choose a
base point $b_{o}\in\mathcal{M}^{0}$ and fix an integral symplectic
basis of $H^{3}(X_{b_{o}},\mathbb{Z})$. With respect to this basis,
we represent the monodromy of the Gauss-Manin connection around the
boundary divisor $D_{i}$ by $T_{D_{i}}:H^{3}(X_{b_{o}},\mathbb{Z})\rightarrow H^{3}(X_{b_{o}},\mathbb{Z})$.
We assume that the monodromy matrices $T_{D_{i}}$ for $i=1,\cdots,r$
are unipotent, i.e., satisfy 
\[
(T_{D_{i}}-\mathrm{id})^{m_{i}}=O\,\,\,\text{for some }m_{i}.
\]
Note that this is a condition for the boundary point $o\in\hat{U}_{\bar{o}}$.
For each $T_{D_{i}}$, we define 
\[
\begin{aligned}\log T_{D_{i}} & =\log(\mathrm{id}+(T_{D_{i}}-\mathrm{id}))\\
 & =(T_{D_{i}}-\mathrm{id})-\frac{1}{2}(T_{D_{i}}-\mathrm{id})^{2}+\cdots+\frac{(-1)^{m_{i}-1}}{m_{i}-1}(T_{D_{i}}-\mathrm{id})^{m_{i}-1}
\end{aligned}
\]
and denote this by $N_{i}=\log T_{D_{i}}$. By definition, $N_{i}$
determines a nilpotent endomorphism in $\mathrm{End}(H^{3}(X_{b_{o}},\mathbb{Q}))$.
Moreover, since the boundary is a normal crossing divisor, we have
$N_{i}N_{j}-N_{j}N_{i}=O$. Now we define $N_{\lambda}:=\sum_{i}\lambda_{i}N_{i}$
in terms of real parameters $\lambda_{i}>0$, which gives a nilpotent
endomorphism acting on $H^{3}(X_{b_{o}},\mathbb{R})$ since $N_{i}$'s
commute with each other. The following result is due to Cattani and
Kaplan \cite[Thm.2]{Cattani}: 
\begin{thm}
For general $\lambda_{i}>0(\lambda_{i}\in\mathbb{Q})$ , the nilpotent
matrix $N_{\lambda}:=\sum_{i}\lambda_{i}N_{i}$ defines a monodromy
weight filtration of the same form on $H^{3}(X_{b_{o}},\mathbb{Q})$. 
\end{thm}

As a general result of monodromy theorem, the nilpotent matrix $N_{\lambda}$
satisfies $N_{\lambda}^{k}=O$ for $k>\dim X(=3)$. When we have $N_{\lambda}^{3}\not=O$,
$N_{\lambda}^{4}=O$, we call the boundary point $o\in\hat{U}_{\bar{o}}$
maximally degenerated. If a boundary point $o\in\hat{M}_{\bar{o}}$
is maximally degenerated and moreover the monodromy weight filtration
by $N_{\lambda}$ has the form 
\[
0\subset W_{0}=W_{1}\subset W_{2}=W_{3}\subset W_{4}=W_{5}\subset W_{6}=H^{3}(X_{b_{o}},\mathbb{Q}),
\]
then we call $o$ a LCSL (Large Complex Structure Limit) point. 
\begin{defn}
If a Calabi-Yau manifold $X$ has a deformation family $\pi:\mathfrak{X}\rightarrow\mathcal{M}$
which gives rise to a LCSL boundary point $o\in\hat{U}_{\bar{o}}$,
then we define $B$-structure of $X$ from $o$ by the nilpotent action
$N_{\lambda}:H^{3}(X_{b_{o}},\mathbb{Q})\rightarrow H^{3}(X_{b_{o}},\mathbb{Q})$
together with the integral structure $(H^{3}(X_{b_{o}},\mathbb{Z}),(\,\,,\,\,))$. 
\end{defn}

\noindent\vskip0.2cm
\begin{rem}
\noindent (1) It is not clear whether every Calabi-Yau manifold admits
a family $\mathfrak{X}\rightarrow\mathcal{M}^{0}$ from which we can
define a B-structure by finding a compactification $\mathcal{M}$.
However, for Calabi-Yau hypersurfaces or complete intersections in
toric varieties, as we will do in the following sections, we can construct
explicitly their families where we find the $B$-structures. 

\noindent (2) The $B$-structure is not necessarily unique for a Calabi-Yau
manifold $X$. That is, we often observe that many different $B$-structures
of $X$ arise from different LCSL points of a family $\mathfrak{X}\rightarrow\mathcal{M}$.
We will see some examples in Sect. \ref{sec:Birat-and-FM}.
\end{rem}

\subsubsection{Mirror symmetry}

\noindent In contrast to the $A$-structure, describing the $B$-structure
of a Calabi-Yau manifold is difficult as well as showing its existence.
Mirror symmetry of Calabi-Yau manifolds is a conjecture that says
for a given Calabi-Yau manifold $X$, there exists another Calabi-Yau
manifold $\check{X}$ such that the $A$-structure of $X$ is isomorphic
to a $B$-structure from $\check{X}$.

To describe the symmetry in more precise, let $\mathcal{K}_{X}$ be
the K\"{a}hler cone of $X$. The K\"{a}hler cone is an open convex
cone in $H^{2}(X,\mathbb{R})$ and coincides with the ample cone in
our case ($\dim X=3$) and satisfies $\mathcal{K}_{X}\cap H^{2}(X,\mathbb{Z})\not=\phi$.
The shape of the ample cone is complicated in general and is related
to the classification problem of algebraic varieties. Take integral
elements $\kappa_{1},\cdots,\kappa_{r}\in\overline{\mathcal{K}}_{X}\cap H^{2}(X,\mathbb{Z})$
so that these generate the group $H^{2}(X,\mathbb{Z})$ and define
a cone $\sigma_{A}:=\mathbb{R}_{\geq0}\kappa_{1}+\cdots+\mathbb{R}_{\geq0}\kappa_{r}$.
Cones of this property may be constructed by decomposing the cone
$\overline{\mathcal{K}}_{X}$ into (possibly infinitely many) simplicial
cones. We determine a cone $\sigma_{A}$ and introduce a K\"{a}hler
class $\kappa$ by a general linear combination 
\[
\kappa=\alpha_{1}\kappa_{1}+\cdots+\alpha_{r}\kappa_{r}
\]
with $\alpha_{i}>0$. The nilpotent operator $L_{\kappa}$ defines
the A-structure of $X$ with parameters $\alpha_{i}$, which we denote
by 
\begin{equation}
\big(L_{\kappa(\alpha)},(K_{num}(X),\chi(-,-)),\Sigma_{A}\big),\label{eq:A-str-a}
\end{equation}
where $\Sigma_{A}:=\left\{ \alpha_{1}L_{\alpha_{1}}+\cdots+\alpha_{r}L_{\alpha_{r}}\mid\alpha_{i}>0\right\} $
is a cone in $\mathrm{End}(H^{even}(X,\mathbb{R}))$ and corresponds
to the cone $\sigma_{A}\in H^{2}(X,\mathbb{Z})$.

The nilpotnet matrix $N_{\lambda}=\sum_{i}\lambda_{i}N_{i}$ which
describe the $B$-structure may be regarded as an element in a cone
$\Sigma_{o}:=\mathbb{R}_{\geq0}N_{1}+\cdots+\mathbb{R}_{\geq0}N_{m}$
$(m=\dim H^{1}(X,TX))$ for a LCSL boundary point $o$. This cone
$\Sigma_{o}$ is called monodromy nilpotent cone in the Hodge theory
\cite{Cattani,Cattani2}. We denote the $B$-structure coming from
a LCSL boundary point $o$ by 
\begin{equation}
\big(N_{\lambda},(H^{3}(X_{b_{o}},\mathbb{Z}),(\,\,,\,\,)),\Sigma_{o}\big).\label{eq:B-str-lambda}
\end{equation}
Here $X_{b_{o}}=\pi^{-1}(b_{o})$ represents a smooth fiber over a
base point $b_{o}\in\mathcal{M}$, but this will often be simplified
by $X$. 
\begin{defn}
The $A$-structure of a Calabi-Yau manifold $X$ and the $B$-structure
of another Calabi-Yau manifold $Y$ are said isomorphic if the following
two properties holds: 

\begin{myitem} 

\item[$(1)$] There exists an isomorphism 
\[
\begin{matrix}\varphi: & H^{even}(X,\mathbb{R}) & \stackrel{\sim}{\rightarrow} & H^{3}(Y,\mathbb{R})\\
 & \cup &  & \cup\\
 & (K_{num}(X),\chi(-,-)) &  & (H^{3}(Y,\mathbb{Z}),(\,\,,\,\,)),
\end{matrix}
\]
which preserves the integral and skew symmetric forms in the both
sides. 

\item[$(2)$] When we set $\alpha_{i}=c\lambda_{i}$ with some constant
$c\in\mathbb{R}$, the nilpotent matrices $N_{\lambda}$ and $L_{\kappa(\alpha)}$
are related by the isomorphism in (1) as 
\[
N_{\lambda}=\varphi\circ L_{\kappa(\alpha)}\circ\varphi^{-1}.
\]
\end{myitem} 
\end{defn}

Now we are ready to define mirror symmetry (of Calabi-Yau threefolds): 
\begin{defn}
\label{def:mirror-def} Two Calabi-Yau manifolds $X$ and $\check{X}$
are said mirror to each other if they have their $A$- and $B$- structures
which are exchanged under the following isomorphisms $\varphi$ and
$\check{\varphi}$: 
\[
\begin{alignedat}{1}\mathrm{(i)\,} & \,\,\varphi:\big(L_{\kappa(\alpha)},(K_{num}(X),\chi(-,-)),\Sigma_{A}\big)\stackrel{\sim}{\rightarrow}\big(N_{\lambda},(H^{3}(\check{X},\mathbb{Z}),(\,\,,\,\,)),\Sigma_{o}\big),\\
\mathrm{(ii)} & \,\,\check{\varphi}:\big(L_{\kappa(\check{\alpha})},(K_{num}(\check{X}),\chi(-,-)),\Sigma_{\check{A}}\big)\stackrel{\sim}{\rightarrow}\big(N_{\check{\lambda}},(H^{3}(X,\mathbb{Z}),(\,\,,\,\,)),\Sigma_{\check{o}}\big).
\end{alignedat}
\]
\end{defn}

\begin{rem}
(1) From the isomorphism between the $A$- and $B$- structures, it
holds 
\begin{equation}
h^{1,1}(X)=h^{2,1}(\check{X}),\;h^{1,1}(\check{X})=h^{2,1}(X)\label{eq:hodge-mirror}
\end{equation}
when $X$ and $\check{X}$ are mirror symmetric. 

\noindent (2) When $X$ and $\check{X}$ are mirror symmetric, we
have families $\mathfrak{X}\rightarrow\mathcal{M}_{X}^{0}$ and $\check{\mathfrak{X}}\rightarrow\mathcal{M}_{\check{X}}^{0}$.
We call, for example, the family $\check{\mathfrak{X}}\rightarrow\mathcal{M}_{\check{X}}^{0}$
a mirror family of $X$.

\noindent (3) A quintic hypersurface $X=(5)$ in $\mathbb{P}^{4}$
is an example of Calabi-Yau manifold. In this case, we have $h^{1,1}(X)=1,h^{2,1}(X)=101$
for the Hodge numbers of $X$. A Calabi-Yau manifold $\check{X}$
having Hodge numbers $h^{1,1}(\check{X})=101,h^{2,1}(\check{X})=1$
has been constructed in the work by Candelas et~al. in 1991 \cite{Cand},
and surprising implications of mirror symmetry have been found there.
In this case, while the isomorphism (i) has been shown, verifying
the isomorphism (ii) seems difficult since the dimension of the deformation
space is large, i.e., $h^{2,1}(X)=101$. Beyond quintic hypersurfaces
in $\mathbb{P}^{4}$, we can construct pairs $(X,\check{X})$ of Calabi-Yau
hypersurfaces or complete intersections satisfying the relation (\ref{eq:hodge-mirror})
in toric varieties \cite{Bat,BB}. If the dimensions of the deformation
spaces are small enough, we can verify either the isomorphism (i)
or (ii) \cite{HKTY1,HKTY2}. However verifying both relations (i)
and (ii) is difficult in general, since the relation (\ref{eq:hodge-mirror})
implies that the smaller dimensions of the deformation space for $X$,
the larger dimensions for $\check{X}$ if the Euler numbers $e(X)=-e(\check{X})$
have a large absolute value. Assuming the central charge formula (Conjecture
\ref{conj:central-charge}), we will show (i) and (ii) in general
for Calabi-Yau hypersurfaces (complete intersections) in toric varieties. 

~

~
\end{rem}

\section{\protect\label{sec:Sec3}Families of Calabi-Yau manifolds and differential
equations}

We continue to say Calabi-Yau manifolds to mean Calabi-Yau threefolds
unless mentioned otherwise. The complete classification of Calabi-Yau
manifolds is not known; however, we have a full list of Calabi-Yau
hypersurfaces in four dimensional toric Fano varieties \cite{KSlist}
and a complete list of Calabi-Yau complete intersections in products
of projective spaces \cite{cicyList1,cicyList2}. For these cases,
Batyrev-Borisov mirror construction applies. And we can find a topological
mirror manifold $\check{X}$ for a given $X$ which satisfies $h^{1,1}(X)=h^{2,1}(\check{X})$
and $h^{2,1}(X)=h^{1,1}(\check{X})$. To verify mirror symmetry (Definition
\ref{def:mirror-def}), we need to have a family of Calabi-Yau manifolds
and differential equations associated with the family. 

\subsection{\protect\label{subsec:Mf-toric-hyper-S}Moduli spaces of Calabi-Yau
hypersurfaces in toric varieties}

\subsubsection{Batyrev mirror }

Let $N\simeq\mathbb{Z}^{4}$ be a lattice and $M:=\mathrm{Hom}(N,\mathbb{Z})\simeq\mathbb{Z}^{4}$
be the dual lattice. A polytope $\Delta\subset M\otimes\mathbb{R}$
is called a lattice polytope if its all vertices (corners) are integral.
Given a polytope, we have a projective toric variety $\mathbb{P}_{\Delta}$
whose dense torus is $(\mathbb{C}^{*})^{4}=\mathrm{Hom}(N,\mathbb{C}^{*})$.
In terms of the normal fan $\mathcal{N}(\Delta)\in N\otimes\mathbb{R}$,
we have $\mathbb{P}_{\Delta}=X_{\mathcal{N}(\Delta)}$ where $X_{\Sigma}$
represents a toric variety defined for a fan. 

The polar dual of a lattice polytope $\Delta$ is defined by 
\[
\Delta^{*}=\left\{ y\in M\otimes\mathbb{R}\mid\langle y,x\rangle\geq-1\,(\,^{\forall}x\in\Delta)\right\} .
\]
A lattice polytople $\Delta$ is called reflexive if it contains the
origin in its interior and its polar dual $\Delta^{*}\subset M\otimes\mathbb{R}$
is a lattice polytope. For a reflexive polytope $\Delta$, its dual
$\Delta^{*}$ contains the origin and also $(\Delta^{*})^{*}=\Delta$
is a lattice polytope. Hence reflexive polytopes come with pairs $(\Delta,\Delta^{*})$.
It is useful to note that the origin is the only inner integral point
for a reflexive polytope. 

For a reflexive polytope $\Delta$, we denote by $\Sigma(\Delta^{*})\subset N\otimes\mathbb{R}$
the fan made by cones over the faces of $\Delta^{*}$. Then it holds
that $\mathcal{N}(\Delta)=\Sigma(\Delta^{*})$ for the normal fan.
All integral points $\Delta^{*}\cap N$, except the origin, appears
on the faces of $\Delta^{*}$. Using these points, we make a subdivision
of the fan $\Sigma(\Delta^{*})$ to obtain a simplicial fan $\hat{\Sigma}(\Delta^{*})$.
We take one from possibly many choices for the subdivision $\hat{\Sigma}(\Delta^{*})$.
The subdivision $\hat{\Sigma}(\Delta^{*})$ consists of simplicial
fans which are not necessarily smooth, i.e., with simplicial volume
greater than one. Because of this, the toric variety $\hat{\mathbb{P}}_{\Delta}:=X_{\hat{\Sigma}(\Delta^{*})}$
defines a partial resolution of $\mathbb{P}_{\Delta}$, which is called
a maximally projective crepant partial (MPCP) resolution in \cite{Bat}. 

The toric variety $\mathbb{P}_{\Delta}$ associated to a reflexive
polytope $\Delta$ is called Fano toric variety, where the integral
points of the polytope $\Delta$ represents the global sections of
the anti-canonical bundle $-K_{\mathbb{P}_{\Delta}}$. This applies
also to $\Delta^{*}$. Based on this, we consider the Laurent polynomials,
with general parameters $a_{\mu},b_{\nu}$,
\[
f_{\Delta^{*}}=\sum_{\mu\in\Delta^{*}\cap N}a_{\mu}u^{\mu},\,\,\,f_{\Delta}=\sum_{\nu\in\Delta\cap M}b_{\nu}v^{\nu}
\]
in terms of torus coordinates $u=(u_{1},..,u_{4})\in(\mathbb{C}^{*})^{4}\subset\mathbb{P}_{\Delta^{*}}$
and $v=(v_{1},...,v_{4})\in(\mathbb{C}^{*})^{4}\subset\mathbb{P}_{\Delta}$.
The Zariski closures 
\[
Z_{f_{\Delta}(b)}=\overline{\left\{ f_{\Delta}(b)=0\right\} }\subset\mathbb{P}_{\Delta},\,\,\,Z_{f_{\Delta^{*}}(a)}=\overline{\left\{ f_{\Delta^{*}}(a)=0\right\} }\subset\mathbb{P}_{\Delta^{*}}
\]
describe the zero loci of general sections of anti-canonical bundles
in the respective toric varieties. The proper transforms $\hat{Z}_{f_{\Delta}(b)}\subset\hat{\mathbb{P}}_{\Delta}$,
$\hat{Z}_{f_{\Delta^{*}(a)}}\subset\hat{\mathbb{P}}_{\Delta^{*}}$
define Calabi-Yau hypersurfaces in the MPCP resolutions. This construction
applies for all dimensions, i.e., $\dim\mathbb{P}_{\Delta}=\dim\mathbb{P}_{\Delta^{*}}=n$. 

Batyrev \cite{Bat} proved that for dimensions $n\leq4$, the hypersurfaces
$\hat{Z}_{f_{\Delta}(b)}$ and $\hat{Z}_{f_{\Delta^{*}(a)}}$ with
general values of parameters do not intersect with the singular loci
of toric varieties $\hat{\mathbb{P}}_{\Delta}$ and $\hat{\mathbb{P}}_{\Delta^{*}}$,
respectively, and defines smooth Calabi-Yau hypersurfaces. Furthermore,
for $n\geq4$, he proved in general that their Hodge numbers are given
by 
\begin{equation}
\begin{aligned} & h^{1,1}(\hat{Z}_{f_{\Delta^{*}}})=h^{n-2,1}(\hat{Z}_{f_{\Delta}})=l(\Delta)-n-\sum_{\theta\in\Xi_{\Delta}(1)}l'(\theta)+\sum_{\theta\in\Xi_{\Delta}(2)}l'(\theta)l'(\breve{\theta})\\
 & h^{n-2,1}(\hat{Z}_{f_{\Delta^{*}}})=h^{1,1}(\hat{Z}_{f_{\Delta}})=l(\Delta^{*})-n-\sum_{\theta\in\Xi_{\Delta^{*}}(1)}l'(\theta)+\sum_{\theta\in\Xi_{\Delta^{*}}(2)}l'(\theta)l'(\breve{\theta})
\end{aligned}
,\label{eq:B-formula}
\end{equation}
where $l(P)$ counts the number of integral points in a polytope $P$.
Also, we denote by $\Xi_{P}(k)$ the set of co-dimension $k$-faces
of a polytope $P$, and for a face $\theta\in\Xi_{P}(k)$, $\breve{\theta}$
represents its dual face in $\Xi_{P^{*}}(n-k-1)$. The functions $l'(\theta),l'(\breve{\theta})$
counts integral points in the relative interiors of the faces. 

When $n=4$, the above formula gives the relations $h^{1,1}(\hat{Z}_{f_{\Delta}})=h^{2,1}(\hat{Z}_{f_{\Delta^{*}}})$
and $h^{2,1}(\hat{Z}_{f_{\Delta^{*}}})=h^{1,1}(\hat{Z}_{f_{\Delta}})$,
which shows that the smooth hypersurfaces $\hat{Z}_{f_{\Delta}(b)}\subset\hat{\mathbb{P}}_{\Delta}$,
$\hat{Z}_{f_{\Delta^{*}(a)}}\subset\hat{\mathbb{P}}_{\Delta^{*}}$
are in topological mirror to each other. When $n=2$, we simply obtain
elliptic curves. For $n=3$, we obtain pairs of K3 surfaces whose
mirror symmetry is described in terms of the K3 lattice $L_{K3}\simeq U^{\oplus3}\oplus E_{8}(-1)^{\oplus2}$
\cite{Do}. 

~

\subsubsection{\protect\label{subsec:Moduli-f}Moduli spaces of $\hat{Z}_{f_{\Delta^{*}}}(a)$. }

We have a pair of smooth Calabi-Yau hypersurfaces $(\hat{Z}_{f_{\Delta}},\hat{Z}_{f_{\Delta^{*}}})$
associated to a pair of reflexive polytopes $(\Delta,\Delta^{*})$
of dimension four. There are 73,800,776 reflexive polytopes of dimension
four due to the classification by Kreuzer and Skarke \cite{KSlist}.
As a convention, we will call $\hat{Z}_{f_{\Delta}}$ as a Calabi-Yau
hyperfurface and $\hat{Z}_{f_{\Delta^{*}}}$ mirror a Calabi-Yau hypersurface.
We write mirror Calabi-Yau hypersurfaces by $\hat{Z}_{f_{\Delta^{*}}}(a)$
with parameters $a=(a_{\nu})\in\mathbb{C}^{|\Delta^{*}\cap N|}$ and
seek a good family where we see the $B$-structure of the family which
is isomorphic to the $A$-structure of $\hat{Z}_{f_{\Delta}}$. 

~

(3.1.2.a) For our purpose, it is helpful to read each term in the
formula 
\begin{equation}
h^{2,1}(\hat{Z}_{f_{\Delta^{*}}})=l(\Delta^{*})-4-\sum_{\theta\in\Xi_{\Delta^{*}}(1)}l'(\theta)+\sum_{\theta\in\Xi_{\Delta^{*}}(2)}l'(\theta)l'(\breve{\theta}).\label{eq:h21}
\end{equation}
The first term in the r.h.s represents the number of the parameters
$a=(a_{\nu})$ in $f_{\Delta^{*}}(a)$. The second and third terms
altogether counts the dimension of ${\rm Aut}(\mathbb{P}_{\Delta^{*}})$,
where the $4$ represents the dimension of the algebraic torus $(\mathbb{C}^{*})^{4}$
and the third term is explained by the roots $R_{\Delta^{*}}$ of
the algebraic group ${\rm Aut}(\mathbb{P}_{\Delta^{*}})$. We refer
\cite[Prop.3.15]{TOda} for the detailed descriptions of ${\rm Aut}(\mathbb{P}_{\Delta^{*}})$,
but we remark here that this group is mostly non-reductive. Having
these meaning, we read the first three terms $l(\Delta^{*})-4-\sum_{\theta\in\Xi_{\Delta^{*}}(1)}l'(\theta)$
of $h^{2,1}$ as the number of the deformations of $\hat{Z}_{f_{\Delta^{*}}}$
which comes from deforming the defining equation $f_{\Delta^{*}}(a)$.
The last term is explained by the deformations which are not represented
by the parameters in the defining equation, which is called twisted
sector \cite{Mavl}. 
\begin{rem}
If we use the relation $h^{1,1}(\hat{Z}_{f_{\Delta}})=h^{2,1}(\hat{Z}_{f_{\Delta^{*}}})$,
we can read the terms in the formula (\ref{eq:h21}) differently.
The first two terms in the r.h.s count the number of torus invariant
divisors in $\hat{\mathbb{P}}_{\Delta}$ up to the rational equivalences.
The third term subtracts the number of divisors which result from
blowing up at points in $\mathbb{P}_{\Delta}$. We subtract these
divisors since these are away from general hypersurfaces $\left\{ f_{\Delta}(a)=0\right\} $.
In contrast to these, the divisors coming from blowing-up of (toric)
curves in $\mathbb{P}_{\Delta}$ contribute to $h^{1,1}(\hat{Z}_{f_{\Delta}})$
since these curves intersect with the general hypersurfaces at multiple
points. The fourth term takes these contributions into account by
counting these intersection numbers. 
\end{rem}

~

(3.1.2.b) For our construction of the moduli space of hypersurfaces
$\left\{ \hat{Z}_{f_{\Delta^{*}}(a)}\right\} $, we denote all integral
points in $\Delta^{*}$ except those on the interior of codimension-one
faces by 
\begin{equation}
\nu_{0},\nu_{1},\cdots,\nu_{p}\,\in N,\label{eq:int-points-Ds}
\end{equation}
where we set $\nu_{0}=0$ as a convention. Considering these integral
points, we define the moduli space by the following GIT quotient:
\begin{equation}
\mathcal{M}_{f}=\left\{ [a_{0}+a_{1}u^{\nu_{1}}+\cdots+a_{p}u^{\nu_{p}}]\mid a=(a_{0},a_{1},\cdots,a_{p})\in\mathbb{C}^{p+1}\right\} //_{\chi}T,\label{eq:Mf}
\end{equation}
where $t\in T=(\mathbb{C}^{*})^{4}$ acts on the parameter $a$ through
the action $t\cdot u^{\nu}=t^{\nu}u^{\nu}$ and $\chi$ represent
a character of the linearization. The notation $[a_{0}+a_{1}u^{\nu_{1}}+\cdots+a_{p}u^{\nu_{p}}]=:[f_{\Delta^{*}}(a)]$
represents the class of non-zero multiples. We extend the lattice
to $\overline{N}=\mathbb{Z}\times N$ and define 
\begin{equation}
\bar{\nu_{k}}=1\times\nu_{k}=\left(\begin{matrix}1\\
\nu_{k}
\end{matrix}\right).\label{eq:bar-nu}
\end{equation}
Then we can encode the action $f_{\Delta^{*}}(a)\mapsto\lambda f_{\Delta^{*}}(a)$
$(\lambda\in\mathbb{C}^{*})$ into the natural action of $\overline{T}:=\mathbb{C}^{*}\times T$.
With this natural extension, the torus $\overline{T}=(\mathbb{C}^{*})^{5}\ni(t_{0},t_{1},\cdots,t_{4})$
acts on $(a_{0},\cdots,a_{p})\in\mathbb{C}^{p+1}$ by $t\cdot a=(t^{\bar{\nu}_{0}}a_{0},\cdots,t^{\bar{\nu}_{p}}a_{p})$
and we can write the GIT quotient by $\mathcal{M}_{f}=\mathbb{C}^{p+1}//_{\chi}\overline{T}$. 
\begin{rem}
\label{rem:non-red-Aut}We have omitted all points in the interior
of co-dimension one faces of $\Delta^{*}$ simply to avoid a GIT quotient
by non-reductive group $\mathrm{Aut}(\mathbb{P}_{\Delta^{*}})$. In
Subsect.\ref{subsec:Non-reductive-gr}, we will show two examples
where non-reductive automorphisms play roles. 
\end{rem}

(3.1.2.c) The GIT quotients of affine spaces by tori are well studied
subjects, and described by the so-called linear Gale transformation
(see \cite{OdaPark},\cite[Sec.14.3]{CoxBook}). To apply the general
construction to our case, we define $5\times(p+1)$ matrix 
\begin{equation}
A=\left(\begin{matrix}1 & 1 & \cdots & 1\\
\nu_{0} & \nu_{1} & \cdots & \nu_{p}
\end{matrix}\right),\label{eq:matrix-A}
\end{equation}
which defines a map $\mathbb{Z}^{p+1}\rightarrow\overline{N}=\mathbb{Z}^{5}$,
and we fit this into the following diagram:
\begin{equation}
\begin{matrix} &  & K & \overset{B}{\leftarrow} & \check{\mathbb{Z}}^{p+1} & \overset{\,^{t}A}{\leftarrow} & \overline{M} & \leftarrow & 0\\
\\0 & \rightarrow & L & \underset{\,^{t}B}{\rightarrow} & \mathbb{Z}^{p+1} & \underset{A}{\rightarrow} & \overline{N}
\end{matrix},\label{eq:Gale-daig}
\end{equation}
where $L:=\mathrm{Ker}\left\{ A:\mathbb{Z}^{p+1}\rightarrow\overline{N}\right\} $,
and we define $\overline{M}=\mathbb{Z}\times M$ and $K$, respectively,
to be the dual of $\overline{N}=\mathbb{Z}\times N$ and $L$. The
$(p+1)\times(p-4)$ matrix $\,^{t}B$ is determined by arranging the
kernel of $A$, and the first line represents the dual of the second
line. Each row of the matrix $A$ represents of the weight (charge)
of the parameter $a_{i}(0\leq i\leq p)$ withe respect to the each
factor of $\mathbb{C}^{*}$ in $\overline{T}=(\mathbb{C}^{*})^{5}$.
The kernel $L$ represents additively the invariant Laurent monomials
$\mathbb{C}[a_{0}^{\pm1},\cdots,a_{p}^{\pm1}]^{\overline{T}}$. Each
column vector $c$ of $A$ may be regarded as a character (one dimensional
representation) of $\overline{T}$ by defining $\chi_{c}(t_{0},t_{1},...,t_{4})=t^{c}$.
Then the image $\mathrm{Im}A\subset\overline{N}$ represents the group
of these characters additively. Specifying a character, the GIT quotient
is defined by the graded ring 
\[
R_{\chi}=\bigoplus_{d\geq0}\left\{ F(a)\in\mathbb{C}[a_{0},\cdots,a_{p}]\mid F(t\cdot a)=\chi(t)^{d}F(a)\right\} .
\]

\begin{prop}
\label{prop:GIT-by-Sec-fan} Define $\mathcal{M}_{f}=(\mathbb{C})^{p+1}//_{\chi}\overline{T}$
for a character $\chi$ in a general position in the cone $\mathrm{Cone}(A)=\mathbb{R}_{>0}\bar{\nu}_{0}+\cdots+\mathbb{R}_{>0}\bar{\nu}_{p}$,
then there is a toric (partial) resolution $\mathbb{P}_{\mathrm{Sec}\Delta^{*}}\rightarrow\mathcal{M}_{f}$
by the secondary polytope $\mathrm{Sec}\Delta^{*}$ in $L\otimes\mathbb{R}.$
When $\chi$ is in the complement $\overline{\mathrm{Cone(A)}}^{c}$,
we have $\mathcal{M}_{f}=\phi$, i.e., there are no semi-stable points. 
\end{prop}

\begin{proof}
These properties follows from the definitions of toric varieties in
terms of the quotients. For the first claim, we note the property
\cite[Prop.1.2,(3)]{SzVergne} and use the relation between basis
index sets and triangulations. Then use the description of the secondary
fan (the normal fan of the secondary polytope) given in \cite[eq.(4.2)]{BFS}.
For the second claim, see \cite[Prop.14.3.5]{CoxBook}. 
\end{proof}
The GIT quotient depends on the choice $\chi\in\mathrm{Cone(A)}$,
however, since we always have a toric (partial) resolution by $\mathrm{Sec}\Delta^{*},$
we simply write $\mathcal{M}_{f}=\mathbb{P}_{\mathrm{Sec}\Delta^{*}}$
assuming this (partial) resolution. 

(3.1.1.d) The secondary polytope $\mathrm{Sec}\Delta^{*}$ is defined
by finding all regular triangulations of $\Delta^{*}$, i.e., simplicial
decompositions of $\Delta^{*}$ using (not necessarily all) integral
points (\ref{eq:int-points-Ds}) with a certain regularity condition.
There are finitely many regular triangulations. We write each triangulation
by a collection of simplices, $T=\left\{ \sigma\right\} $, and define
an integral vector $v_{T}=(v_{T}^{0},v_{T}^{1},\cdots,v_{T}^{p})\in\mathbb{Z}^{p+1}$
by 
\[
v_{T}^{i}=\sum_{\sigma\in T,\nu_{i}\prec\sigma}{\rm vol}(\sigma),
\]
where $\mathrm{vol}$ is normalized so that ${\rm vol(\sigma)=1}$
for the standard simplex and the summation is taken over $\sigma$
which has $\nu_{i}$ as a vertex. It was shown in \cite{GKZ2} and
\cite{BFS} that, when we fix a (regular) triangulation $T_{o}$,
\[
\mathrm{Sec}\Delta^{*}=\mathrm{Conv.}\left\{ v_{T}-v_{T_{o}}\mid T:\text{regular triangulations of }\Delta^{*}\right\} 
\]
defines a lattice polytope in $L$. Calculating all regular triangulations
for a given lattice polytope $\Delta^{*}$ becomes a hard problem
if the number integral points in $\Delta^{*}$ becomes large. We can
find efficient computer codes for the computations in \cite{TOPCOM},\cite{CYtools}.

\subsection{Period integrals and Picard-Fuchs differential equations}

We have smooth Calabi-Yau hypersurfaces $\hat{Z}_{f_{\Delta^{*}}(a)}\subset\hat{\mathbb{P}}_{\Delta^{*}}$
for general $a\in(\mathbb{C}^{*})^{p+1}$. By changing of the coordinate
of $\hat{\mathbb{P}}_{\Delta^{*}}$, under the action $t\in\overline{T}$,
we naturally have the isomorphism $\hat{Z}_{f_{\Delta^{*}}(a)}\simeq\hat{Z}_{f_{\Delta^{*}}(t\cdot a)}$.
The GIT quotient identifies these isomorphisms, and also takes into
account ``non-generic'' parameters, for which hypersurfaces are
singular, to have a compactified moduli space $\mathcal{M}_{f}=\mathbb{P}_{\mathrm{Sec}\Delta^{*}}$.
The loci in $\mathcal{M}_{f}$ which correspond to singular hypersurfaces
are called discriminant loci. We set $\mathcal{M}_{f}^{0}=\mathcal{M}_{f}\setminus\left\{ \text{discriminant loci}\right\} $.
Applying this to the general arguments in Subsect.\ref{subsec:B-str.},
we have a (mirror) family of Calabi-Yau manifolds $\check{\mathfrak{X}}\rightarrow\mathcal{M}_{f}^{0}$
and the local system $R^{3}\pi_{*}\mathbb{C}_{\hat{\mathfrak{X}}}$
over $\mathcal{M}_{f}^{0}$. This local system is described by differential
equations satisfied by period integrals of a holomorphic three form,
which can be expressed by 
\begin{equation}
\Pi(a)=\int_{T(\gamma)}\frac{1}{a_{0}+a_{1}u^{\nu_{1}}+\cdots+a_{p}u^{\nu_{p}}}\prod_{i=1}^{4}\frac{du_{i}}{u_{i}},\label{eq:Period-Pi}
\end{equation}
where $T(\gamma)$ represents a ``tubular'' cycle of $\gamma\in H_{3}(\hat{Z}_{f_{\Delta^{*}}}(a),\mathbb{Z})$
\cite{Griff}. The following result is due to Batyrev and Cox \cite{BC}.
\begin{prop}
\label{prop:GKZ-Pi}The period integral (\ref{eq:Period-Pi}) satisfies
a GKZ system with the data $\mathcal{A}=\left\{ \bar{\nu}_{0},\bar{\nu}_{1},\cdots,\bar{\nu}_{p}\right\} $
and exponent $\beta=-1\times0\in\overline{N}$. Namely it satisfies
\begin{equation}
\left\{ \prod_{l_{i}>0}\left(\frac{\partial\;}{\partial a_{i}}\right)^{l_{i}}-\prod_{l_{i}<0}\left(\frac{\partial\;}{\partial a_{i}}\right)^{-l_{i}}\right\} \Pi(a)=0\,\,(\ell\in L),\,\,\mathcal{Z}\Pi(a)=0,\label{eq:GKZsys}
\end{equation}
where $L$ is the lattice in the diagram (\ref{eq:Gale-daig}) and
$\mathcal{Z}=\sum_{i}\bar{\nu}_{i}a_{i}\frac{\partial\;}{\partial a_{i}}-\beta$
is a vector valued differential operator defined by $\bar{\nu}_{i}$
in (\ref{eq:bar-nu}) and $\beta$. 
\end{prop}

By construction, the family of Calabi-Yau manifolds $\check{\mathfrak{X}}\rightarrow\mathcal{M}_{f}^{0}$
extends over $\mathcal{M}_{f}=\mathbb{P}_{\mathrm{Sec}\Delta^{*}}$.
Correspondingly, the differential equations satisfied by period integrals
extends also to $\mathcal{M}_{f}$ with regular singularities over
$\mathcal{M}_{f}\setminus\mathcal{M}_{f}^{0}$. We can see this property
in more general theorem about $\mathcal{A}$-hypergeometric systems
\cite{GKZ} defined for more general configurations of integral points
$\mathcal{A}=\left\{ \bar{v}_{0},\bar{v}_{1},\cdots,\bar{v}_{p}\right\} $; 
\begin{thm}
$\mathcal{A}$-hypergeometric systems (GKZ systems) are holonomic
systems defined over the toric varieties of the secondary polytopes.
For general exponent $\beta$ (i.e., non-resonant $\beta$), there
are $vol(Conv(\mathcal{A}))$ independent power series solutions at
each toric boundary points. 
\end{thm}

It is easy to see that the formal form 
\begin{equation}
\Pi(a)=\sum_{l\in L}\frac{1}{\prod_{i=0}^{p}\Gamma(l_{i}+1)}a^{l+c}\label{eq:formal-Pi}
\end{equation}
with $Ac=\beta$ satisfies the equations (\ref{eq:GKZsys}). Let us
denote the normal fan of the secondary polytope $\mathrm{Sec}\Delta^{*}$
by $\mathcal{N}_{\mathrm{Sec}\Delta^{*}}=\left\{ \sigma_{T}\mid T:\text{regular triangulation}\right\} .$
Then the affine charts of $\mathbb{P}_{\mathrm{Sec}\Delta^{*}}$ are
described by $U_{T}=\mathrm{Spec}\mathbb{C}[\sigma_{T}^{\vee}\cap L]$,
where we recall that the vertices of $\mathrm{Sec}\Delta^{*}$ are
labeled by regular triangulations of $\Delta^{*}$. The (dual of)
normal cone $\sigma_{T}^{\vee}$ introduces a positive direction by
$\sigma_{T}^{\vee}\cap L$ in the lattice $L$. By subdividing the
cone $\sigma_{T}\subset K\otimes\mathbb{R}$ if necessary, we obtain
a power series (of the form $a^{\gamma}(1+\cdots)$) from the formal
solution (\ref{eq:formal-Pi}). This makes the local solutions of
(\ref{eq:GKZsys}) around the origin of $U_{T}$, i.e., the toric
boundary point represented by a vertex of $\mathrm{Sec}\Delta^{*}$.
See \cite{GKZ} for more details of writing local solutions. If $\beta$
is non-resonant, we have $\mathrm{vol}(\Delta^{*})$ independent power
series solutions for every vertex of $\mathrm{Sec}\Delta^{*}$. In
our case, $\beta=-1\times0$ turns out to be ``maximally'' resonant,
i.e., there is only one power series solution at the origin of $U_{T_{o}}$
for a maximal regular triangulation $T_{o}$ (see \cite{HLY}). Here
a maximal triangulation means a simplicial decomposition of $\Delta^{*}$
where all $\nu_{0},\nu_{1},\cdots,\nu_{p}$ are used as vertices of
the triangulation and also all simplices contain the origin $\nu_{0}$
as a vertex. 

Suppose $T_{o}$ is a maximal triangulation and also the normal cone
$\sigma_{T_{o}}$ is a regular simplicial cone (if not we subdivide
$\sigma_{T_{o}}$ to have a regular simplicial cone), and represents
the dual cone as 
\begin{equation}
\sigma_{T_{o}}^{\vee}=\mathbb{\mathbb{R}}_{>0}l^{(1)}+\cdots+\mathbb{R}_{>0}l^{(p-n)}\label{eq:To-L-basis}
\end{equation}
with the integral generators $l^{(1)},\cdots,l^{(p-n)}$ of the semigroup
$\sigma_{T_{o}}^{\vee}\cap L$. Then the unique power series can be
expressed by $a_{0}\Pi(a)=:w_{0}^{T_{0}}(x)$ with 
\begin{equation}
w_{0}^{T_{o}}(x)=\sum_{n_{1},...,n_{p-n}\geq0}\frac{\Gamma(-n\cdot l_{0}+1)}{\prod_{i\geq1}\Gamma(n\cdot l+1)}x_{1}^{n_{1}}\cdots x_{p-n}^{n_{p-n}}\,,\label{eq:w0-To}
\end{equation}
where set $n\cdot l:=\sum_{k}n_{k}l^{(k)}$ and $x_{k}:=(-1)^{l_{0}^{(k)}}a^{l^{(k)}}$.
It was found in \cite{HKTY1} that other solutions at the origin of
$U_{T_{o}}$ contain logarithmic singularities. 
\begin{rem}
As observed in \cite{HKTY1,HLYcmp}, the GKZ system (\ref{eq:GKZsys})
is reducible in general, and the period integrals of the family $\check{\mathfrak{X}}\rightarrow\mathcal{M}_{f}$
appear as an irreducible component of the system. For many cases,
the differential equations which characterize the irreducible components
are determined by finding suitable factorizations of differential
operators in the GKZ system or its extension (the extended GKZ system
in \cite{HLYcmp,HKTY1}) with additional first order differential
operators representing infinitesimal actions of $\mathrm{Aut}(\mathbb{P}_{\Delta^{*}})$.
We can characterize the irreducible part implicitly as the set of
differential operators which annihilate the power series $w_{0}^{T_{o}}(x)$
(\ref{eq:w0-To}). However, more direct description of the irreducible
system is still missing. 
\end{rem}

\subsection{\protect\label{subsec:CICY}Calabi-Yau complete intersections in
toric varieties}

The construction of Calabi-Yau hypersurfaces in toric varieties was
extended to complete intersections by Batyrev and Borisov \cite{BB}.
We can describe the moduli spaces of Calabi-Yau complete intersections
by the secondary fans. 

\subsubsection{Calabi-Yau complete intersections via nef partitions }

Let $\Delta\subset M\otimes\mathbb{R}\simeq\mathbb{R}^{n}$ be \label{subsec:CICY-BB-const}
reflexive polytope of dimension $n$. Using the fact that the integral
points of this polytope describe global sections of $-K_{\mathbb{P}_{\Delta}}$,
we have constructed Calabi-Yau hypersurfaces by the defining equation
$f_{\Delta}=0$. To have complete intersections of the form $\left\{ f_{1}=\cdots=f_{r}=0\right\} $
in $\hat{\mathbb{P}}_{\Delta}$, we consider the case where we have
the following splitting 
\begin{equation}
\Delta=\Delta_{1}+\Delta_{2}+\cdots+\Delta_{r}\label{eq:nef-part-D}
\end{equation}
by the Minkowski sum. Since this corresponds to $f_{\Delta}=f_{\Delta_{1}}f_{\Delta_{2}}\cdots f_{\Delta_{r}}$,
i.e., splitting the anti-canonical bundle $-K_{\mathbb{P}_{\Delta}}$
into a product of nef line bundles $\otimes_{i}\mathcal{L}_{i}$,
by a slight abuse of words, the decomposition (\ref{eq:nef-part-D})
is often called nef partition of $\Delta$. Note that the convex hull
$\mathrm{Conv.}(\Delta_{1},\Delta_{2},\cdots,\Delta_{r})$ is contained
in $\Delta$. Batyrev and Borisov \cite{BB} noted that $\mathrm{Conv.}(\Delta_{1},\Delta_{2},\cdots,\Delta_{r})$
is a reflexive polytope and defined its dual by $\nabla:=(\mathrm{Conv.(\Delta_{1},\Delta_{2},\cdots,\Delta_{r}))^{*}}$.

The line bundle $-K_{\mathbb{P}_{\Delta}}$ is described by the corresponding
piecewise linear function $\varphi$ on the normal fan $\mathcal{N}(\Delta)$,
which takes values $1$ at each primitive generator of the normal
fan. Using $\varphi$, the polytope $\Delta$ representing the sections
of $-K_{\mathbb{P}_{\Delta}}$ is written by $\Delta=\left\{ x\in M\otimes\mathbb{R}\mid\langle x,y\rangle\geq-\varphi(y)\,(y\in N\otimes\mathbb{R})\right\} $,
and the nef partition (\ref{eq:nef-part-D}) determines the corresponding
splitting $\varphi=\varphi_{1}+\cdots+\varphi_{r}$. This splitting
of $\varphi$ determines the dual nef partition of (\ref{eq:nef-part-D}),
\[
\nabla=\nabla_{1}+\nabla_{2}+\cdots+\nabla_{r}
\]
with $\nabla_{i}:=\mathrm{Conv.}\left(\left\{ 0\right\} \cup\left\{ v\in\Delta^{*}\cap N\mid\varphi_{i}(v)=1\right\} \right)$.
Then it holds naturally that $\mathrm{Conv.}(\nabla_{1},\nabla_{2},\cdots,\nabla_{r})=\Delta^{*}$,
hence $\Delta=\left(\mathrm{Conv.}(\nabla_{1},\nabla_{2},\cdots,\nabla_{r})\right)^{*}$. 

Using the duality of nef partitions, Batyrev-Borisov \cite{BB} defines
the following complete intersections in MPCP resolutions $\hat{\mathbb{P}}_{\Delta}$
and $\hat{\mathbb{P}}_{\nabla}$, respectively:
\[
\begin{aligned}\hat{Z}_{f_{\Delta_{1}},..,f_{\Delta_{r}}}=\left\{ f_{\Delta_{1}}(b)=\cdots=f_{\Delta_{r}}(b)=0\right\} \subset\hat{\mathbb{P}}_{\Delta}\,,\\
\hat{Z}_{f_{\nabla_{1}},..,f_{\nabla_{r}}}=\left\{ f_{\nabla_{1}}(a)=\cdots=f_{\nabla_{r}}(a)=0\right\} \subset\hat{\mathbb{P}}_{\nabla}\,,
\end{aligned}
\]
where $f_{\Delta_{i}}(b)$ and $f_{\nabla_{j}}(a)$ are Laurent polynomials
with corresponding parameters. 
\begin{thm}
When $n-r\leq3$ and the parameters are general, the above complete
intersections define Calabi-Yau manifolds of dimension $n-r$. In
particular, when $n-r=3$, they are topologically mirror to each other,
i.e., satisfy 
\begin{equation}
h^{1,1}(\hat{Z}_{f_{\Delta_{1}},..,f_{\Delta_{r}}})=h^{2,1}(\hat{Z}_{f_{\nabla_{1}},..,f_{\nabla_{r}}})\,,\,\,h^{2,1}(\hat{Z}_{f_{\Delta_{1}},..,f_{\Delta_{r}}})=h^{1,1}(\hat{Z}_{f_{\nabla_{1}},..,f_{\nabla_{r}}}).\label{eq:Hodge-numb-CICY}
\end{equation}
\end{thm}

We refer \cite[Prop.4.15,Thm.4.15]{BB} for a proof. A closed formula
for the Hodge numbers (\ref{eq:Hodge-numb-CICY}) are determined in
\cite[Thm.4,14,4.15]{BB2}. Also, a package of computer codes calculating
them can be found in \cite{PALP}. 

\subsubsection{\protect\label{subsec:Moduli-CICY}Moduli spaces of complete intersections}

Calabi-Yau complete intersection $\hat{Z}_{f_{\nabla_{1}},..,f_{\nabla_{r}}}$
is defined by the Laurent polynomials
\[
f_{\nabla_{1}}(a^{(1)})=\sum_{\nu\in\nabla_{1}\cap N}a_{\nu}^{(1)}u^{\nu}\,,\;\cdots\;,\,f_{\nabla_{r}}(a^{(r)})=\sum_{\nu\in\nabla_{r}\cap N}a_{\nu}^{(r)}u^{\nu}
\]
with general parameters $a^{(1)}\in\mathbb{C}^{p_{1}+1},\cdots,a^{(r)}\in\mathbb{C}^{p_{r}+1}$
$(p_{k}+1:=\#\nabla_{k}\cap N)$. We write the entire parameters by
$a=(a^{(1)},\cdots,a^{(r)})\in\mathbb{C}^{p+r}$ $(p=\sum_{k}p_{k})$.
We assume that $n-r=3$ so that $\hat{Z}_{f_{\nabla_{1}},..,f_{\nabla_{r}}}$
is a smooth Calabi-Yau manifold.

Corresponding to (\ref{eq:Mf}), we can define the moduli space of
$\hat{Z}_{f_{\nabla_{1}},..,f_{\nabla_{r}}}$ by the following GIT
quotient: 
\[
\mathcal{M}_{f_{1}\cdots f_{r}}:=\left\{ ([f_{\nabla_{1}}(a^{(1)})],\cdots,[f_{\nabla_{r}}(a^{(r)})])\mid a\in\mathbb{C}^{p+r}\right\} //_{\chi}T,
\]
where $T=(\mathbb{C}^{*})^{n}\subset\mathrm{Aut(\mathbb{P}_{\nabla})}$
acts on $a$ by $t\cdot a_{\nu}=t^{\nu}a_{\nu}$ as before. To write
the quotient in a more familiar form, we introduce auxiliary variables
$\lambda_{2},\cdots,\lambda_{r}$ and define 
\[
F(\lambda,a)=f_{\nabla_{1}}(a^{(1)})+\lambda_{2}f_{\nabla_{2}}(a^{(2)})+\cdots+\lambda_{r}f_{\nabla_{r}}(a^{(r)}).
\]
Then we can write the quotient in the same form as (\ref{eq:Mf}):
\[
\mathcal{M}_{f_{1}\cdots f_{r}}=\left\{ [F(\lambda,a)]\mid(\lambda,a)\in\mathbb{C}^{r-1}\times\mathbb{C}^{p+r}\right\} //_{\chi}(\mathbb{C}^{*})^{r-1}\times T,
\]
where $(\mathbb{C}^{*})^{r-1}\times T\ni(\Lambda,t)$ acts on $(\lambda,a)$
by $\lambda_{i}\mapsto\Lambda_{i}\lambda_{i}$ and $a_{\nu}^{(k)}\mapsto t^{\nu}a_{\nu}^{(k)}$. 

Finally we can express the quotient by 
\[
\mathcal{M}_{f_{1}\cdots f_{r}}=\mathbb{C}^{p+2r-1}//_{\chi}\overline{T}
\]
with $\overline{T}:=(\mathbb{C}^{*})^{r}\times T$. In this from,
the GIT quotient is described by the diagram (\ref{eq:Gale-daig})
with $\overline{N}=\mathbb{Z}^{r}\times N$ and $(r+n)\times(p+2r-1)$
integral matrix $A=\Big(\bar{\nu}_{0}^{(1)}\cdots\bar{\nu}_{0}^{(r)};$$\bar{\nu}_{1}^{(1)}\cdots\bar{\nu}_{p_{1}}^{(1)};$
$\cdots;\bar{\nu}_{1}^{(r)}\cdots\bar{\nu}_{p_{r}}^{(r)}\Big)$. Here
corresponding to (\ref{eq:bar-nu}), we define 
\begin{equation}
\bar{\nu}_{i}^{(k)}:=1\times\tilde{\nu}_{i}^{(k)}=\left(\begin{matrix}1\\
\,\,\,\,\,\,\,\,\,\,\,\overset{\cdot}{:}\,\,\,\,\,\left\} 0\right.\\
\,\,\,\,1\,\,\,\,\,\\
\,\,\,\,\,\,\,\,\,\,\,\overset{\cdot}{:}\,\,\,\,\,\left\} 0\right.\\
\nu_{i}^{(k)}
\end{matrix}\right)\lhook k,\label{eq:bar-vect-A}
\end{equation}
where $\tilde{\nu}_{i}^{(k)}\,(i=0,\cdots,p_{k})$ represent the exponents
of Laurent monomials contained in $\lambda_{k}f_{\nabla_{k}}(a^{(k)})$
with respect to the variables $\lambda_{2},\cdots,\lambda_{r},$ $u_{1},\cdots,u_{d}$
with a convention $\nu_{0}^{(k)}\equiv0$ for all $k$. In this from,
we can apply Proposition \ref{prop:GIT-by-Sec-fan} with the secondary
polytope of $\mathrm{Conv.\left(\left\{ \tilde{\nu}_{i}^{(k)}\right\} \right)}$
or the (ordered) set 
\begin{equation}
\mathcal{A}=\left\{ \bar{\nu}_{0}^{(1)},\cdots,\bar{\nu}_{0}^{(r)};\bar{\nu}_{1}^{(1)},\cdots\bar{\nu}_{p_{1}}^{(1)};\cdots;\bar{\nu}_{1}^{(r)}\cdots\bar{\nu}_{p_{r}}^{(r)}\right\} \label{eq:A-CICY}
\end{equation}
of integral vectors in $\overline{N}\otimes\mathbb{R}$. We sometimes
write the secondary polytope by $\mathrm{Sec}\mathcal{A}$ but there
should be no confusion from here. 

\subsubsection{Period integrals and Picard-Fuchs differential equations }

Now we have a family $\check{\mathfrak{X}}$ of complete intersection
Calabi-Yau manifold over $\mathcal{M}_{f_{1}\cdots f_{r}}^{0}=\mathcal{M}_{f_{1}\cdots f_{r}}\setminus\left\{ \text{\text{discriminant loci}}\right\} $
which extends to $\mathcal{M}_{f_{1}\cdots f_{r}}$. The local system
associated to this family describes Picard-Fuchs differential equations.
Since all constructions from here are parallel to the previous case
of hypersurfaces, we summarize only some relevant formulas and results
to later discussions. 

Corresponding to (\ref{eq:Period-Pi}), we first have 
\begin{equation}
\Pi(a)=\int_{T(\gamma)}\frac{1}{f_{\nabla_{1}}(a^{(1)})\cdots f_{\nabla_{r}}(a^{(r)})}\prod_{i=1}^{n}\frac{du_{i}}{u_{i}}\label{eq:Pi-CICY}
\end{equation}
for period integrals of the family. Proposition \ref{prop:GKZ-Pi}
is now stated as 
\begin{prop}
\label{prop:GKZ-Pi-CICY}Period integrals of the family of complete
intersection Calabi-Yau manifolds satisfy the GKZ system with the
data $\mathcal{A}$ in (\ref{eq:A-CICY}) and the exponent $\beta=-\bm{1}^{r}\times0\in\overline{N}\otimes\mathbb{R}$,
where $\bm{1}^{r}:=(1,\cdots,1)\in\mathbb{Z}^{r}\otimes\mathbb{R}$. 
\end{prop}

It is straightforward to verify the above property. Now, as before,
the GKZ system defines a holonomic system over the parameter space
$\mathrm{Sec}\mathcal{A}$, but this turns out to be reducible. Its
irreducible part determines the Picard-Fuchs differential equations
which characterizes the period integrals of the family. We say a triangulation
maximal triangulation if all simplices contain the zero-th vectors
$\left\{ \bar{\nu}_{0}^{(1)},\cdots,\bar{\nu}_{0}^{(r)}\right\} $
(cf. Subsect.\ref{subsec:Birat-P4P4}). When the affine chart $U_{T_{o}}=\mathrm{Spec}\mathbb{C}[\check{\sigma}_{T_{o}}\cap L]$
for a maximal triangulation $T_{o}$ is smooth, we can take the integral
basis $l^{(1)},\cdots,l^{(s)}$ of the semi-group $\sigma_{T_{0}}^{\vee}\cap L$
$(s=\mathrm{rk}L=p+r-n-1)$. Using this basis, we have a power series
solution $(a_{0}^{(1)}\cdots a_{0}^{(r)}\Pi(a)=:w_{o}^{T_{o}}(x))$
of the Picard-Fuchs differential equations by 
\begin{equation}
w_{0}^{T_{0}}(x)=\sum_{n_{1},\cdots,n_{s}\geq0}\frac{\prod_{i=0}^{r-1}\Gamma(-n\cdot l_{i}+1)}{\prod_{i=r}^{p+2r-2}\Gamma(n\cdot l_{i}+1)}x_{1}^{n_{1}}\cdots x_{s}^{n_{s}}\label{eq:w0-CICY}
\end{equation}
with $x_{k}:=(-1)^{\sum_{i=0}^{r-1}l_{i}^{(k)}}a^{l^{(k)}}$ and $n\cdot l:=\sum_{k}n_{k}l^{(k)}$.
Note that we index the components of $l=(l_{i})\in L$ by $l=(l_{0},l_{1},\cdots,l_{p+2r-2})$
in accord to the order of $\mathcal{A}$ in (\ref{eq:A-CICY}). 
\begin{rem}
GKZ systems of Calabi-Yau complete intersections were first studied
in \cite{HKTY2}. There we can find many examples where the Picard-Fuchs
differential operators arise from the GKZ systems thorough nontrivial
factorizations of differential operators. 
\end{rem}

\subsection{\protect\label{subsec:Non-reductive-gr}Non-reductive group actions }

In order to clarify our definition of the moduli spaces by (\ref{eq:int-points-Ds})
and (\ref{eq:Mf}), we present two (lower dimensional) examples of
hypersurfaces which admit non-reductive group actions (cf. Remark
\ref{rem:non-red-Aut}). 

\begin{figure}
\includegraphics[scale=0.4]{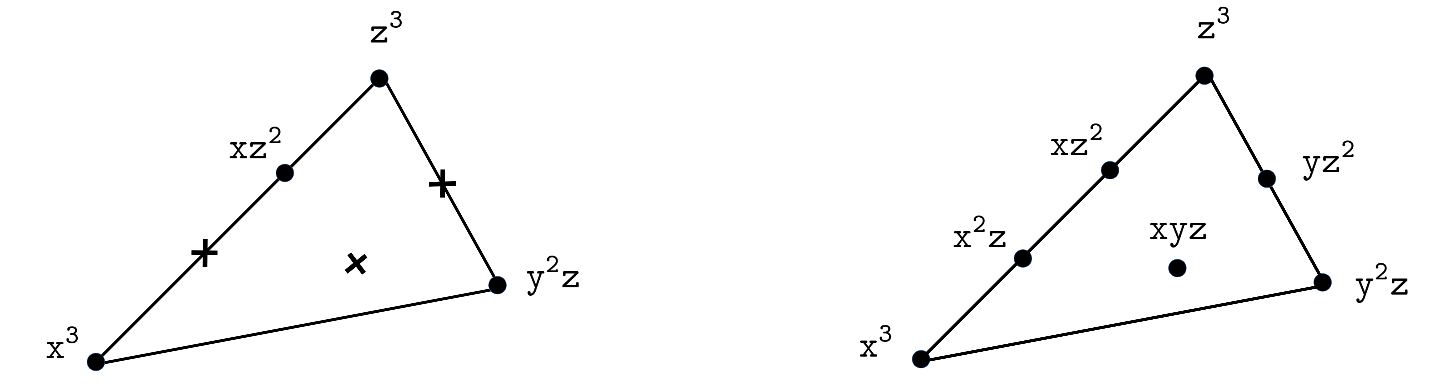}\caption{Fig.1. Reflexive polytope of the Weierstrass normal form.}
\end{figure}

\subsubsection{Weierstrass normal form}

It is standard to express an elliptic curve by Weierstrass normal
form, which we write as a cubic equation in $\mathbb{P}^{2}$
\begin{equation}
\ty^{2}\tz-(4\tx^{3}-g_{2}\tx\tz^{2}-g_{3}\tz^{3})=0.\label{eq:W-normal}
\end{equation}
We arrange each monomials as in Fig.1 (left), where we write monomials
in a lattice polytope in the (affine) lattice $N:=\left\{ (a,b,c)\in\mathbb{Z}^{3}\mid a+b+c=3\right\} $.
We regard the Weierstrass normal form is a special form of the polynomial
\[
F=a_{0}\tx\ty\tz+a_{1}\ty^{2}\tz+a_{2}\tx^{3}+a_{3}\tz^{3}+a_{4}\tx\tz^{2}+a_{5}\tx^{2}\tz+a_{6}\ty\tz^{2}
\]
with $(a_{0},a_{1},\cdots,a_{6})=(0,1,-4,g_{3},g_{2},0,0)$. By noting
the points on codimension-one faces, we find that the form of $F$
is preserved under the actions $^{t}(\tx,\ty,\tz)\mapsto g{}^{t}(\tx,\ty,\tz)$
$(g\in G)$ with 
\[
G=\left\{ \left(\begin{matrix}\lambda_{1} & 0 & c_{3}\\
c_{1} & \lambda_{2} & c_{2}\\
0 & 0 & \lambda_{3}
\end{matrix}\right)\mid\lambda_{1},\lambda_{2},\lambda_{3}\in\mathbb{C}^{*},c_{1},c_{2},c_{3}\in\mathbb{C}\right\} ,
\]
which is non-reductive. We can set the parameters $a_{4}=a_{5}=a_{6}=0$
by using this $G$-action, and assume that $F$ has the reduced form
$F_{0}=a_{0}\tx\ty\tz+a_{1}\ty^{2}\tz+a_{2}\tx^{3}+a_{3}\tz^{3}$.
In fact, in this reduced form, we defined the moduli space $\mathcal{M}_{f}$
for hypersurfaces of toric varieties. In the present case, we have
$\mathcal{M}_{F_{0}}=\mathbb{P}^{1}$. In terms of $F_{0}$, the general
form of period integral is written by 
\begin{equation}
\Pi(a)=\int_{T(\gamma)}\frac{1}{a_{0}\tx\ty\tz+a_{1}\ty^{2}\tz+a_{2}\tx^{3}+a_{3}\tz^{3}}d\mu\label{eq:Period-E-cv}
\end{equation}
with $d\mu=i_{E}d\tx d\ty d\tz$ and $E=\tx\frac{\partial\;}{\partial\tx}+\ty\frac{\partial\;}{\partial\ty}+\tz\frac{\partial\;}{\partial\tz}$.
Proposition \ref{prop:GKZ-Pi} applies to $\Pi(a)$ and, from the
general formula (\ref{eq:w0-To}), we obtain 
\[
w_{0}^{T_{0}}(x)=\sum_{n\geq}\frac{(6n)!}{(3n)!(2n)!n!}x^{n}
\]
with $x=\frac{a_{1}^{3}a_{2}^{2}a_{3}}{a_{0}^{6}}$. Using the $G$-action,
we can translate these results to the standard Weierstrass normal
form and vice versa. 
\begin{prop}
By the action of $g\in G$,$^{t}(\tx,\ty,\tz)\mapsto g{}^{t}(\tx,\ty,\tz)$$=^{t}(x-\frac{\sqrt{g_{2}}}{2\sqrt{3}}\tz,\ty-i\sqrt{2}(3g_{2})^{\frac{1}{4}}\tx,\tz)$,
The Weierstrass normal form is transformed to $F_{0}$ with 
\[
(a_{0},a_{1},a_{2},a_{3})=(-2i\sqrt{2}(3\,g_{2})^{\frac{1}{4}},1,-4,-\frac{g_{2}^{\frac{3}{2}}}{3\sqrt{3}}+g_{3}).
\]
Correspondingly, we have $x=\frac{\sqrt{3}g_{2}^{\frac{3}{2}}-9g_{3}}{864\sqrt{3}g_{2}^{\frac{3}{2}}}$
for the affine coordinate of $\mathcal{M}_{F_{0}}=\mathbb{P}^{1}$.
\end{prop}

\begin{rem}
\label{rem:W-normal-f}The Weierstrass normal form defines a (canonical)
family of elliptic curve over a weighted projective space $\mathbb{WP}^{1}(2,3)$,
with the fiber over $[g_{2},g_{3}]\in\mathbb{WP}^{1}(2,3)$ being
an elliptic curve of its $j$-invariant $j=\frac{g_{2}^{3}-27g_{3}^{2}}{1728g_{2}^{3}}$.
The family $\left\{ F_{0}(a)\right\} $ over $\mathcal{M}_{F_{0}}$
parametrizes elliptic curves which have the defining equation of the
form $F_{0}$, hence, two are distinguished when the form $F_{0}$
(up to the torus action) are different even if two have the same $j$-invarint.
We can observe this difference of the family in the monodromy of the
Picard-Fuchs equation. The fractional powers of $g_{2},g_{3}$ in
the above proposition originate from this fact. 
\end{rem}

\subsubsection{Clinger-Doran normal form of a lattice polarized K3 surfaces}

In \cite{CD}, Clinger and Doran generalized the Weierstrass normal
form to a certain lattice polarized K3 surfaces, where the cubic equation
(\ref{eq:W-normal}) is generalized to a quartic equation in $\mathbb{P}^{3}$,
\begin{equation}
f=\ty^{2}\tz\tw-4\tx^{3}\tz+3\alpha\tx\tz\tw^{2}+\beta\tz\tw^{3}+\gamma\tx\tz^{2}\tw-\frac{1}{2}(\delta\tz^{2}\tw^{2}+\tw^{4})\label{eq:CD-normal-f}
\end{equation}
with $[\alpha,\beta,\gamma,\delta]\in\mathbb{WP}^{3}(2,3,5,6)\,(\gamma\not=0\text{ or }\delta\not=0)$.
Corresponding to the expression $g_{2}=E_{4}(\tau),g_{3}=E_{6}(\tau)$
in terms of $\tau\in\mathbb{H}_{+}$, it has been shown that we have
\[
\alpha=\mathcal{E}_{4}(\tau),\,\,\,\,\beta=\mathcal{E}_{6}(\tau),\,\,\,\,\gamma=2^{12}3^{5}\chi_{10}(\tau),\,\,\,\,\delta=2^{12}3^{6}\chi_{12}(\tau)
\]
with the genus two Eisenstein series $\mathcal{E}_{4}(\tau),\mathcal{E}_{6}(\tau)$
and the Igusa cusp forms $\chi_{10}(\tau),\chi_{12}(\tau$) with specified
weights on the Siegel upper half space $\mathbb{H}_{2}.$ The constructions
in the preceding subsection apply to this case words to words: First
we observe that the normal form (\ref{eq:CD-normal-f}) is a special
form of 
\[
\begin{aligned}\begin{alignedat}{1}F(a)= & a_{0}\,\tx\ty\tz\tw+a_{1}\,\ty^{2}\tz\tw+a_{2}\,\tx^{3}\tz+a_{3}\,\tx\tz^{2}\tw+a_{4}\,\tz^{2}\tw^{2}+a_{5}\,\tw^{4}+a_{6}\,\tz\tw^{3}\\
 & \;\;\;\;+a_{7}\,\ty\tz\tw^{2}+a_{8}\,\tx\tz\tw^{2}+a_{9}\,\tx^{2}\tz\tw
\end{alignedat}
\end{aligned}
,
\]
whose Newton polytope $\Delta_{F}$ in $\left\{ (a,b,c,d)\in\mathbb{Z}^{4}\mid a+b+c+d=4\right\} $
is reflexive. The monomials with parameters $a_{7},a_{8},a_{9}$ are
on the codimension-one faces of $\Delta_{F}$, and corresponding to
this fact, we find that $F(a)$ preserves its form under the action
$g\in G$,
\[
G=\left\{ \left(\begin{smallmatrix}\lambda_{1} & 0 & 0 & c_{1}\\
c_{3} & \lambda_{2} & 0 & c_{2}\\
0 & 0 & \lambda_{3} & 0\\
0 & 0 & 0 & \lambda_{4}
\end{smallmatrix}\right)\Big|\lambda_{1},...,\lambda_{4}\in\mathbb{C}^{*},c_{1},c_{2},c_{3}\in\mathbb{C}\right\} .
\]
Using this action we can set $a_{7}=a_{8}=a_{9}=0$ and define $F_{0}(a)=F(a)\vert_{a_{7},a_{8},a_{9}=0}$. 
\begin{prop}
The normal form $f=f(\alpha,\beta,\gamma,\delta)$ (\ref{eq:CD-normal-f})
is transformed into the form $F_{0}(a)$ by
\[
\left(\begin{smallmatrix}\tx\\
\ty\\
\tz\\
\tw
\end{smallmatrix}\right)\mapsto g\left(\begin{smallmatrix}\tx\\
\ty\\
\tz\\
\tw
\end{smallmatrix}\right)=\left(\begin{smallmatrix}\tx-\frac{\sqrt{\alpha}}{2}\tw\\
\ty-i\sqrt{6}\alpha^{\frac{1}{4}}\tx\\
\tz\\
\tw
\end{smallmatrix}\right)\,\,\,(g\in G),
\]
with the resulting relations
\[
(a_{0},a_{1},a_{2},a_{3},\cdots,a_{6})=(-2i\sqrt{6}\alpha^{\frac{1}{4}},1,-4,\gamma,-\frac{1}{2}(\delta+\sqrt{\alpha}\gamma),-\frac{1}{2},(\beta-\alpha^{\frac{3}{2}})).
\]
\end{prop}

For the reduced form $F_{0}$, the toric construction of the moduli
space $\mathcal{M}_{F_{0}}$ applies, and the period integral of the
form (\ref{eq:Period-E-cv}) satisfies the GKZ system over $\mathcal{M}_{F_{0}}$
in Proposition \ref{prop:GKZ-Pi}. The form of the local solution
$w_{0}^{T_{o}}(x)$ is determined by calculating the secondary polytope
of $\Delta_{F_{0}}$. See a recent work \cite{HK} for more details,
there we can find a finite set of differential operators which characterizes
all period integrals. 
\begin{rem}
The parameter space $\mathbb{WP}^{3}(2,3,5,6)$ of the Clinger-Doran
family is a coarse moduli space of $U\oplus E_{8}(-1)\oplus E_{7}(-1)$-polarized
K3 surfaces. It parametrizes isomorphism classes of the lattice-polarized
K3 surfaces. On the other hand, the family $\left\{ F_{0}(a)=0\right\} $
over $\mathcal{M}_{F_{0}}$ parametrizes K3 surfaces of the form $\left\{ F_{0}(a)=0\right\} \subset\mathbb{P}^{3}$.
Because of this, monodromy groups of the two family differ from each
other (cf. Remark \ref{rem:W-normal-f}). 
\end{rem}

\subsection{\protect\label{subsec:K3lambda}Special forms of defining equations }

The both families $\check{\mathfrak{X}}\rightarrow\mathcal{M}_{f}$
and $\check{\mathfrak{X}}\rightarrow\mathcal{M}_{f_{1}\cdots f_{r}}$
of Calabi-Yau manifolds assume that the defining equations $f,f_{1},\cdots,f_{r}$
take general forms. If defining equations are required to have special
forms from some reasons, the general constructions of families in
the subsections \ref{subsec:Moduli-f}, \ref{subsec:Moduli-CICY}
do not apply. 

\subsubsection{Elliptic lambda function }

The elliptic $\lambda$-function is a well-studied modular function
with respect to the congruence subgroup $\Gamma(2)$. 

(3.5.1.a) The double cover of $\mathbb{P}^{1}$ branched at four general
points determines an elliptic curve. Based on this, from the configurations
of four points in $\mathbb{P}^{1}$, we obtain a family of elliptic
curves, which is called Legendre family. It is standard to represent
this family by the equation 
\[
\ty^{2}=\tx(\tx-\tz)(\tx-\lambda\tz)\tz\,\,\,(\lambda\not=0,1,\infty),
\]
where $[\ty,\tx,\tz]\in\mathbb{WP}^{2}(2,1,1)$ represents the homogeneous
coordinate of the weighted projective space. The four branch points
are identified with $[0,0,1],[0,1,1],$ $[0,\lambda,1]$ and $[0,1,0]$
on the line $\left\{ y=0\right\} \simeq\mathbb{P}^{1}\subset\mathbb{WP}^{2}(2,1,1)$.
The fourth point $[0,1,0]$ is at the infinity of $\mathbb{P}^{1}$. 

(3.5.1.b) Since the weighted projective space $\mathbb{WP}^{2}(2,1,1)$
is a Fano toric variety, the construction in Subsect.\ref{subsec:Mf-toric-hyper-S}
applies to the moduli space $\mathcal{M}_{f}$ for a general family
of elliptic curves. Since the Legendre family requires the spacial
form of the equation as above, the general construction of $\mathcal{M}_{f}$
does not apply. However, it is easy to recognize that the configuration
space $\mathcal{M}_{4}(\mathbb{P}^{1})$ of four points on $\mathbb{P}^{1}$
parametrizes the family. As we sketch in more general setting below,
the configuration space $\mathcal{M}_{4}(\mathbb{P}^{1})$ is a classical
moduli space studied by the GIT quotient, and we have $\mathcal{M}_{4}(\mathbb{P}^{1})\simeq\mathbb{P}^{1}$
with the affine parameter given by the anharmonic ratio of the four
points,
\[
\lambda=\frac{(z_{1}-z_{3})(z_{2}-z_{4})}{(z_{1}-z_{4})(z_{2}-z_{3})}
\]
where we set $(z_{1},z_{2},z_{3},z_{4})=(\infty,0,1,\lambda)$.

(3.5.1.c) The form (\ref{eq:Period-E-cv}) of period integrals of
the family applies to the present family with the Euler vector field
being replaced by $E=2\ty\frac{\partial\;}{\partial\ty}+\tx\frac{\partial\;}{\partial\tx}+\tz\frac{\partial\;}{\partial\tz}$
(where we take into account the weight of the coordinates). Taking
the residue along the tubular cycle $T(\gamma)$, we obtain the standard
form of the elliptic integral ($t:=\frac{\tx}{\tz}$) 
\[
\Pi(\lambda)=\int_{\gamma}\frac{\tz d\tx-\tx d\tz}{\sqrt{\tx(\tx-\tz)(\tx-\lambda\tz)\tz}}=\int_{\gamma}\frac{dt}{\sqrt{t(t-1)(t-\lambda)}}.
\]
Corresponding to (\ref{eq:w0-To}), we define $w_{0}(\lambda):=\sqrt{-1}\,\Pi(\lambda)$.
Then we find Picard-Fuchs equations and its solution around $\lambda=0$,
\begin{equation}
\left\{ \theta_{\lambda}^{2}-\lambda(\theta_{\lambda}+\frac{1}{2})^{2}\right\} w_{0}(\lambda)=0,\,\,\,\,\,w_{0}(\lambda)=\sum_{n\geq0}\frac{1}{\Gamma(\frac{1}{2})^{2}}\frac{\Gamma(n+\frac{1}{2})^{2}}{\Gamma(n+1)^{2}}\lambda^{n},\label{eq:PF-w0-lambda}
\end{equation}
where $\theta_{\lambda}=\lambda\frac{d\;}{d\lambda}$. Another solution
has a logarithmic singularity, and using it we define the period map
$\mathcal{P}:\mathbb{P}\setminus\left\{ 0,1,\infty\right\} \rightarrow\mathbb{H}_{+}$.
The classical lambda function is defined as the inverse of this, i.e.,
$\lambda(\tau):=\mathcal{P}^{-1}(\tau)$. 

\subsubsection{K3 lambda functions}

The above construction naturally extends to the double cover of $\mathbb{P}^{2}$
branched along six lines in general position, which we write by 
\[
\tw^{2}=\prod_{i=1}^{6}(a_{0i}\tx+a_{1i}\ty+a_{2i}\tz)
\]
with the homogeneous coordinate $[\tw,\tx,\ty,\tz]\in\mathbb{WP}^{3}(3,1,1,1)$
of specified weights. This generalizes the equation for the Legendre
family, but there are singularities of $A_{1}$ type at the 15 intersection
points of the 6 lines. After blowing-up these points, we obtain a
K3 surface whose Picard lattice is isomorphic to the orthogonal lattice
$T_{X}^{\perp}$, in the K3 lattice, of the transcendental lattice
$T_{X}\simeq U(2)\oplus U(2)\oplus A_{1}\oplus A_{1}$ (see \cite{MSYintJ,Yoshida}
and references in therein). Thus for general configurations of six
lines in $\mathbb{P}^{2}$, we have a family of K3 surfaces having
a specified Picard lattice $T_{X}^{\perp}$. For the construction
of the moduli space of this family, we can consider the GIT quotient.
However, since the resulting moduli space is singular, it is not trivial
to find a resolution where the local system of the family degenerate
in a special way to define K3 lambda functions which generalizes the
elliptic lambda function. In early 90s, K3 lambda functions were studied
extensively in many papers e.g. \cite{MStYi,MStYii,MSYintJ}. Since
it would become lengthy to list all of them, we refer a book by Yoshida
\cite{Yoshida} for the contributed papers on this subject. However,
in these works in the 90s, the system of differential operators and
its unique (up to constant multiple) power series solution which define
the K3 lambda functions in a parallel way to (\ref{eq:PF-w0-lambda})
were not known. This problem was settled recently in \cite{HLTYk3}
by carefully studying the moduli space of the family and reducing
the problem to the framework of Proposition \ref{prop:GKZ-Pi-CICY}
and (\ref{eq:w0-CICY}).

(3.5.2.a) The moduli space of the configurations of six lines in $\mathbb{P}^{2}$
(or six points in the dual $\check{\mathbb{P}}^{2}$) is a well-studied
subject in terms of the GIT quotient, see e.g. \cite{DoOr} and \cite{Reuv}.
Arranging the parameters of the six lines into $3\times6$ matrices,
it is defined by 
\[
\mathcal{M}_{6}=\left\{ \tA=\left(\begin{matrix}a_{01} & a_{02} &  & a_{06}\\
a_{11} & a_{12} & \cdots & a_{16}\\
a_{21} & a_{22} &  & a_{26}
\end{matrix}\right)\Bigg|a_{ij}\in\mathbb{C}\right\} //_{\mathrm{\chi}}G,
\]
where $(g,t)\in G=GL(3,\mathbb{C})\times(\mathbb{C}^{*})^{6}$ acts
on $\tA$ by the matrix multiplication $(g,t)\cdot\tA=g\tA\mathrm{diag}(t_{1},\cdots,t_{6})$.
The character $\chi$ is taken by $\chi(g,t)=(\det g)^{2}t_{1}\cdots t_{6}$
to define the degree in $\mathbb{C}[a_{ij}]$. By classical invariant
theory, the generators of the graded ring $R_{\chi}$ are obtained
by the minors of $\tA$, and they are given by five polynomials $Y_{0}(\tA),\cdots,Y_{4}(\tA)$
of degree $1\chi$ and an additional polynomial $Y_{5}(\tA)$ of degree
$2\chi$ \cite{DoOr}. These generators have a relation and we have
\begin{equation}
\mathcal{M}_{6}=\left\{ Y_{5}^{2}=F_{4}(Y_{0},\cdots,Y_{4})\right\} \subset\mathbb{WP}^{5}(1^{5},2)\label{eq:Coble-M6}
\end{equation}
with a quartic polynomial $F_{4}$ which defines the so-called Igusa
quartic in $\mathbb{P}^{4}$. Namely, the moduli space $\mathcal{M}_{6}$
is a double cover of $\mathbb{P}^{4}$ branched along the Igusa quartic.
The singularity of $\mathcal{M}_{6}$ is a well-studied subject; it
has $A_{1}$ singularities along 15 lines which intersect at 15 points;
see \cite[(1.1) Prop.]{vGeer} for their detailed configuration. Related
to $\mathcal{M}_{6}$, there is a toric variety 
\[
\mathcal{M}_{3,3}=\left\{ X_{0}X_{1}X_{2}-X_{3}X_{4}X_{5}=0\right\} \subset\mathbb{P}^{5},
\]
which turns out to be birational to $\mathcal{M}_{6}$ as follows:
Given six lines in general position, we may replace three of them
by their three intersection points, and thus we obtain three points
and three lines. The toric variety $\mathcal{M}_{3,3}$ comes from
the GIT quotient \cite{Reuv} of the latter configurations in $\mathbb{P}^{2}$.
Since, if the configuration is in general position, we can reverse
the correspondence, these two moduli spaces are naturally birational
to each other. 

(3.5.2.b) We have a smooth family of K3 surfaces over the open set
$\mathcal{M}_{6}^{0}$ of the complement of the discriminant loci.
This defines a family $\mathfrak{X}\rightarrow\mathcal{M}_{6}$ with
the K3 surfaces being degenerated over the discriminant loci. The
associated local system is often called $E(3,6)$ system \cite{MSYintJ},
which is a special case of the so-called Aomoto-Gel'fand systems associated
Grassmannians, see \cite{AK,Gelfand} for more details. In contrast
to the GKZ systems, where local solutions are described by the secondary
fans in general, it is not known to write local solutions in general
for Aomoto-Gel'fand systems. To have some connection to the GKZ system
in Proposition \ref{prop:GKZ-Pi-CICY}, we replace the GIT quotient
by 
\[
\mathcal{M}_{6}^{(1,2,3)}=\left\{ \tA=\left(\begin{matrix}1 & 0 & 0 & a_{04} & a_{05} & a_{06}\\
0 & 1 & 0 & a_{14} & a_{15} & a_{16}\\
0 & 0 & 1 & a_{24} & a_{25} & a_{26}
\end{matrix}\right)\Bigg|a_{ij}\in\mathbb{C}\right\} //_{\mathrm{\chi_{red}}}(\mathbb{C}^{*})^{5},
\]
where $(\mathbb{C}^{*})^{5}\subset G$ is the torus which preserve
the displayed form of $\tA$. More generally, we define $\mathcal{M}_{6}^{(ijk)}$
$(1\leq i<j<k\leq6)$ by moving the first three columns $(1,2,3)$
to $(i,j,k)$ and consider to cover the moduli space by the union
of these $\mathcal{M}_{6}^{(ijk)}$ (see Proposition \ref{prop:M6-covering}
below). For the quotient $\mathcal{M}_{6}^{(ijk)}$, the diagram (\ref{eq:Gale-daig})
applies. Also, by evaluating the residue (\ref{eq:Period-E-cv}) in
the present case and writing the result in $\tz=1$ coordinate, we
arrive at 
\[
\Pi(a)=\int_{\gamma}\frac{1}{\sqrt{(a_{04}+a_{14}\frac{\ty}{\tx}+a_{24}\frac{1}{\tx})(a_{15}+a_{25}\frac{1}{\ty}+a_{05}\frac{\tx}{\ty})(a_{06}+a_{16}\tx+a_{26}\ty)}}\frac{d\tx d\ty}{\tx\ty}.
\]
Observing the obvious similarity to (\ref{eq:Pi-CICY}) and (\ref{eq:A-CICY}),
we define the ordered set by 
\[
\mathcal{A}=\left\{ \left(\begin{smallmatrix}1\\
0\\
0\\
0\\
0
\end{smallmatrix}\right),\left(\begin{smallmatrix}1\\
1\\
0\\
0\\
0
\end{smallmatrix}\right),\left(\begin{smallmatrix}1\\
0\\
1\\
0\\
0
\end{smallmatrix}\right),\left(\begin{smallmatrix}1\\
0\\
0\\
\text{-}1\\
1
\end{smallmatrix}\right),\left(\begin{smallmatrix}1\\
0\\
0\\
\text{-}1\\
0
\end{smallmatrix}\right),\left(\begin{smallmatrix}1\\
1\\
0\\
0\\
\text{-}1
\end{smallmatrix}\right),\left(\begin{smallmatrix}1\\
1\\
0\\
1\\
\text{-}1
\end{smallmatrix}\right),\left(\begin{smallmatrix}1\\
0\\
1\\
1\\
0
\end{smallmatrix}\right),\left(\begin{smallmatrix}1\\
0\\
1\\
0\\
1
\end{smallmatrix}\right)\right\} .
\]
It should be easy to deduce the following proposition. 
\begin{prop}
\label{prop:GKZ-red-AG}The period integral satisfies the GKZ system
with the data $\mathcal{A}$ and the exponent $\beta=^{t}(-\frac{1}{2},-\frac{1}{2},-\frac{1}{2},0,0)$. 
\end{prop}

Although the exponent $\beta$ differs from Proposition \ref{prop:GKZ-Pi-CICY},
the general construction of local solutions applies by describing
the parameter space of the GKZ system in terms of the secondary fan
$\mathbb{P}_{\mathrm{Sec}\mathcal{A}}$. Here, we should note that
$\mathcal{M}_{6}^{(1,2,3)}\simeq\mathbb{P}_{\mathrm{Sec}\mathcal{A}}$
and also $\mathcal{M}_{6}^{(i,j,k)}\simeq\mathcal{M}_{6}^{(1,2,3)}$
by symmetry reason. It turns out that there are 108 regular triangulations,
which describe the affine charts of $\mathbb{P}_{\mathrm{Sec}\mathcal{A}}$.
Among 108 affine charts, 6 are singular at the origins. By symmetry,
6 singular points are locally isomorphic, and it turns out that there
are two possible subdivisions which resolve the singularity. It also
turns out that these two resolutions are related by the four dimensional
flip. Furthermore, for each of the two resolutions, we have the special
boundary point where there is only one regular solution and all others
have logarithmic singularities. We refer \cite[Sect.3.2]{HLTYk3}
for more details. To present the results partially, let $T_{o}$ be
one of the six triangulations and $\sigma_{T_{o}}$ be the normal
cone. One of the two resolutions is given by a subdivision $\sigma_{T_{o}}=\sigma_{T_{o}}^{(1)}\cup\sigma_{T_{o}}^{(2)}$
(eq.(14) of \cite{HLTYk3}). We take an integral basis $l^{(1)},\cdots,l^{(4)}$
of the semigroup $(\sigma_{T_{o}}^{(1)})^{\vee}\cap L$. With these
data, we have a set of differential operators 
\[
\begin{matrix}\mathcal{D}_{1}=(\theta_{1}+\theta_{2}-\theta_{4})(\theta_{1}+\theta_{3}-\theta_{4})+z_{1}(\theta_{1}+\frac{1}{2})(\theta_{1}-\theta_{4}),\\
\mathcal{D}_{2}=(\theta_{1}+\theta_{2}-\theta_{4})(\theta_{2}+\theta_{3}-\theta_{4})+z_{2}(\theta_{2}+\frac{1}{2})(\theta_{2}-\theta_{4}),\\
\mathcal{D}_{3}=(\theta_{1}+\theta_{3}-\theta_{4})(\theta_{2}+\theta_{3}-\theta_{4})+z_{3}(\theta_{3}+\frac{1}{2})(\theta_{3}-\theta_{4}),\\
\mathcal{D}_{4}=(\theta_{2}-\theta_{4})(\theta_{3}-\theta_{4})-z_{1}z_{4}(\theta_{1}+\frac{1}{2})(\theta_{2}+\theta_{3}-\theta_{4}),\;\;\;\;\\
\mathcal{D}_{5}=(\theta_{1}-\theta_{4})(\theta_{3}-\theta_{4})-z_{2}z_{4}(\theta_{2}+\frac{1}{2})(\theta_{1}+\theta_{3}-\theta_{4}),\;\;\;\;\\
\mathcal{D}_{6}=(\theta_{1}-\theta_{4})(\theta_{2}-\theta_{4})-z_{3}z_{4}(\theta_{3}+\frac{1}{2})(\theta_{1}+\theta_{2}-\theta_{4}),\;\;\;\;\\
\mathcal{D}_{7}=(\theta_{1}+\theta_{2}-\theta_{4})(\theta_{3}-\theta_{4})+z_{1}z_{2}z_{4}(\theta_{1}+\frac{1}{2})(\theta_{2}+\frac{1}{2}),\;\;\\
\mathcal{D}_{8}=(\theta_{1}+\theta_{3}-\theta_{4})(\theta_{2}-\theta_{4})+z_{1}z_{3}z_{4}(\theta_{1}+\frac{1}{2})(\theta_{3}+\frac{1}{2}),\;\;\\
\mathcal{D}_{9}=(\theta_{2}+\theta_{3}-\theta_{4})(\theta_{1}-\theta_{4})+z_{2}z_{3}z_{4}(\theta_{2}+\frac{1}{2})(\theta_{3}+\frac{1}{2}),\;\;
\end{matrix}
\]
which determines the period integrals $w^{T_{o}}(z):=\sqrt{a_{04}a_{15}a_{06}}\,\Pi(a)$,
where we set $z_{k}:=(-1)^{l_{0}^{(k)}}a^{l^{(k)}}$ and $\theta_{k}:=z_{k}\frac{\partial\;}{\partial z_{k}}$.
At $z_{1}=\cdots=z_{4}=0$, there is a unique (up to constant multiples)
power series solution $w_{0}^{T_{o}}(z)=\sum_{n\in\mathbb{Z}_{\geq0}^{4}}c(n)z^{n}$
with 
\[
c(n):=\frac{1}{\Gamma(\frac{1}{2})^{3}}\frac{\Gamma(n_{1}+\frac{1}{2})\Gamma(n_{2}+\frac{1}{2})\Gamma(n_{3}+\frac{1}{2})}{\prod_{i=1}^{3}\Gamma(n_{4}-n_{i}+1)\,\prod_{1\leq j\leq k\leq3}\Gamma(n_{j}+n_{k}-n_{4}+1)}.
\]
These generalize the equations in (\ref{eq:PF-w0-lambda}). Together
with the solutions with logarithmic singularities, as before, we can
define a period map and also K3 lambda functions as its inverse. Making
a parallel argument to the elliptic lambda function, we can also write
these K3 lambda functions in terms of suitable theta functions \cite{HLYk3}.
In 1990s, possible form of differential operators with its monodromy
properties and also the period domain for the family of K3 surfaces
$\mathfrak{X}\rightarrow\mathcal{M}_{6}^{0}$ were studied extensively
by many authors (see \cite{MStYi,MStYii,MSYintJ} and references therein).
However, the final step, i.e., obtaining the differential operators
which give the K3 lambda functions in terms of the solutions was left
open until very recently (cf. \cite[Appendix E2]{HLYk3}). 

(3.5.2.c) It is natural to ask how the moduli space $\mathcal{M}_{6}^{(i,j,k)}=\mathbb{P}_{\mathrm{Sec}\mathcal{A}}$,
over which we have a nice GKZ system as above, is related to $\mathcal{M}_{6}$
and $\mathcal{M}_{3,3}$. Their relations come from the following
proposition: 
\begin{prop}
The toric variety $\mathbb{P}_{\mathrm{Sec}\mathcal{A}}\simeq\mathcal{M}_{6}^{(i,j,k)}$
is a toric resolution of $\mathcal{M}_{3,3}$. Namely the normal fan
$\mathcal{N}(\mathrm{Sec}\mathcal{A})$ can be identified with a subdivision
of the fan of $\mathcal{M}_{3,3}$. 
\end{prop}

We can see this property by looking into the fans of toric varieties,
see \cite[Prop.4.9]{HLTYk3}. Fix a toric resolution $\hat{\mathcal{M}}_{6}^{(i,j,k)}\rightarrow\mathcal{M}_{6}^{(i,j,k)}$
(there are 6 singular points, each of which has two possible toric
resolutions as described below Proposition \ref{prop:GKZ-red-AG}),
and consider the composite of the (birational) maps,
\[
\hat{\mathcal{M}}_{6}^{(i,j,k)}\rightarrow\mathcal{M}_{6}^{(i,j,k)}\rightarrow\mathcal{M}_{3,3}\overset{\sim}{\dashrightarrow}\mathcal{M}_{6}.
\]
We note the fact that the birational map $\phi:\mathcal{M}_{3,3}\overset{\sim}{\dashrightarrow}\mathcal{M}_{6}$
defines a one-to-one map to its image in the complement $\mathcal{M}_{3,3}\setminus D_{0}$
of a certain divisor $D_{0}$, see \cite[Prop.6.2]{HLYk3}. We also
note that the symmetry group $S_{6}$, which exchanges the columns
of the matrix $\tA$, acts naturally on both $\mathcal{M}_{3,3}$
and $\mathcal{M}_{6}$, and the birational map $\phi:\mathcal{M}_{3,3}\overset{\sim}{\dashrightarrow}\mathcal{M}_{6}$
is equivariant under this action. Using these, we have the following
covering property:
\begin{prop}
\label{prop:M6-covering}The moduli space $\mathcal{M}_{6}$ is covered
by the copies of $\mathcal{M}_{3,3}\setminus D_{0}$ as follows: 
\begin{equation}
\mathcal{M}_{6}=\bigcup_{\sigma\in S_{6}}\psi_{\sigma}\circ\phi\,(\mathcal{M}_{3,3}\setminus D_{0}),\label{eq:covering-M6}
\end{equation}
where $\psi_{\sigma}:\mathcal{M}_{6}\rightarrow\mathcal{M}_{6}$ represents
the action of $\sigma\in S_{6}$. 
\end{prop}

Any of the resolution $\hat{\mathcal{M}}_{6}^{(i,j,k)}\rightarrow\mathcal{M}_{3,3}$
gives corresponding resolution of the singularities on $\mathcal{M}_{3,3}\setminus D_{0}$.
The above covering property shows that the singularities of $\mathcal{M}_{6}$,
which consists of 15 lines of $A_{1}$ singularity intersecting at
15 points, are not a toric singularity, but these are locally toric
and admits a toric resolution in terms of the secondary fan, i.e.,
the normal fan of the secondary polytope. Since the resolution of
$\mathcal{M}_{3,3}\setminus D_{0}$ by the secondary fan is compatible
with the GKZ system in Proposition \ref{prop:GKZ-red-AG}, we obtain
the Picard-Fuchs differential operators to define K3 lambda functions.
We may summarize this result in general terms as follow:
\begin{prop}
The local system of the family of K3 surfaces $\mathfrak{X}\rightarrow\mathcal{M}_{6}$
is trivialized by the GKZ systems over $\psi_{\sigma}\circ\phi(\mathcal{M}_{3,3}\setminus D_{0})$
in (\ref{eq:covering-M6}).
\end{prop}

(3.5.2.d) Using the covering property (\ref{eq:covering-M6}), we
obtain a resolution of $\mathcal{M}_{6}$ at least complex analytically.
There are two choices for each of the toric resolutions $\hat{\mathcal{M}}_{6}^{(i,j,k)}\rightarrow\mathcal{M}_{6}^{(i,j,k)}$
at 6 singular points of $\mathcal{M}_{6}^{(i,j,k)}=\mathbb{P}_{\mathrm{Sec}\mathcal{A}}$,
which are related by a four dimensional flip. Correspondingly, we
have $2\times6$ possible toric resolutions for$\hat{\mathcal{M}}_{6}^{(i,j,k)}\rightarrow\mathcal{M}_{3,3}$
and this determines a resolution of $\mathcal{M}_{6}$ with the covering
(\ref{eq:covering-M6}). On the other hand, by blowing up along the
15 singular lines in $\mathcal{M}_{6}$, we see $2\times15$ singular
points in the blow up $\mathcal{M}_{6}'$ of $\mathcal{M}_{6}$, namely
over each 15 intersection points of the lines, we have 2 singular
points after the blow-up (see \cite[Prop.5.2]{HLTYk3}). Blowing-up
$\mathcal{M}_{6}'$ at these $2\times15$ points, we obtain a projective
resolution $\widetilde{\mathcal{M}}_{6}\rightarrow\mathcal{M}_{6}$
(see \cite[Prop.5.3, Sect.6]{HLTYk3}). By studying local geometry
of the resolution, we see that $\widetilde{\mathcal{M}}_{6}$ represents
the resolution where we chose the same type of toric resolution for
all 6 singular points in $\hat{\mathcal{M}}_{6}^{(i,j,k)}\rightarrow\mathcal{M}_{3,3}$.
Flipping the type of the toric resolution for all at one time gives
another projective resolution $\widetilde{\mathcal{M}}_{6}^{+}$ (\cite[Thm.6.12]{HLTYk3}).
The variety $\mathcal{M}_{6}$ defined by (\ref{eq:Coble-M6}) is
called Coble fourfold \cite{ChKS}. We summarize our result as follows:
\begin{prop}
The Coble fourfold $\mathcal{M}_{6}$ defined by (\ref{eq:Coble-M6})
has two projective resolutions $\widetilde{\mathcal{M}}_{6}$ and
$\widetilde{\mathcal{M}}_{6}^{+}$ which are related by a four dimensional
flip. 
\end{prop}

These projective resolutions are compatible with the natural $S_{6}$
action. Correspondingly, the K3 lambda functions defined over them
admit the $S_{6}$ action naturally, which corresponds to the standard
$S_{3}\simeq\Gamma/\Gamma(2)$ action on the elliptic lambda function
(see \cite[Appendix C]{HLYk3} for more details). 

~

~

\section{\protect\label{sec:Sec4}Mirror symmetry from the moduli spaces $\mathcal{M}_{f}$
and $\mathcal{M}_{f_{1}\cdots f_{r}}$ }

Suppose a Calabi-Yau manifold $X$ (of dimension three) is given.
In general, it is a difficult problem to find its mirror Calabi-Yau
manifold. However, for Calabi-Yau complete intersections in toric
varieties, Batyrev-Borisov mirror construction in Subsect.\ref{subsec:CICY-BB-const}
gives pairs of Calabi-Yau manifolds $(X,\check{X})$ which are topological
mirror to each other. In the works \cite{Cand-II,Cand-IIs,HKTY1,HKTY2}
in early 90s, it was verified that these pairs are in fact mirror
symmetric when $h^{2,1}(\check{X})$ is small. Now, efficient packages,
e.g. \cite{CYtools}, of computer codes are available to verify this
for huge numbers of Calabi-Yau hypersurfaces in the list \cite{KSlist}. 

\subsection{Mirror symmetry from moduli spaces }

To verify mirror symmetry in Definition \ref{def:mirror-def} for
a pair $(X,\check{X})$ satisfying $h^{1,1}(X)=h^{2,1}(\check{X})$,
we need to have a family $\check{\mathfrak{X}}\rightarrow\mathcal{M}_{\check{X}}$
where we see a special degeneration point described in Sect. \ref{sec:Families-of-Calabi-Yau},
i.e., a boundary point called large complex structure limit (LCSL).
For families of Calabi-Yau hypersurfaces and complete intersections
defined for reflexive polytopes $(\Delta,\Delta^{*})$ surveyed in
Sect. \ref{sec:Families-of-Calabi-Yau}, we can answer to this problem
in terms of the diagram (\ref{eq:Gale-daig}) which we used to define
the moduli spaces $\mathcal{M}_{f}$ and $\mathcal{M}_{f_{1}\cdots f_{r}}$.
Below, for brevity, we restrict our attention to the case of hypersurfaces. 

As we notice, the diagram (\ref{eq:Gale-daig}) is symmetric with
respect to the matrices $A$ and $B$. We used the diagram to describe
the GIT quotient $\mathcal{M}_{f}=\mathbb{C}^{p+1}//_{\chi}\overline{T}$
by arranging the weight vectors of the group action $\overline{T}=(\mathbb{C}^{*})^{5}$
into a $5\times(p+1)$ matrix $A$ in the second line. The $(p-4)\times(p+1)$
integral matrix $B$ is defined by arranging the kernels into row
vectors. Here we can change the picture by regarding the matrix $B$
representing the weight of a torus action $\check{T}:=(\mathbb{C}^{*})^{p-4}$
in the affine space $\mathbb{C}^{p+1}$. Namely, we can read the first
line of (\ref{eq:Gale-daig}) with $B$ as the sequence defining a
GIT quotient $\mathbb{C}^{p+1}//_{\check{\chi}}\check{T}$, where
we take a character $\check{\chi}$ in the image $\mathrm{Im}B\subset K$.
In general, GIT quotient depends on the choice of the character $\check{\chi}\in\mathrm{Im}B$
for the linearization $L_{\check{\chi}}$, and this introduces a fan
in $\mathrm{Im}B\otimes\mathbb{R}=K\otimes\mathbb{R}$ which is called
GIT fan. As we see in Proposition \ref{prop:GIT-by-Sec-fan}, the
GIT fan for the quotient $\mathcal{M}_{f}=\mathbb{C}^{p+1}//_{\chi}\overline{T}$
is described by the cone $\mathrm{Cone}(A)$. In contrast to this,
the GIT quotient $\mathbb{C}^{p+1}//_{\check{\chi}}\check{T}$ has
a complicated ``phase'' structure. 
\begin{prop}
The GIT fan for the quotient $\mathbb{C}^{p+1}//_{\check{\chi}}\check{T}$,
represented by the first line of (\ref{eq:Gale-daig}), is given by
the secondary fan, i.e., the normal fan $\mathcal{N}(\mathrm{Sec}\Delta^{*})$. 
\end{prop}

\begin{proof}
See for example \cite[Thm.14.4.7]{CoxBook}. 
\end{proof}
We recall that the secondary polytope is determined by the regular
triangulations of $\Delta^{*}$. Among regular triangulations, we
have already encountered with a maximal triangulation $T_{o}$ which
brings us a distinguished property of a local solution (\ref{eq:w0-To})
in $U_{T_{0}}=\mathbb{C}[\sigma_{T_{o}}^{\vee}\cap L]$. Here we note
that the cone $\sigma_{T_{0}}$ is one of the cones (phases) of the
GIT fan. Also, we note by definition of the maximal triangulation
(given in a paragraph above the equation (\ref{eq:To-L-basis})) that,
corresponding to $T_{o}$, we have a fan $\Sigma(\Delta^{*},T_{o})$
which is a subdivision of the normal fan $\mathcal{N}(\Delta)$ and
defines a partial resolution $\mathbb{P}_{\Delta}'\rightarrow\mathbb{P}_{\Delta}$. 
\begin{prop}
For a character $\check{\chi}\in\sigma_{T_{0}}$, we have $\mathbb{C}^{p+1}//_{\check{\chi}}\check{T}=K_{\mathbb{P}'_{\Delta}}$.
Namely, the GIT quotient describes the total space of the canonical
bundle $K_{\mathbb{P}_{\Delta}'}$ over the partial resolution $\mathbb{P}_{\Delta}'$
of $\mathbb{P}_{\Delta}.$
\end{prop}

\begin{proof}
We can show this by applying \cite[Thm.14.4.7]{CoxBook} to our setting.
The total space of the canonical bundle $K_{\mathbb{P}_{\Delta}'}=X_{\Sigma}$
comes from the normal fan $\Sigma$ of the virtual polytope $P_{\bm{a}}$
with $\check{\chi}=B\bm{a}$. 
\end{proof}
The partial resolution $\mathbb{P}_{\Delta}'\rightarrow\mathbb{P}_{\Delta}$
factors the MPCP resolution $\hat{\mathbb{P}}_{\Delta}\rightarrow\mathbb{P}_{\Delta}$.
Here we should note that, since we remove points on codimension-one
faces to define the diagram (\ref{eq:Gale-daig}), the partial resolution
$\mathbb{P}_{\Delta}'$ does not coincides with $\hat{\mathbb{P}}_{\Delta}$
in general. However we note that we still have a smooth Calabi-Yau
manifold $\hat{Z}_{f_{\Delta}}$ in $\mathbb{P}_{\Delta}'$ because
general hypersurfaces avoid the singularities of $\mathbb{P}_{\Delta}'$.
The total space $K_{\mathbb{P}_{\Delta}'}$ describes exactly the
geometry of the so-called gauged linear sigma models (GLSMs) in physics
\cite{HoriT,HoriKapp}. In GLSMs, the above GIT fan is called phase
diagram and other chambers (cones of maximal dimensions) of the fan
has important meanings as representing different phases of the theory.

\subsection{Cohomology-valued hypergeometric series}

We now see two GIT quotients in the diagram (\ref{eq:Gale-daig})
defined for a pair of reflexive polytopes $(\Delta,\Delta^{*})$;
one is to define the moduli space $\mathcal{M}_{f_{\Delta^{*}}}$
of a family $\left\{ \hat{Z}_{f_{\Delta^{*}}}(a)\right\} $ in $\hat{\mathbb{P}}_{\Delta^{*}}$
and the other is to define the total space $K_{\mathbb{P}_{\Delta}'}$
with a character $\check{\chi}\in\sigma_{T_{o}}\cap\mathrm{Im}B$
for a maximal triangulation. Here, we note that the maximal triangulation
$T_{o}$ is used in two different contexts; one is to describe an
affine chart $U_{T_{o}}=\mathrm{Spec}\mathbb{C}[\check{\sigma}_{T_{o}}\cap L]$
of $\mathcal{M}_{f_{\Delta^{*}}}$ and the other is to describe a
character $\check{\chi}\in\sigma_{T_{o}}\cap\mathrm{Im}B$. 

\subsubsection{Cohomology ring $H^{*}(\hat{Z}_{f_{\Delta}},\mathbb{Z})$}

The zero section $\mathbb{P}_{\Delta}'$ of the total space $K_{\mathbb{P}_{\Delta}'}$
is a toric variety defined by the fan $\Sigma(\Delta^{*},T_{o})$.
From this viewpoint, we modify the first line of (\ref{eq:Gale-daig})
to
\[
K\,\,\overset{\,B'}{\leftarrow}\,\,\check{\mathbb{Z}}^{p}\,\,\overset{^{t}A'}{\leftarrow}\,\,M\,\,\leftarrow\,\,0
\]
with $4\times p$ matrix $A'=(\nu_{1}\nu_{2}\cdots\nu_{p})$ and $(p-4)\times p$
matrix $B'$. The matrix $A'$ is obtained by removing the first column
(i.e., $\nu_{0})$ and the first row of $A$, and correspondingly
$B'$ is obtained by removing the first column of $B$. Note that
we still have $K\simeq\mathbb{Z}^{p-4}$ and $\sigma_{T_{o}}\subset K\otimes\mathbb{R}$.
We associate a piecewise linear function $\varphi_{c}:\Sigma(\Delta^{*},T_{o})\rightarrow\mathbb{R}$
to each elements $c=(c_{1},\cdots,c_{p})\in\check{\mathbb{Z}}^{p}$
by defining $\varphi_{c}(\nu_{k})=c_{k}$ $(k=1,\cdots,p)$. The image
of $M$ by $\check{\mathbb{Z}}^{p}\overset{^{t}A'}{\leftarrow}M$
represents global linear functions defined by the property $\varphi_{m}(\nu_{k})=\langle m,\nu_{k}\rangle$
$(m\in M).$ If the fan $\Sigma(\Delta^{*},T_{o})$ is smooth, these
piecewise linear functions up to the image of $^{t}A'$ describe isomorphism
classes of torus equivariant line bundles on $\mathbb{P}_{\Delta}'$
\cite[Sect.2.1]{TOda}. Then, the cone $\sigma_{T_{o}}\subset K\otimes\mathbb{R}$
represents strongly convex piecewise linear functions, which corresponds
to the ample line bundles on $\mathbb{P}_{\Delta}'$ .
\begin{prop}
The normal cone $\sigma_{T_{o}}\subset K\otimes\mathbb{R}$ represents
the ample cone (K\"ahler cone) of $\mathbb{P}_{\Delta}'$. 
\end{prop}

\begin{proof}
Since $\mathbb{P}_{\Delta}'$ is a projective toric variety, we can
apply in \cite[Thm.15.1.10]{CoxBook} and \cite[Thm.6.3.22]{CoxBook}.
\end{proof}
We write the line bundle corresponding to $\varphi_{c}$ with $c=(0,..,1,...,0)$
(zero except the $i$-th entry) by $\mathcal{O}_{\mathbb{P}_{\Delta}'}(D_{i})$
or simply $D_{i}$. Also, from the combinatorial property of the fan
$\Sigma(\Delta^{*},T_{o})$, we have the so-called Stanley-Reisner
ideal $SR_{\Sigma(\Delta^{*},T_{o})}$ \cite[Sect.12.4]{CoxBook}\cite[Def.1.6]{MStern}
in the polynomial ring $\mathbb{Z}[D_{1},\cdots,D_{p}]$. The following
proposition is a general result in toric geometry (e.g. \cite[Sect.5.2,Prop.]{Fulton}
): 
\begin{prop}
\label{prop:Chow-PD}When $\mathbb{P}_{\Delta}'$ is smooth, the cohomology
ring $H^{*}(\mathbb{P}_{\Delta}',\mathbb{Z})$ is described by 
\begin{equation}
\mathbb{Z}[D_{1},\cdots,D_{p}]/(SR_{\Sigma(\Delta^{*},T_{o})}+\mathrm{Im}^{t}A),\label{eq:coh-by-SR}
\end{equation}
where we normalize the ring by $\prod(1+D_{i})[\mathbb{P}_{\Delta}]=e(\mathbb{P}_{\Delta}')$
with the Euler number $e(\mathbb{P}_{\Delta}')$ which is given by
the number of maximal dimensional cones of $\Sigma(\Delta^{*},T_{o})$. 
\end{prop}

Note that for toric varieties we have $H^{*}=\oplus_{p}H^{p,p}$ in
general (see e.g. \cite[Sect.3.3]{TOda}). 

For most cases, the toric variety $\mathbb{P}_{\Delta}'$ is singular.
In that case, the above proposition does not hold since some divisors
have fractional intersection numbers due to toric singularities. However
the form of the ring structure (\ref{eq:coh-by-SR}) still holds over
$\mathbb{Q}$ (see \cite[Thm.12.4.1]{CoxBook}). Since our Calabi-Yau
surfaces $\hat{Z}_{f_{\Delta}}$ do not intersect with the toric singularities,
we can describe $H^{*}(\hat{Z}_{f_{\Delta}},\mathbb{Z})$ by restricting
the ring (\ref{eq:coh-by-SR}) to $\hat{Z}_{f_{\Delta}}$ with the
normalization $\int_{\mathbb{P}_{\Delta}'}c(\hat{Z}_{f_{\Delta}})[\hat{Z}_{f_{\Delta}}]=2(h^{1,1}(\hat{Z}_{f_{\Delta}})-h^{2,1}(\hat{Z}_{f_{\Delta}}))$,
where we use 
\[
c(\hat{Z}_{f_{\Delta}})=\frac{\Pi_{i=1}^{p}(1+D_{i})}{1+[\hat{Z}_{f_{\Delta}}]}=1+0+c_{2}(\hat{Z}_{f_{\Delta}})+c_{3}(\hat{Z}_{f_{\Delta}})
\]
with $[\hat{Z}_{f_{\Delta}}]=D_{1}+\cdots+D_{p}$ and the formula
(\ref{eq:B-formula}) for the Hodge numbers. The restriction to $\hat{Z}_{f_{\Delta}}$
may be done by replacing the ideal $I=(SR_{\Sigma(\Delta^{*},T_{o})}+\mathrm{Im}^{t}A)$
with its ideal quotient $I:[\hat{Z}_{f_{\Delta}}]$. We remark that
we have removed the points on codimension-one faces of $\Delta^{*}$
to define $A=\left\{ \nu_{1},\cdots,\nu_{p}\right\} $, hence we have
$\mathrm{rk}K=\mathrm{rk}H^{2}(\hat{Z}_{f_{\Delta}})$ if the fourth
term (representing the twisted sectors) in the r.h.s of (\ref{eq:Mf})
has no contribution in $h^{1,1}(\hat{Z}_{f_{\Delta}})$. 

\subsubsection{\protect\label{subsec:Coh-v-hypg-ser}Cohomology-valued hypergeometric
series}

Now we combine the above results with the diagram (\ref{eq:Gale-daig}).
As above, we fix a maximal triangulation $T_{o}$ of $\Delta^{*}$.
We start with the GIT quotient represented by the second line to define
the moduli space $\mathcal{M}_{f_{\Delta^{*}}}$ of the the family
$\left\{ \hat{Z}_{f_{\Delta^{*}}}(a)\right\} $. From this point of
view, the normal cone $\sigma_{T_{0}}$ describes an affine chart
$U_{T_{o}}\mathrm{=Spec}\mathbb{C}[\sigma_{T_{0}}^{\vee}\cap L]$
of $\mathcal{M}_{f_{\Delta^{*}}}$. If $\sigma_{T_{o}}$ is not a
regular cone, we can subdivide it to have regular simplicial cones
and take one from them. So we may assume $\sigma_{T_{o}}$ is a regular
simplicial cone, and take a basis 
\begin{equation}
\left\{ l^{(1)},\cdots,l^{(p-4)}\right\} \label{eq:Basis-sigma-V-To}
\end{equation}
of the semigroup $\sigma_{T_{o}}^{\vee}\cap L$. This basis introduces
a ``good'' coordinate $x_{k}=(-1)^{l_{0}^{(k)}}a^{l^{(k)}}$ on
the affine chart $U_{T_{o}}$ where we have a unique power series
solution $w_{0}^{T_{o}}(x)$ (\ref{eq:w0-To}). 

On the other hand, turning our attention to the GIT quotient represented
by the first line of (\ref{eq:Gale-daig}), we can read the same cone
$\sigma_{T_{o}}\subset K\otimes\mathbb{R}$ as the K\"ahler cone
of $\mathbb{P}_{\Delta}'$ in the GIT quotient $\mathbb{C}^{p+1}//_{\check{\chi}}\check{T}=K_{\mathbb{\mathbb{P}}_{\Delta}'}$
with a character $\check{\chi}\in\sigma_{T_{o}}$. For simplicity,
we assume that K\"ahler cone of $\mathbb{P}_{\Delta}'$ coincides
with that of the hypersurface $\hat{Z}_{f_{\Delta}}\subset\mathbb{P}_{\Delta}'$.
Also, as above, we may assume that $\sigma_{T_{o}}$ is a regular
simplicial cone, and take a basis of the semi-group $\sigma_{T_{0}}\cap K$,
\begin{equation}
\left\{ J_{1},\cdots,J_{p-4}\right\} \label{eq:Basis-sigma-To}
\end{equation}
so that this is the dual basis of (\ref{eq:Basis-sigma-V-To}). Then,
by the natural dual pairing $K\times L\rightarrow\mathbb{Z}$, we
have the relations $\langle J_{i},l^{(k)}\rangle=\delta_{i}^{k}$.
By our assumption, the elements $J_{k}$ generate the cohomology $H^{2}(\hat{Z}_{f_{\Delta}},\mathbb{Z})$
and also $H^{even}(\hat{Z}_{f_{\Delta}},\mathbb{Z})=\oplus_{p}H^{p,p}(\hat{Z}_{f_{\Delta}},\mathbb{Z})$
due to Proposition \ref{prop:Chow-PD} and the paragraph after it. 

The following definition may be recognized naturally in the equations
of the works \cite{HKTY1,HKTY2,HLYcmp}, but was written out explicitly
in a sentence in \cite{Hos} to introduce the central charge formula
(see Conjecture \ref{conj:central-charge} below). 
\begin{defn}
As an element of $\mathbb{C}\{x_{1},..,x_{p-4}\}[\log x_{1},..,\log x_{p-4}]\otimes H^{\mathrm{even}}(\hat{Z}_{f_{\Delta}},\mathbb{Q})$,
we define 
\begin{equation}
w_{o}^{T_{o}}(x,\frac{J}{2\pi i}):=\sum_{n_{1},...,n_{p-n}\geq0}\frac{\Gamma(-(n+\frac{J}{2\pi i})\cdot l_{0}+1)}{\prod_{i\geq1}\Gamma((n+\frac{J}{2\pi i})\cdot l+1)}x^{n+\frac{J}{2\pi i}},\label{eq:omega-J}
\end{equation}
where we set $(n+\frac{J}{2\pi i})\cdot l:=\sum_{k}(n_{k}+\frac{J_{k}}{2\pi i})l^{(k)}$
and $x^{n+\frac{J}{2\pi i}}:=\prod_{k}x_{k}^{n_{k}}e^{\frac{J_{k}}{2\pi i}\log x_{k}}$.
We define the right hand side by the Taylor expansion with respect
to $J_{1},\cdots,J_{p-n}$. 
\end{defn}

We should note that, if there were not singularities from the Gamma
functions, the Taylor expansion of the right hand side trivially makes
sense because all $J_{k}\in H^{\mathrm{even}}(\hat{Z}_{f_{\Delta}},\mathbb{Q})$
are nilpotent. However, since singular expressions of Gamma functions,
i.e. $\Gamma(-n)$ for positive integers (and derivatives of these)
appear in the evaluation, we replace them by $\Gamma(-n+\varepsilon)$
and consider the limit $\varepsilon\rightarrow0$ to define the right
hand side of (\ref{eq:omega-J}). It was found \cite{HKTY1,HKTY2,HLYcmp}
that each coefficient of $x^{n}$ might be divergent, however, if
we sum up all the coefficients, we have a finite limit (see \cite[Appendix A]{HKTY2}
and \cite{Lib} for the resulting expressions of the limit). Now,
after Taylor expansion, we obtain hypergeometric series $f_{k}(x)$
$(k=1,\cdots,\dim H^{\mathrm{even}}(\hat{Z}_{f_{\Delta}},\mathbb{Q}))$
with $f_{0}(x)=w_{0}^{T_{o}}(x)$ and others containing logarithmic
singularity. 
\begin{rem}
If we write (\ref{eq:omega-J}) as $w_{0}(x,\rho)=\sum_{n}c(n+\rho)x^{n+\rho}$,
then the Taylor expansion with respect to $\rho=\frac{J}{2\pi i}$
reminds us the classical Frobenius method for hypergeometric differential
equations of one variable, where our limit $\varepsilon\rightarrow0$
corresponds to $\rho\rightarrow0$ after differentiation of $c(n+\rho)$
by $\rho$. It should be noted, however, that the limit $\rho\rightarrow0$
for the derivations $c(n+\rho),$$\frac{\partial}{\partial\rho_{i}}c(n+\rho),\frac{\partial^{2}}{\partial\rho_{i}\partial\rho_{j}}c(n+\rho),\cdots$
behaves quite different for multi-variables: For example, it is often
the case that $c(n)=0$ for some region, e.g., $\left\{ (n_{1},n_{2})\mid n_{1}<n_{2}\right\} \subset\mathbb{Z}_{\geq0}^{2}$
but the support of $\frac{\partial}{\partial\rho_{i}}c(n)$ extends
to entire $\mathbb{Z}_{\geq0}^{2}$. In general the support of the
derivatives $\frac{\partial}{\partial\rho_{i}}c(n),\cdots$ can be
even larger than $\mathbb{Z}_{\geq0}^{2}$. The discovery made in
\cite{HKTY1,HKTY2} is that, for the series $w_{0}^{T_{o}}(x)$ (\ref{eq:omega-J}),
the support stays in $\mathbb{Z}_{\geq0}^{p-n}$ and the naive extension
of the Frobenius method applies to $w_{0}^{T_{o}}(x)$. It is easy
to find a hypergeometric series of multi-variables whose support of
the $\rho$-derivatives extends to $\mathbb{Z}^{p-n}$, i.e., negative
power appears and the naive Frobenius method does not work (see \cite[Appendix E2]{HLYk3}
to find an example). 
\end{rem}

\begin{prop}
By handling the limit as described above, we obtain 
\begin{equation}
w_{o}^{T_{o}}(x,\frac{J}{2\pi i})=f_{0}(x)1+\sum_{k=1}^{d}f_{k}(x)e_{k}\label{eq:wo-x-J}
\end{equation}
from the formal definition (\ref{eq:omega-J}), where $\left\{ e_{0},e_{1},\cdots,e_{d}\right\} $
is a basis of $H^{even}(\hat{Z}_{f_{\Delta}},\mathbb{Q})$. Then it
holds that $f_{0}(x)=w_{0}^{T_{o}}(x)$ and the coefficients $f_{0}(x),f_{1}(x),\cdots,f_{d}(x)$
form a complete set of the period integrals of the (mirror) family
$\check{\mathfrak{X}}\rightarrow\mathcal{M}_{f_{\Delta^{*}}}$. 
\end{prop}

The above proposition is based on explicit calculations. We will rephrase
the proposition as a conjecture in the next subsection (Conjecture
\ref{conj:central-charge}). 
\begin{lem}
\label{lem:T-exp-J}Let $\mathcal{T}_{k}$ be the analytic continuation
of the solutions $\left\{ f_{0}(x),\cdots,f_{d}(x)\right\} $ around
a divisor $\left\{ x_{k}=0\right\} \subset U_{T_{o}}$ along a path
$x_{k}(t)=x_{k}e^{2\pi it}(0\leq t\leq1)$. Then the following relation
holds: 
\[
\mathcal{T}_{k}\,w_{o}^{T_{o}}(x,\frac{J}{2\pi i})=e^{J_{k}}\,w_{o}^{T_{o}}(x,\frac{J}{2\pi i}).
\]
\end{lem}

\begin{proof}
This property simply follows from the formal calculation:
\[
\mathcal{T}_{k}:x^{n+\frac{J}{2\pi i}}\mapsto\prod_{i\not=k}x^{n_{i}}e^{\frac{J_{i}}{2\pi i}\log x_{i}}\times x^{n_{k}}e^{\frac{J_{k}}{2\pi i}\left(\log x_{k}+2\pi i\right)}=e^{J_{k}}x^{n+\frac{J}{2\pi i}}.
\]
Here $e^{J_{k}}$ should be understood as the multiplication by $e^{J_{k}}=1+J_{k}+J_{k}^{2}/2+\cdots$. 
\end{proof}
We will use this formal relation in Proposition \ref{prop:mirror-sym}.

\subsubsection{Canonical integral and symplectic basis}

In \cite[Prop.1]{HosIIa}, the following basis of $H^{\mathrm{even}}(\hat{Z}_{f_{\Delta}},\mathbb{Q})$
was introduced:
\begin{equation}
b^{(0)}=1,\,\,\,\,\,\,b_{i}^{(1)}=J_{i}-\frac{c_{2}(\hat{Z}_{f_{\Delta}})J_{i}}{12}-\sum_{k}a_{ki}b_{k}^{(2)},\,\,\,\,\,\,b_{j}^{(2)},\,\,\,\,\,\,b^{(3)}=-\mathrm{vol},\label{eq:basis-b}
\end{equation}
where $b_{j}^{(2)}\in H^{2,2}(\hat{Z}_{f_{\Delta}},\mathbb{Q})$ is
defined by the property $\int_{\hat{Z}_{f_{\Delta}}}J_{i}b_{j}^{(2)}=\delta_{ij}$
and $\mathrm{vol}\in H^{3,3}(\hat{Z}_{f_{\Delta}},\mathbb{Q})$ represents
the normalized volume form, $\int_{\hat{Z}_{f_{\Delta}}}\mathrm{vol}=1$.
We say that $b_{i}^{(2)},b^{(3)}$ have degree 2 and three, respectively.
We note that, contrary to the notation, the element $b_{i}^{(1)}$
is not pure; it consists of degree one part $J_{i}$ , degree three
part $\frac{c_{2}(\hat{Z}_{f_{\Delta}})J_{i}}{12}$, and the sum of
degree two elements $b_{k}^{(2)}$ with unknown parameters $a_{ki}\in\mathbb{Q}$.
For an element $\alpha=\alpha_{0}+\alpha_{1}+\alpha_{2}+\alpha_{3}\in H^{0,0}\oplus\cdots\oplus H^{3,3}$,
we define $*\alpha$ by linearly extending the action $*\alpha_{i}:=(-1)^{i}\alpha_{i}$,
and introduce a bi-linear form (which represents the Riemann-Roch
formula),
\[
\langle\alpha,\beta\rangle=\int_{\hat{Z}_{f_{\Delta}}}(*\alpha)\wedge\beta\wedge\mathrm{Todd_{\hat{Z}_{f_{\Delta}}},}
\]
with $\mathrm{Todd_{\hat{Z}_{f_{\Delta}}}}=1+\frac{c_{2}(\hat{Z}_{f_{\Delta}})}{12}$.
For the above basis, we can verify relations 
\[
\langle b^{(0)},b^{(3)}\rangle=-1=-\langle b^{(3)},b^{(0)}\rangle\,,\,\,\,\langle b_{i}^{(1)},b_{j}^{(2)}\rangle=-1=-\langle b_{j}^{(2)},b_{i}^{(1)}\rangle
\]
and all others vanish. Namely, $\mathcal{B}=\left\{ b^{(0)},b_{i}^{(1)},b_{j}^{(2)},b^{(3)}\right\} $
is a symplectic basis of $H^{even}(\hat{Z}_{f_{\Delta}},\mathbb{Q})$
with respect to $\langle\,,\,\rangle$. 
\begin{prop}
\label{prop:Pi-integral-symp}Arrange the expansion (\ref{eq:wo-x-J})
by the basis $\mathcal{B}$ as 
\begin{equation}
w_{o}^{T_{o}}(x,\frac{J}{2\pi i})=\Pi_{0}(x)b^{(0)}+\sum_{i}\Pi_{i}^{(1)}(x)b_{i}^{(1)}+\sum_{j}\Pi_{j}^{(2)}(x)b_{j}^{(2)}+\Pi^{(3)}(x)b^{(3)},\label{eq:w0-J-symplectic}
\end{equation}
and set $\Pi(x):=^{t}\left(\Pi_{0},\Pi_{1}^{(1)},\cdots,\Pi_{r}^{(1)},\Pi_{r}^{(2)},\cdots,\Pi_{1}^{(2)},\Pi^{(3)}\right)$.
There exist parameters $a_{ij}\in\mathbb{Q}$ for which $\Pi(x)$
gives an integral and symplectic basis of the Picard-Fuchs equations
with the symplectic form $\Sigma=\left(\begin{matrix}0 & J\\
-J & 0
\end{matrix}\right)$ where $J:=\left(\begin{smallmatrix} &  & 1\\
 & \overset{\,\,\,\,\,\cdot}{\cdot}\\
1
\end{smallmatrix}\right)$.
\end{prop}

The integral and symplectic basis $\Pi(x)$ above appears in many
applications of mirror symmetry. For convenience, we will collect
some useful formulas in Appendix~\ref{sec:AppendixA}.

\subsection{\protect\label{subsec:Cent-charge-for}Central charge formula}

We interpret the cohomology-valued hypergeometric series $w_{0}^{T_{o}}(x,\frac{J}{2\pi i})$
as the mirror counterpart of the period map. 

\subsubsection{Period maps}

The integral and symplectic basis arising from (\ref{eq:w0-J-symplectic})
has an interesting interpretation from mirror symmetry \cite{HosIIa,Hos}.
To see mirror symmetry in the expansion (\ref{eq:w0-J-symplectic}),
let us start with the family $\check{\mathfrak{X}}\rightarrow\mathcal{M}_{f_{\Delta^{*}}}$
and write by $\Omega_{x}$ the holomorphic three form of the fiber
$\check{X}_{x}:=\hat{Z}_{f_{\Delta^{*}},x}$ over $x\in\mathcal{M}_{f_{\Delta^{*}}}^{0}$.
We fix a base point $b_{o}\in\mathcal{M}_{f_{\Delta^{*}}}^{0}$ and
take an integral basis $\left\{ \gamma_{0},\cdots,\gamma_{d}\right\} $
of $H_{3}(\check{X}_{b_{o}},\mathbb{Z})$. We define the Poincar\'e
dual $\gamma_{i}^{D}$ by $\int_{\gamma_{i}}\omega=\int_{\check{X}}\omega\wedge\gamma_{i}^{D}$,
and make a basis $\left\{ \gamma_{0}^{D},\cdots,\gamma_{d}^{D}\right\} $
of $H^{3}(\check{X}_{b_{o}},\mathbb{Z})$. We represent the symplectic
form $(\,,\,)$ on $H^{3}(\check{X}_{b_{o}},\mathbb{Z})$ by a skew
symmetric matrix $\check{\chi}=(\check{\chi}_{ij})$ with $\check{\chi}_{ij}=(\gamma_{i}^{D},\gamma_{j}^{D})=\int_{\check{X}}\gamma_{i}^{D}\wedge\gamma_{j}^{D}$.
By making a local trivialization of the family $\check{\mathfrak{X}}\rightarrow\mathcal{M}_{f_{\Delta^{*}}}$,
we define period integrals by $\left(\int_{\gamma_{0}}\Omega_{x},\cdots,\int_{\gamma_{d}}\Omega_{x}\right)$,
or more generally we write period integrals by 
\begin{equation}
\tilde{\mathcal{P}}(x):=\sum_{i,j}\left(\int_{\gamma_{i}}\Omega_{x}\right)\check{\chi}^{ij}\gamma_{j}^{D},\label{eq:tilde-P}
\end{equation}
where $(\check{\chi}^{ij}):=\check{\chi}^{-1}$. Here we note that,
$\tilde{\mathcal{P}}(x)\in H^{3}(\check{X}_{b_{o}},\mathbb{C})$ is
defined invariantly under the change of basis, i.e., it is invariant
under a symplectic transformation $\gamma_{i}\mapsto\sum\gamma_{k}m_{ki}$
with $^{t}M\check{\chi}M=\check{\chi}$ and $M=(m_{ij})$. It is standard
to take a symplectic basis $\left\{ \gamma_{i}\right\} =\left\{ A_{0},A_{i},B_{j},B_{0}\right\} $
when we define the period map $\mathcal{P}:\mathcal{M}_{f_{\Delta^{*}}}\rightarrow\mathbb{P}(H^{3}(\check{X}_{b_{o}},\mathbb{C}))$
by $\mathcal{P}(x):=[\tilde{\mathcal{P}}(x)]$. The integral and symplectic
basis $\Pi(x)$ in Proposition \ref{prop:Pi-integral-symp} gives
a local expression of the period map on $U_{T_{o}}$ with respect
to a symplectic basis although cycles $\left\{ A_{0},A_{i},B_{j},B_{0}\right\} $
are implicit. 

\subsubsection{Local inverse of the period map}

In case of lower dimensional Calabi-Yau manifolds of dimensions less
than three, the image of the period map is described by the so-called
period domain given by certain symmetric spaces. For Calabi-Yau threefolds,
to have such global descriptions seems to be a hard problem. However,
due to Bryant and Griffiths \cite{BG}, it is known that the ratios
of period integrals with respect to a symplectic basis $\left\{ A_{0},A_{i},B_{j},B_{0}\right\} $,
\begin{equation}
t_{k}(x)=\int_{A_{k}}\Omega_{x}\Bigg/\int_{A_{0}}\Omega_{x}\,,\label{eq:Bryant-Griffiths}
\end{equation}
give an affine coordinate of the image $\mathrm{Im}\mathcal{P}\subset\mathbb{P}(H^{3}(\check{X}_{b_{o}},\mathbb{C}))$.
If we apply this general result to the case $\Pi(x)$ above, we have
locally on $U_{T_{o}}$ 
\begin{equation}
t_{k}(x)=\frac{\Pi_{i}^{(1)}(x)}{\Pi^{(0)}(x)}=\frac{1}{2\pi i}\log x_{k}+\text{powerseries of \ensuremath{x}}(k=1,\cdots,r).\label{eq:mirror-map}
\end{equation}
Inverting these relation, we obtain $x_{k}=x_{k}(t)=e^{2\pi\sqrt{-1}t_{k}}+\cdots$
which locally express the $\mathcal{P}^{-1}$. This local inverse
is called mirror map and calculated first in the work by Candelas
et al \cite{Cand}. We recall the definition $\mathcal{M}_{\check{X}}^{0}:=\mathcal{M}_{\check{X}}\setminus\left\{ \text{discriminant loci}\right\} $.
The image $\mathcal{P}(\mathcal{M}_{\check{X}}^{0})$ is expected
to be a certain domain $\mathcal{D}_{\check{X}}$, however we do not
know well about $\mathcal{D}_{\check{X}}$ in general. When we restrict
$\mathcal{M}_{\check{X}}^{0}$ to $U_{T_{0}}\setminus\left\{ x_{1}=\cdots=x_{r}=0\right\} $,
we can describe the image by (\ref{eq:mirror-map}) and find a local
isomorphism to a neighborhood at infinity of the so-called complexified
K\"ahler cone $\mathcal{K}_{X}^{\mathbb{C}}=H^{2}(X,\mathbb{R})+i\mathcal{K}_{X}$
of a mirror manifold $X$. Because of this, we write the domain $\mathcal{D}_{\check{X}}$
by $\mathcal{D}_{X}$. 

\subsubsection{Central charge formula}

The integral structure in Proposition \ref{prop:Pi-integral-symp}
can be understood well by the following diagram (\ref{eq:Hom-mir-diag})
which naturally follows from homological mirror symmetry \cite{Ko}.
Let us recall that two Calabi-Yau manifolds $X$ and $\check{X}$
are said mirror symmetric if there are equivalences $Mir:D(X)\simeq DFuk(\check{X},\beta)$
and $Mir':D(\check{X})\simeq DFuk(X,\beta')$ between the derived
category of coherent sheaves on a Calabi-Yau manifold $X$ (resp.
$\check{X}$) and the derived Fukaya category of another Calabi-Yau
manifold $\check{X}$ (resp. $X$). We focus on $Mir:D(X)\simeq DFuk(\check{X},\beta)$
and write natural isomorphisms which follow from the equivalence in
the following diagram: 
\begin{equation}
\xymatrix{\left\{ E_{0},\cdots,E_{d}\right\} ,\\
\left\{ E_{0}^{\vee},\cdots,E_{d}^{\vee}\right\} ,
}
\,\,\,\,\begin{matrix}\xymatrix{D(X)\;\;\ar[r]^{Mir\;\;\;}\ar[d]_{[\cdot]} & \;\;\;DFuk(\check{X},\beta)\ar[d]^{[\cdot]}\\
K_{num}(X)\;\;\ar[r]^{mir\;\;\;}\ar[d]_{\cdot^{\vee}} & \;\;H_{3}(\check{X},\mathbb{Z})\ar[d]_{\simeq}^{\mathrm{P.D.}}\,,\\
K_{num}(X)\;\;\ar[r]^{mir_{D}} & \;\;H^{3}(\check{X},\mathbb{Z})\,,
}
\end{matrix}\xymatrix{\left\{ \gamma_{0},\cdots,\gamma_{d}\right\} \\
\left\{ \gamma_{0}^{D},\cdots,\gamma_{d}^{D}\right\} 
}
\label{eq:Hom-mir-diag}
\end{equation}
In the diagram, we only focus on the free parts of $H_{3}(\check{X},\mathbb{Z})$.
We take a basis $\left\{ \gamma_{0},\cdots,\gamma_{d}\right\} $ and
also the corresponding basis $\left\{ \gamma_{0}^{D},\cdots,\gamma_{d}^{D}\right\} $
of $H^{3}(\check{X},\mathbb{Z})$. By assumption of the mirror equivalence,
we have a basis of $K_{num}(X)$ by $E_{i}=mir^{-1}(\gamma_{i})$.
As we discussed in Sect. \ref{sec:Families-of-Calabi-Yau}, we consider
$K_{num}(X)$ together with the symplectic form $\chi(-,-)$, which
is expressed by the Riemann-Roch formula as
\[
\chi(E_{i,}E_{j})=\int_{X}ch(E_{i}^{\vee})ch(E_{j})\mathrm{Todd}_{X},
\]
and write this in a matrix form $\chi=(\chi_{ij}):=(\chi(E_{i,}E_{j}))$.
For the cycles $\gamma_{i}$, we define a signed intersection number
by $(\gamma_{i},\gamma_{j}):=\int_{\check{X}}\gamma_{i}^{D}\wedge\gamma_{j}^{D}(=:\check{\chi}_{ij})$.
Due to the equivalence of the categories, where coherent sheaves are
mapped to Lagrangian submanifolds and sheaf cohomologies are mapped
to Floer homologies, we naturally have $\chi(E_{i},E_{j})=(mir(E_{i}),mir(E_{j})))$$=-(\gamma_{i},\gamma_{j})$$=-\check{\chi}_{ij}$.
Here we introduce the minus sign to have $\chi(E_{i}^{\vee},E_{j}^{\vee})=(\gamma_{i}^{D},\gamma_{j}^{D})$.
In the derived category $D(X),$ since a Fourier-Mukai functor $\Phi_{\mathcal{E}}$
representing an auto-equivalence preserves the Riemann-Roch pairing
\cite[Prop.5.44]{HuyBook}, the resulting linear action $ch(\Phi_{\mathcal{E}})$
on $K_{num}(X)$ preserves the symplectic form $\chi=(\chi(E_{i},E_{j}))$.
Corresponding to $\Phi_{\mathcal{E}}$, we have a symplectic diffeomorphism
$Mir(\Phi_{\mathcal{E}})$ and its linear action on $H_{3}(\check{X},\mathbb{Z})$
which preserves the symplectic form $\check{\chi}=(\check{\chi}_{ij})$.
We assume below that $\left\{ \gamma_{i}\right\} $ is taken to be
a symplectic basis. The following definition was made in \cite[Def.2.1]{Hos}. 
\begin{defn}
Applying homological mirror symmetry to the period map $\tilde{\mathcal{P}}(x)$
(\ref{eq:tilde-P}) of the family $\check{\mathfrak{X}}\rightarrow\mathcal{M}_{f_{\Delta^{*}}}$,
we define the \textit{central charge (map)} of mirror Calabi-Yau manifold
$X$ by a map $\mathcal{Z}:\mathcal{D}_{X}\rightarrow K_{mum}(X),$
\[
t\mapsto\mathcal{Z}_{t}=\sum_{i,j}\left(\int_{mir(E_{i})}\Omega_{x}\right)\chi^{ij}E_{j}^{\vee}
\]
with $\chi:=(\chi(E_{i},E_{j}))$ and $(\chi^{ij}):=\chi^{-1}.$ Where
$t\in\mathcal{D}_{X}$ is defined by (\ref{eq:Bryant-Griffiths})
in general and locally given by (\ref{eq:mirror-map}) on $U_{T_{o}}$. 
\end{defn}

By introducing an additive subcategory $\mathcal{P}(\phi)$ of $D(X)$
in terms of an element (called central charge) of $\mathrm{Hom}(K_{num}(X),\mathbb{C})$,
Bridgeland \cite{Bridge2} defined stability conditions on $D(X)$
and also the space of the stability conditions; these have provided
a mathematical ground to the corresponding notion of the central charges
of branes in string theory \cite{Dougl}. The element $\mathcal{Z}_{t}\in K_{num}(X)$
defines the central charge as an element $Z_{t}\in\mathrm{Hom}(K_{num}(X),\mathbb{C})$
by
\[
Z_{t}(E):=\int_{X}ch(E)ch(\mathcal{Z}_{t})\mathrm{Todd}_{X}
\]
which varies with the parameter $t$. The value $Z_{t}(E)$ is called
the central charge of $E$. As described briefly in the paragraph
below (\ref{eq:mirror-map}), the parameter $t$ is called K\"ahler
moduli and can be regarded as the parameter changing the polarization
of $X$ at least locally. 

Identifying the symplectic basis $\mathcal{B}=\left\{ b^{(0)},b_{i}^{(1)},b_{j}^{(2)},b^{(3)}\right\} $
in Proposition \ref{prop:Pi-integral-symp} with $\sum_{j}\chi^{ij}ch(E_{j}^{\vee}$),
it was conjectured in \cite[Conj.2.2]{Hos} that 
\begin{conjecture}[\textbf{Central charge formula}\footnote{The cohomology valued hypergeometric series here coincides with the
$\Gamma$-class introduced later in \cite{Iri,GGI} in the context
of quantum cohomology.}]
 \label{conj:central-charge}The cohomology valued hypergeometric
series $w_{0}^{T_{o}}(x,\frac{J}{2\pi i})$ represents the central
charge $ch(\mathcal{Z}_{t})$ on $U_{T_{o}}$, i.e., we have the following
equality
\[
w_{0}^{T_{o}}(x,\frac{J}{2\pi i})=ch(\mathcal{Z}_{t}).
\]
\end{conjecture}

It should be noted that the undetermined parameters $a_{kl}\in\mathbb{Q}$
in $\mathcal{B}$ come into the formula as a choice of a symplectic
basis $\left\{ E_{0},\cdots,E_{d}\right\} $ to express the central
charge $\mathcal{Z}_{t}$. 

\subsubsection{Mirror symmetry}

The central charge formula implicitly describes the isomorphism $mir_{D}:K_{num}(\hat{Z}_{f_{\Delta}})\rightarrow H^{3}(\hat{Z}_{f_{\Delta^{*}}},\mathbb{Z})$
which preserves the symplectic structures, $\chi(E,F)=\check{\chi}(mir_{D}(E),mir_{D}(F))$.
Let $T_{k}$ be the monodromy matrix which represents the monodromy
action $\mathcal{T}_{k}$ in Lemma \ref{lem:T-exp-J} by $(\gamma_{0}^{D}\cdots\gamma_{d}^{D})\mapsto(\gamma_{0}^{D}\cdots\gamma_{d}^{D})T_{k}$
for a basis $\left\{ \gamma_{0}^{D},\cdots,\gamma_{d}^{D}\right\} $
of $H^{3}(\hat{Z}_{f_{\Delta^{*}}},\mathbb{Z})$. In the following
proposition, we identify $K_{num}(\check{Z}_{f_{\Delta}})$ with the
image $ch(K_{num}(\check{Z}_{f_{\Delta}}))$ in $H^{even}(\hat{Z}_{f_{\Delta}},\mathbb{Q})$
as described in Subsect. \ref{subsec:A-structure}. 
\begin{prop}
\label{prop:mirror-sym}Assume that the central charge formula (Conjecture
\ref{conj:central-charge}) holds and read an isomorphism $mir_{D}:K_{num}(\hat{Z}_{f_{\Delta}})\rightarrow H^{3}(\hat{Z}_{f_{\Delta^{*}}},\mathbb{Z})$
from it. Then it holds that 
\[
N_{\lambda}=mir_{D}\circ L_{\kappa(\alpha)}\circ mir_{D}^{-1}
\]
for $N_{\lambda}=\sum_{k}\lambda_{k}\log T_{k}$ and the Lefschetz
operator $L_{\kappa(\alpha)}$ with $\kappa(\alpha)=\sum_{k}\alpha_{k}J_{k}$
$(\alpha_{k}>0)$ with $\alpha=-\lambda$. Namely, the A-structure
of $\hat{Z}_{f_{\Delta}}$ is isomorphic to the B-structure of $\hat{Z}_{f_{\Delta^{*}}}$
from the origin of $U_{T_{o}}$.
\end{prop}

\begin{proof}
Under the assumption, we write the central charge formula as
\[
w_{0}^{T_{o}}(x,\frac{J}{2\pi i})=\sum_{i,j}\Pi_{i}(x)\chi^{ij}ch(E_{j}^{\vee}).
\]
Since $\Pi_{i}(x)=\int\Omega_{x}\wedge\gamma_{i}^{D}$ represents
period integral of $\gamma_{i}^{D}=mir_{D}(E_{i}^{\vee})$, the monodromy
action $\mathcal{T}_{k}$ is represented by $\mathcal{T}_{k}\cdot(\Pi_{0}\,\cdots\,\Pi_{d})=(\Pi_{0}\,\cdots\,\Pi_{d})T_{k}$.
Also, since the monodromy matrix $T_{k}$ preserves the intersection
form $\check{\chi}(\gamma_{i}^{D},\gamma_{j}^{D})$, it satisfies
$^{t}(T_{k})^{-1}\check{\chi}(T_{k})^{-1}=\check{\chi}.$ Then by
the relation $\chi=\check{\chi}$, we have $T_{k}\chi^{-1}({}^{t}T_{k})=\chi^{-1}$
and 
\begin{equation}
\mathcal{T}_{k}w_{0}^{T_{o}}(x,\frac{J}{2\pi i})=\sum_{i,j}\Pi_{i}(x)\chi^{ij}\left(ch(E^{\vee})T_{k}^{-1}\right)_{j},\label{eq:Tk-w0-J}
\end{equation}
where $\left(ch(E^{\vee})T_{k}^{-1}\right)_{j}$ represents the $j$-th
component of $(ch(E_{0}^{\vee})\,\cdots\,ch(E_{d}^{\vee}))T_{k}^{-1}$.
By the categorical equivalence shown in (\ref{eq:Hom-mir-diag}),
there exists a homomorphism on $K_{num}(\hat{Z}_{f_{\Delta}})$ satisfying
$\mathcal{T}_{k}=mir_{D}\circ\mathbf{T}_{k}\circ mir_{D}^{-1}$. By
invariant definition of $\mathcal{Z}_{t}$, when both $\mathcal{T}_{k}$
and $\mathbf{T}_{k}$ are acted on the central charge formula $ch(\mathcal{Z}_{t})=w_{0}^{T_{o}}(x,\frac{J}{2\pi i})$,
it should be invariant. Then the equation (\ref{eq:Tk-w0-J}) shows
that $\mathbf{T}_{k}$ acts on the basis $\left\{ ch(E_{0}^{\vee}),\cdots,ch(E_{d}^{\vee})\right\} $
by the matrix $T_{k}$, i.e., $\mathbf{T}_{k}\cdot(ch(E_{0}^{\vee})\,\cdots\,ch(E_{d}^{\vee}))=(ch(E_{0}^{\vee})\,\cdots\,ch(E_{d}^{\vee}))T_{k}$.
On the other hand, we read from Lemma \ref{lem:T-exp-J} and (\ref{eq:Tk-w0-J})
that $(ch(E_{0}^{\vee})\,\cdots\,ch(E_{d}^{\vee}))T_{k}^{-1}$ represents
the linear action of $e^{J_{k}}$ on the cohomology ring. Form these,
we have a relation $\mathcal{T}_{k}=mir_{D}\circ e^{-J_{k}}\circ mir_{D}^{-1}$,
and writing this by the basis $\left\{ \gamma_{0},\cdots,\gamma_{d}\right\} $,
we obtain a matrix relation
\[
T_{k}=mir_{D}\circ e^{-J_{k}}\circ mir_{D}^{-1}
\]
where $e^{-J_{k}}$ should be understood a matrix representation of
$e^{-J_{k}}$ with respect to the basis $\left\{ ch(E_{0}^{\vee}),\cdots,ch(E_{d}^{\vee})\right\} $.
Since we define $N_{k}=\log T_{k}$, we have the corresponding matrix
relation $N_{k}=-mir_{D}\circ J_{k}\circ mir_{D}^{-1}$. The sum of
these relations with the coefficients $\lambda_{k}$ gives the claimed
relation. 
\end{proof}

\subsection{Mirror symmetry and Gromov-Witten invariants}

By Proposition \ref{prop:mirror-sym}, if we assume that the central
charge formula holds, we see that Calabi-Yau hypersurfaces $\hat{Z}_{f_{\Delta}}$
and $\hat{Z}_{f_{\Delta^{*}}}$ are mirror symmetric (Definition \ref{def:mirror-def})
for a pair of reflexive polytopes $(\Delta,\Delta^{*})$. The same
arguments also apply to pairs of Calabi-Yau complete intersections
defined by nef-partitions. For these, by using the integral and symplectic
basis of the period integrals in Proposition \ref{prop:Pi-integral-symp},
we can calculate Gromov-Witten invariants in general \cite{HKTY1,HKTY2,HLYcmp}
following the work by Candelas et al \cite{Cand}. In particular,
recent computer codes, e.g., CYtools \cite{CYtools}, provide an efficient
environment for calculating Gromov-Witten invariants for hypersurfaces
defined by the reflexive polytopes in \cite{KSlist}. In 1993, Bershadsky,
Cecotti, Ooguri, and Vafa (BCOV) \cite{BCOV1,BCOV2} extended the
work \cite{Cand} to higher genus Gromov-Witten invariants. Calculating
higher genus Gromov-Witten invariants requires tedious calculations
compared to the genus-zero calculations. Gromov-Witten theory and
quantum cohomology are among the subjects which have developed greatly
from the study of mirror symmetry; however, they are not the scope
of this survey article. 

~

~

\section{\protect\label{sec:Birat-and-FM}Birational geometry and Fourier-Mukai
partners from mirror families}

Assume that a Calabi-Yau manifold $X$ is a mirror of $\check{X}$
(and vice versa). The B-structure of $\check{X}$ is a local property
of a deformation family $\check{\mathfrak{X}}\rightarrow\mathcal{M}_{\check{X}}$
of $\check{X}$ specified by a special boundary point (LCSL point).
We often encounter the cases where we find many LCSLs on the moduli
space $\mathcal{M}_{\check{X}}$, and come to the following general
picture: 
\[
\begin{aligned} & \text{LCSL boundary points in }\mathcal{M}_{\check{X}}\\
 & \qquad\leftrightarrow\quad\text{birational geometry and/or Fourier-Mukai partners of }X.
\end{aligned}
\]
For a family of Calabi-Yau hypersurfaces $\text{\ensuremath{\left\{  \hat{Z}_{f_{\Delta^{*}}}(a)\right\} } }$
over $\mathcal{M}_{f_{\Delta^{*}}}=\mathbb{P}_{\mathrm{Sec}\Delta^{*}}$,
we have LCSL points at least as many as the maximal triangulations
of $\Delta^{*}$. Different triangulations corresponds to different
resolutions of the ambient toric variety $\hat{\mathbb{P}}_{\Delta}$.
Interesting examples arise in mirror symmetry of complete intersection
Calabi-Yau manifolds. Moreover, it turns out that such examples of
Calabi-Yau manifolds have nice descriptions in classical projective
geometry. 

\subsection{\protect\label{subsec:Birat-P4P4}Birational geometry related to
determinantal quintics in $\mathbb{P}^{4}$}

Our first example is a general complete intersection of five $(1,1)$
divisors in $\mathbb{P}^{4}\times\mathbb{P}^{4}$, 
\begin{equation}
X=\cap_{i=1}^{5}(1,1)\subset\mathbb{P}^{4}\times\mathbb{P}^{4}.\label{eq:P4.P4}
\end{equation}
$X$ is a smooth Calabi-Yau manifold when the five $(1,1)$ divisors
are general. 

\subsubsection{Batyrev-Borisov construction}

We can describe the Calabi-Yau complete intersection $X$ above in
terms of the product $\Delta=\Delta_{1}\times\Delta_{2}\subset\mathbb{R}^{4}\times\mathbb{R}^{4}$
of reflexive polytopes $\Delta_{1},\Delta_{2}\subset\mathbb{R}^{4}$
which describe a projective space $\mathbb{P}^{4}$ by $\mathbb{P}_{\Delta_{i}}=\mathbb{P}^{4}$.
Let $\bm{e}_{1},\cdots,\bm{e}_{4}$ be the unit vectors of $\mathbb{R}^{4}$
and $\check{\e}_{1},\cdots,\check{\e}_{4}$ be the dual basis of $\check{\mathbb{R}^{4}}$
satisfying $\langle\check{\e}_{i},\e_{j}\rangle=\delta_{ij}$. Then
the dual polytope $(\Delta_{1}\times\Delta_{2})^{*}\subset\check{\mathbb{R}}^{4}\times\check{\mathbb{R}}^{4}$
of $\Delta_{1}\times\Delta_{2}$ is given by $\mathrm{Conv}\left\{ \check{\e}_{1}\times0,\cdots,\check{\e}_{5}\times0,0\times\check{\e}_{1},\cdots,0\times\check{\e}_{5}\right\} $,
where $\check{\e}_{5}:=-\check{\e}-\cdots-\check{\e}_{5}$ and each
vertex, in order, represents the torus invariant divisors $\left\{ z_{i}=0\right\} $
$(i=1,\cdots,5)$ and $\left\{ w_{i}=0\right\} $ $(i=1,\cdots,5)$
on the product $\mathbb{P}_{z}^{4}\times\mathbb{P}_{w}^{4}$. A nef
partition (\ref{eq:nef-part-D}) determines a splitting of the anti-canonical
bundle $-K_{\mathbb{P}_{\Delta}}$ in general. The splitting for $X=\cap_{i=1}^{5}(1,1)$
may be written as $-K_{\mathbb{P}_{\Delta}}=(D_{z_{1}}+D_{w_{1}})+\cdots+(D_{z_{5}}+D_{w_{5}})$
where we write a divisor by $D_{x}:=\left\{ x=0\right\} $. The general
section of the line bundle $\mathcal{L}_{i}=\mathcal{O}_{\mathbb{P}_{\Delta}}(D_{x_{i}}+D_{w_{i}})$
is represented by $(1,1)$ in (\ref{eq:P4.P4}), and by $f_{\Delta_{i}}$
in the toric description. As sketched in Subsect \ref{subsec:CICY},
we can read from this the corresponding decomposition $\varphi=\varphi_{1}+\cdots+\varphi_{5}$
of the convex piecewise linear function of $-K_{\mathbb{P}_{\Delta}},$
and this decomposition determines the dual nef partition $\nabla=\nabla_{1}+\cdots+\nabla_{5}$
with $\nabla_{i}=\mathrm{Conv.}\left(\left\{ 0\times0,\check{\e}_{i}\times0,0\times\check{\e}_{i}\right\} \right)$
$(i=1,\cdots,5)$. These $\nabla_{i}$ determine Laurent polynomials
$f_{\nabla_{i}}=f_{\nabla_{i}}(a)$ in toric variety $\mathbb{P}_{\nabla}$.
To summarize, by Batyrev-Borisov construction, we have mirror Calabi-Yau
manifold $\check{X}$ of $X$ as 
\[
\check{X}=\left\{ f_{\nabla_{1}}(a)=\cdots=f_{\nabla_{5}}(a)=0\right\} \subset\hat{\mathbb{P}}_{\nabla}.
\]
By using the Lefschetz hyperplane theorem, and calculating Euler number,
we determine the Hodge numbers of $X$ as $h^{1,1}(X)=2,h^{2,1}(X)=52$.
For the mirror $\check{X}$, we can use the general formula in \cite{BB2}
or \cite{PALP}. 

Now, applying the construction in Subsect. \ref{subsec:CICY}, we
have a family $\check{\mathfrak{X}}\rightarrow\mathcal{M}_{f_{1}\cdots f_{5}}$
over $\mathcal{M}_{f_{1}\cdots f_{5}}=\mathbb{P}_{\mathrm{Sec}\mathcal{A}}$.
The secondary polytope is determined by the data $\mathcal{A}$ in
(\ref{eq:A-CICY}) where we define $\bar{\nu}_{i}^{(k)}(0\leq i\leq3,1\leq k\leq5)$
in (\ref{eq:bar-vect-A}) by $\nu_{0}^{(k)}=0\times0,$ $\nu_{1}^{(k)}=\check{\e}_{k}\times0,$
$\nu_{2}^{(k)}=0\times\check{\e}_{k}$$(1\leq k\leq5)$. We calculate
the triangulations of $\mathcal{A}$, by TOPCOM \cite{TOPCOM} or
CYTools \cite{CYtools}, to see that there are three regular triangulations
$T_{1},T_{2},T_{3}$ and $\mathbb{P}_{\mathrm{Sec}\mathcal{A}}=\mathbb{P}^{2}$.
One triangulation, say $T_{1}$, is a maximal triangulation (see the
paragraph below Proposition \ref{prop:GKZ-Pi-CICY}) and denote the
corresponding boundary point by $o_{T_{1}}$. From $o_{T_{1}}$, we
obtain the $B$-structure of $\check{X}$ which is isomorphic to the
$A$-structure of $X$. Interestingly, the other two boundary points
$o_{T_{2}}$ and $o_{T_{3}}$ also define the $B$-structures which
we can identify with the $A$-structures of birational models of $X$
which we describe below. The appearance of these birational models
can be understood naturally if we describe $X$ in terms of the so-called
Landau-Ginzburg model (or the gauged sigma model \cite{HoriT,HoriKapp})
on the GIT quotient represented by the first line of (\ref{eq:Gale-daig}).
We leave this for the reader's exercise. 

Below, following \cite{HTcmp,HTmov}, we describe briefly the Calabi-Yau
manifold $X(=:X_{1})$ by the language of classical projective geometry
to see its birational models $X_{2},X_{3}$. We also construct mirror
family $\check{\mathfrak{X}}\rightarrow\mathbb{P}^{2}$ in classical
projective geometry to have global picture of the family. 

\subsubsection{Projective geometry of $X$}

We write the equations of general $(1,1)$ divisors on $\mathbb{P}_{z}^{4}\times\mathbb{P}_{w}^{4}$
by $5\times5$ matrices $A_{k}=(a_{ij}^{(k)})$ as $f_{k}=\sum_{i,j}a_{ij}^{(k)}z_{i}w_{j}$
$(1\leq k\leq5)$. We may arrange the five equations $f_{1}=\cdots=f_{5}=0$
into a matrix form
\begin{equation}
A_{z}w=0\label{eq:Az-w-0}
\end{equation}
by defining $A_{z}=\left(\sum_{i}z_{i}a_{ij}^{(k)}\right)_{1\leq k,j\leq5}$.
Since we have a solution $w\not=0$, we must have $\det A_{z}=0$.
When $z$ is general, i.e., $\mathrm{rk}A_{z}=4$, the equation (\ref{eq:Az-w-0})
determines a unique point $[w]\in\mathbb{P}_{w}^{4}$ by Cramer's
formula. Since the co-rank $s$ locus, i.e., $\left\{ [z]\in\mathbb{P}_{z}^{4}\mid5-\mathrm{rk}A_{z}=s\right\} $,
has codimension $s^{2}-1^{2}$ in $\left\{ \det A_{z}=0\right\} \subset\mathbb{P}_{z}^{4}$,
the co-rank two locus appears as a set of points. The number of these
points are calculated by Porteous' formula \cite{Port}, and turns
out to be $50$. Over these $50$ points, the equation $A_{z}w=0$
determines a line $\ell_{z}\simeq\mathbb{P}^{1}$. Thus, writing the
projection to the first factor by $\pi_{11}:\mathbb{P}_{z}^{4}\times\mathbb{P}_{w}^{4}\rightarrow\mathbb{P}_{z}^{4}$,
we obtain a contraction 
\[
\pi_{11}:X_{1}\rightarrow Z_{1}:=\left\{ [z]\mid\det A_{z}=0\right\} \subset\mathbb{P}_{z}^{4}
\]
of $50$ lines to a determinantal quintic $Z_{1}$. Clearly, by symmetry,
we have also a contraction $\pi_{21}:X_{1}\rightarrow Z_{2}$ to a
determinantal quintic $Z_{2}\subset\mathbb{P}_{z}^{4}$. We may say
that $X_{1}$ provides small resolutions of $Z_{1}$ and $Z_{2}$.
See the diagram (\ref{eq:Birat-diagram}). 

Let us focus on the small resolution $\pi_{11}:X_{1}\rightarrow Z_{1}$,
and note that this resolution is made by considering the right kernel
of the matrix $A_{z}$. Since $A_{z}\not=^{t}A_{z}$ in general, we
will obtain another small resolution $X_{3}$=$\left\{ ([\lambda],[z])\mid^{t}\lambda A_{z}=0\right\} \subset\mathbb{P}_{\lambda}^{4}\times\mathbb{P}_{z}^{4}$
by using the left kernel of the matrix $A_{z}$. By symmetry again,
we obtain $X_{2}$=$\left\{ ([\lambda],[w])\mid^{t}\lambda A_{w}=0\right\} \subset\mathbb{P}_{\lambda}^{4}\times\mathbb{P}_{w}^{4}$.
We note that $X_{1}$, $X_{2}$, $X_{3}$ are all complete intersections
of the same form of (\ref{eq:Az-w-0}) but different equations. They
are birational by the flops of $50\,\mathbb{P}^{1}$s, $\varphi_{ji}:X_{i}\dashrightarrow X_{j}$
for each $i,j$.

As shown in the diagram (\ref{eq:Birat-diagram}), from $X_{2}$ and
$X_{3}$ we have a contraction to $Z_{3}:=\left\{ \det A_{\lambda}=0\right\} \subset\mathbb{P}_{\lambda}^{4}$
with $A_{\lambda}:=\sum_{k}\lambda_{k}a_{ij}^{(k)}$, and the birational
maps make a loop $\rho:=\varphi_{13}\circ\varphi_{32}\circ\varphi_{21}:X_{1}\dashrightarrow X_{1}$.
It was shown \cite{HTmov} that this generate the group of birational
automorphisms, $\mathrm{Bir(X_{1})=\langle\rho,\mathrm{Aut(X_{1})\rangle}}$,
which has an infinite order. \def\XXX{
\begin{xy} 
(-26,0)*++{X_{1}}="Xi", 
(0,0)*++{X_{2}}="Xii", 
(26,0)*++{X_{3}}="Xiii", 
(52,0)*++{X_{1}}="XiR", 
(-13,-13)*+{Z_{2}}="Zii", 
(13,-13)*+{Z_{3}}="Ziii", 
(39,-13)*+{Z_{1}}="ZiR", 
\ar_{\pi_{21}} "Xi";"Zii" 
\ar^{\pi_{22}} "Xii";"Zii" 
\ar_{\pi_{32}} "Xii";"Ziii" 
\ar^{\pi_{33}} "Xiii";"Ziii" 
\ar_{\pi_{13}} "Xiii";"ZiR" 
\ar^{\pi_{11}} "XiR";"ZiR" 
\ar@{-->}^{\varphi_{21}} "Xi";"Xii" 
\ar@{-->}^{\varphi _{32}} "Xii";"Xiii" 
\ar@{-->}^{\varphi_{13}} "Xiii";"XiR" 
\end{xy} 
} 
\begin{equation}
\begin{matrix}\XXX\end{matrix}\label{eq:Birat-diagram}
\end{equation}

\begin{rem}
\label{rem:Reye-cong} The above Calabi-Yau manifold $X=\cap_{i=1}^{5}(1,1)$
appears when we study the three dimensional Reye congruence in \cite{HTreye}.
Historically Reye congruence stands for an Enriques surface \cite{Co},
which we can describe as a $\mathbb{Z}_{2}$ quotient of complete
intersection of four symmetric $(1,1)$ divisors in $\mathbb{P}^{3}\times\mathbb{P}^{3}$.
Three dimensional Reye congruence is a generalization of it, which
we describe by a $\mathbb{Z}_{2}$ quotient of $X_{s}=\cap_{i=1}^{5}(1,1)_{s}$
with five symmetric divisors $(1,1)_{s}$. For symmetric divisors,
the determinantal quintic $Z_{3}=\left\{ \det A_{\lambda}=0\right\} $
in the diagram (\ref{eq:Birat-diagram}) determines the so-called
symmetroid which is singular along a smooth curve of genus 26 and
degree 20 where we have $\mathrm{rk}A_{\lambda}=3$ \cite[Prop.3.8]{HTreye}.
It was shown in \cite{HTreyeDx} that the double cover of $Z_{3}$
branched along the singular locus defines a Calabi-Yau manifold which
is a Fourier-Mukai partner of the three dimensional Reye congruence
$X_{s}/\mathbb{Z}_{2}$. 
\end{rem}

\subsection{Mirror family $\check{\mathfrak{X}}\rightarrow\mathbb{P}^{2}$}

The birational geometry above may be found by studying the Landau-Ginzburg
theory in terms of the GIT quotient. On the other hand, no canonical
way is known to rewrite the toric mirror constructions due to Batyrev-Borisov
in terms of classical projective geometry. In the present case, however,
it was found in \cite{HTcmp} that we can obtain a mirror Calabi-Yau
manifold $\check{X}$ from a special family $X_{sp}$ of $X$ with
parameters $a,b$; 
\[
X_{sp}:=\left\{ z_{i}w_{i}+a\,z_{i}w_{i+1}+b\,z_{i+1}w_{i}=0\;(i=1,..,5)\right\} \subset\mathbb{P}_{z}^{4}\times\mathbb{P}_{w}^{4}.
\]
The special family $X_{sp}$ is singular for general parameters $a,b$,
but there exists a crepant ($c_{1}=0$) resolution of the singularity
which gives rise to a mirror Calabi-Yau manifold \cite[Thm.5.11,Thm.5.17]{HTcmp}. 
\begin{prop}
When $a,b\in\mathbb{C}$ are general, the following properties hold:

\begin{myitem}

\item[$(1)$] $X_{sp}$ is singular along 20 lines of $\mathbb{P}^{1}$with
the $A_{1}$-singularities.

\item[$(2)$] There exists a crepant resolution $\check{X}\rightarrow X_{sp}$
which gives Calabi-Yau manifolds with the Hodge numbers $h^{1,1}(\check{X})=52$,
$h^{2,1}(\check{X})=2$.

\end{myitem} 
\end{prop}

Taking the symmetry of the defining equations of $X_{sp}$ into account,
we see that the special family is parametrized by $a^{5}$ and $b^{5}$.
Moreover, these affine parameters are projectivised naturally by $[a^{5},b^{5},1]\in\mathbb{P}^{2}$.
The family parametrized by $[a^{5},b^{5},1]$ has been studied in
detail in \cite[Prop.3.11]{HTmov}. 
\begin{prop}
We have a family of Calabi-Yau manifolds $\check{\mathfrak{X}}\rightarrow\mathcal{M}_{\check{X}}^{0}\subset\mathbb{P}^{2}$
with the fibers being crepant resolutions $\check{X}\rightarrow X_{sp}$.
The closure of the parameter space $\mathcal{M}_{\check{X}}^{0}$
is given by $\mathcal{M}_{\check{X}}=\mathbb{P}^{2}$, and the fibers
over $\mathcal{M}_{\check{X}}^{0}$ extend to over $\mathcal{M}_{\check{X}}\setminus\mathcal{M}_{\check{X}}^{0}=D_{1}\cup D_{2}\cup D_{3}\cup Dis_{0}$
with having degenerations. Here $D_{1},D_{2},D_{3}$ are the toric
divisors on $\mathbb{P}^{2}$ and $Dis_{0}$ represents the discriminant
of $X_{sp}$ which has degree 5 (see (\ref{eq:dis0})). 
\end{prop}

From the form of the equation, or making some connections to the toric
mirror construction, we see that the point $[0,0,1]$ is the degeneration
(LCSL) point (corresponding to the maximal triangulation $T_{1}$
of $\mathcal{A}$) where we have an isomorphism of the $B$-structure
of the family to the $A$-structure of $X_{1}$. Thus we obtain a
mirror family $\check{\mathfrak{X}}\rightarrow\mathcal{M}_{\check{X}}$
of $X_{1}$ in terms of classical projective geometry. By the obvious
symmetry of the family, we also observe mirror symmetry to the birational
models $X_{2},X_{3}$ as follows \cite[Prop.3.12]{HTmov}. 
\begin{prop}
\label{prop:Reye-ok}The boundary points $o_{1}=D_{1}\cap D_{2},o_{2}=D_{2}\cap D_{3},o_{3}=D_{3}\cap D_{1}$
are all LCSLs, and the B-structures from each point of $o_{1},o_{2},o_{3}$
are isomorphic, respectively, to the A-structures of $X=:X_{1}$ and
its birational models $X_{2},X_{3}$. 
\end{prop}

In the definition of mirror symmetry (Definition \ref{def:mirror-def}),
we focus on two different cones with similar properties; one is the
K\"ahler cone to define $A$-structure, the other is the monodromy
nilpotent cone to define $B$-structure. If we have birational models,
K\"ahler cones are naturally glued together to make a larger cone,
which is called movable cone, see e.g. \cite{KawC1,KawC2}, \cite{MoC}.
Correspondingly, we have a gluing of monodromy nilpotent cones. As
we sketch in the following subsection, the identifications of the
$B$-structures in Proposition \ref{prop:Reye-ok} are justified by
studying this gluing property. 

\subsection{Movable vs monodromy nilpotent cones}

We show that the monodromy nilpotent cones from the boundary points
$o_{i}$ in Proposition \ref{prop:Reye-ok} are naturally glued through
monodromy relations. 

\subsubsection{Movable cone of $X_{1}$}

Here we briefly summarize the movable cone $\mathrm{Mov}(X_{1})$.
We first note, from the diagram (\ref{eq:Birat-diagram}), that the
K\"ahler cone $\mathcal{K}_{X_{1}}$ is given by 
\[
\mathcal{K}_{X_{1}}=\mathbb{R}_{>0}H_{1}+\mathbb{R}_{>0}H_{2}\,\subset H^{2}(X_{1},\mathbb{R})
\]
with $H_{i}=\pi_{i1}^{*}H_{Z_{i}}$, where $H_{Z_{i}}$ represent
the hyperplane classes of $Z_{i}\subset\mathbb{P}^{4}$. Similarly,
we write 
\[
\mathcal{K}_{X_{2}}=\mathbb{R}_{>0}L_{Z_{2}}+\mathbb{R}_{>0}L_{Z_{3}},\,\,\,\,\,\,\,\mathcal{K}_{X_{3}}=\mathbb{R}_{>0}M_{Z_{3}}+\mathbb{R}_{>0}M_{Z_{1}}
\]
with the pullbacks of the hyperplane classes $L_{Z_{i}}=\pi_{i2}^{*}H_{Z_{i}}$
and $M_{Z_{j}}=\pi_{j3}^{*}H_{Z_{j}}$. 
\begin{prop}
In terms of the birational maps $\varphi_{21}:X_{1}\dashrightarrow X_{2}$
and $\varphi_{31}:X_{1}\dashrightarrow X_{3}$, respectively, we have
the following cones in $H^{2}(X_{1},\mathbb{R})$:
\[
\begin{aligned}\varphi_{21}^{*}(\mathcal{K}_{X_{2}}) & =\mathbb{R}_{>0}H_{2}+\mathbb{R}_{>0}(4H_{2}-H_{1})\\
\varphi_{31}^{*}(\mathcal{K}_{X_{3}}) & =\mathbb{R}_{>0}(4H_{1}-H_{2})+\mathbb{R}_{>0}H_{1}
\end{aligned}
.
\]
\end{prop}

\begin{proof}
These are derived from the following relations:
\begin{equation}
\varphi_{21}^{*}(L_{Z_{2}})=H_{2}\,,\,\,\,\,\,\,\varphi_{21}^{*}(L_{Z_{3}})=4H_{2}-H_{1},\label{eq:vphi21}
\end{equation}
and 
\begin{equation}
\varphi_{31}^{*}(M_{Z_{3}})=4H_{1}-H_{2}\,,\,\,\,\,\,\,\varphi_{31}^{*}(M_{1})=H_{1}.\label{eq:vphi31}
\end{equation}
 Proof of these relations can be found in \cite[Appendix A]{HTmov}. 
\end{proof}
From the above proposition, we see that K\"ahler cones $\mathcal{K}_{X_{2}}$
and $\mathcal{K}_{X_{3}}$ are moved to the next of the K\"ahler
cone $\mathcal{K}_{X_{1}}$ making a larger cone 
\begin{equation}
C_{123}=\varphi_{21}^{*}(\mathcal{K}_{X_{2}})\cup\overline{\mathcal{K}_{X_{1}}}\cup\varphi_{31}^{*}(\mathcal{K}_{X_{3}})\,\,\subset H^{2}(X_{1},\mathbb{R}).\label{eq:C123}
\end{equation}
Using the birational automorphism $\rho:=\varphi_{12}\circ\varphi_{23}\circ\varphi_{31}:X_{1}\dashrightarrow X_{1}$
of infinite order, and making the union $\cup_{n\in\mathbb{Z}}(\rho^{*})^{n}\,\overline{C}_{123}$,
we obtain the movable cone $\mathrm{Mov(X_{1})}$, 
\[
\mathrm{Mov}(X_{1})=\mathbb{R}_{>0}(-H_{1}+(2+\sqrt{3})H_{1})+\mathbb{R}_{>0}(H_{1}+(-2+\sqrt{3})H_{2}),
\]
which we draw schematically in Fig.\ref{fig:fig.2}. We refer \cite[Prop.3.10]{HTmov}
for more details. 
\begin{center}
\begin{figure}
\begin{centering}
\includegraphics[scale=0.3]{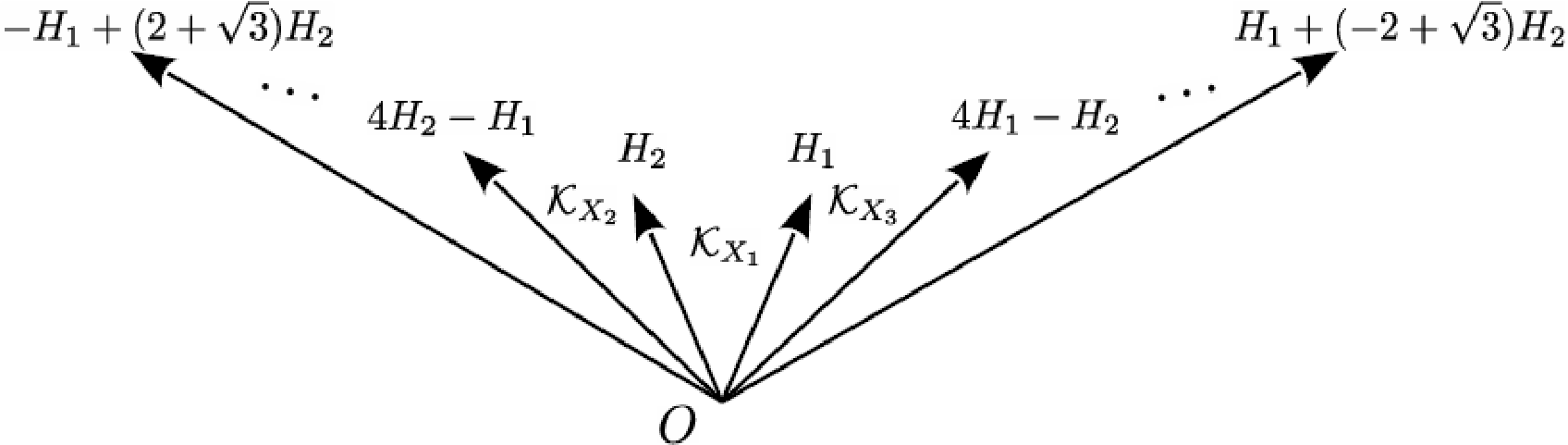}
\par\end{centering}
\caption{Fig.2. \protect\label{fig:fig.2}The movable cone $\mathrm{Mov}(X_{1})$
in $H^{2}(X_{1},\mathbb{R})$}

\end{figure}
\par\end{center}

\subsubsection{Gluing monodromy nilpotent cones}

Each of the toric boundary points $o_{1},o_{2},$ $o_{3}\in\mathbb{P}^{2}$
of the mirror family $\check{\mathfrak{X}}\rightarrow\mathcal{M}_{\check{X}}=\mathbb{P}^{2}$
defines a $B$-structure which we identify with the respective $A$-structure
of the birational models $X_{1}$,$X_{2}$ and $X_{3}$. For each
$k$, we fix a base point $b_{o_{k}}$ near the origin of the affine
chart $U_{o_{k}}$. We denote by $(x,y)$ the affine coordinate of
$U_{o_{1}}$. Then following the sign convention used in \cite{HTcmp,HTmov},
we determine coordinates of $U_{o_{2}}$ and $U_{o_{3}}$ by $(x',y')=(\frac{1}{x},-\frac{y}{x})$
and $(x'',y'')=(\frac{1}{y},-\frac{x}{y})$, respectively. We define
the local monodromy representation by $T_{x},T_{y}$ as the monodromy
around the divisors $D_{x}=\left\{ x=0\right\} $ and $D_{y}=\left\{ y=0\right\} $,
respectively. We do not specify explicitly, but we choose a natural
integral and symplectic basis $\mathcal{B}_{o_{1}}$ (cf. eq.(\ref{eq:w0-J-symplectic}))
and determine the matrices $T_{x},T_{y}$ explicitly\footnote{The matrix $T_{y}$ in \cite[Prop.6.6]{HTcmp} and \cite[(4.1.a)]{HTmov}should
be read $T_{y}=\left(\begin{smallmatrix}1 & 0 & 0 & 0 & 0 & 0\\
0 & 1 & 0 & 0 & 0 & 0\\
1 & 0 & 1 & 0 & 0 & 0\\
2 & 10 & 5 & 1 & 0 & 0\\
5 & 10 & 10 & 0 & 1 & 0\\
-5 & -5 & -3 & -1 & 0 & 1
\end{smallmatrix}\right)$. The author would like to thank Dr. Haohua Deng for notifying him
this typo. }. Using these unipotent matrices, we define the monodromy nilpotent
cone at $o_{1}$ by 
\[
\Sigma_{o_{1}}:=\mathbb{R}_{>0}N_{1}+\mathbb{R}_{>0}N_{2}
\]
with $N_{1}=\log T_{x}$ and $N_{2}=\log T_{y}$.

In a similar way, we determine $T_{x'}',T_{y'}'$ and $T_{x''}'',T_{y''}''$
independently for the data $(U_{o_{2}},b_{o_{2}},\mathcal{B}_{o_{2}})$
and $(U_{o_{3}},b_{o_{3}},\mathcal{B}_{o_{3}})$. We define $\tilde{N}_{1}'=\log T_{x'}',\tilde{N}_{2}'=\log T_{y'}'$
and $\tilde{N}_{1}''=\log T_{x''}'',\tilde{N}_{2}''=\log T_{y''}''$,
and using these data, we determine monodromy nilpotent cones at $o_{2}$
and $o_{3}$, respectively, by 
\[
\Sigma_{o_{2}}'=\mathbb{R}_{>0}\tilde{N_{1}'}+\mathbb{R}_{>0}\tilde{N}_{2}',\,\,\,\,\,\,\Sigma_{o_{3}}''=\mathbb{R}_{>0}\tilde{N}_{1}''+\mathbb{R}_{>0}\tilde{N}_{2}''.
\]
It was found in \cite{HTmov} that $\Sigma_{o_{1}},\Sigma_{o_{2}}'$
and $\Sigma_{o_{3}}''$ correspond to the K\"ahler cones $\mathcal{K}_{X_{1}},\mathcal{K}_{X_{2}}$
and $\mathcal{K}_{X_{3}}$. In fact, if write by $\varphi_{b_{o_{2}}b_{o_{1}}}$
the connection matrix relating the local data $(b_{o_{1}},\mathcal{B}_{o_{1}})$
to $(b_{o_{2}},\mathcal{B}_{o_{2}})$ along the line $\overline{b_{o_{1}}b_{o_{2}}}$
and similarly for $\varphi_{b_{o_{3}}b_{o_{1}}}$, corresponding to
$C_{123}$ in (\ref{eq:C123}), we naturally have 
\[
\check{C}_{123}=(\varphi_{b_{2}b_{o_{1}}})^{-1}\Sigma'_{o_{2}}\varphi_{b_{2}b_{o_{1}}}\,\cup\,\overline{\Sigma}_{o_{1}}\,\cup\,(\varphi_{b_{3}b_{o_{1}}})^{-1}\Sigma''_{o_{3}}\varphi_{b_{3}b_{o_{1}}}\,\,\subset\mathrm{End}(H^{3}(X_{1},\mathbb{Q})).
\]
It is easy to deduce that $\check{C}_{123}$ is glued by $\check{\rho}=\varphi_{b_{o_{1}}b_{o_{3}}}\circ\varphi_{b_{o_{3}}b_{o_{2}}}\circ\varphi_{b_{o_{2}}b_{o_{1}}}$
to make a larger cone (see \cite[Prop.4.11]{HTmov} for details).
We claim that this larger cone is the mirror analogue of the movable
cone $\mathrm{Mov}(X_{1})$. 

To see the correspondence more precisely, we note that the complement
$\mathcal{M}_{\check{X}}^{0}=\mathcal{M}_{\check{X}}\setminus\left\{ \text{discriminant loci}\right\} $
has the form as shown in Fig.~\ref{fig:fig.3}, where the discriminant
loci consist of three coordinate lines and one component $Dis_{0}=\left\{ dis_{0}=0\right\} $
with 
\begin{equation}
dis_{0}=(1-x-y)^{5}-5^{4}xy(1-x-y)^{2}+5^{5}xy(xy-x-y).\label{eq:dis0}
\end{equation}
The last component contact each of the coordinate line with 5th order,
and we blow up the singularities as shown in Fig.~\ref{fig:fig.3}
introducing exceptional divisors $E_{1},E_{1}'$ and $E_{1}''$. We
denote by $T_{E_{1}}^{L},T_{E_{1}'}^{L}$ and $T_{E_{1}''}^{L}$ the
local monodromy around these divisors.

The following proposition justifies the claim above, also shows that
the larger cone thus obtained contains a quantum correction coming
from flopping curves. 
\begin{center}
\begin{figure}
\begin{centering}
\includegraphics[scale=0.35]{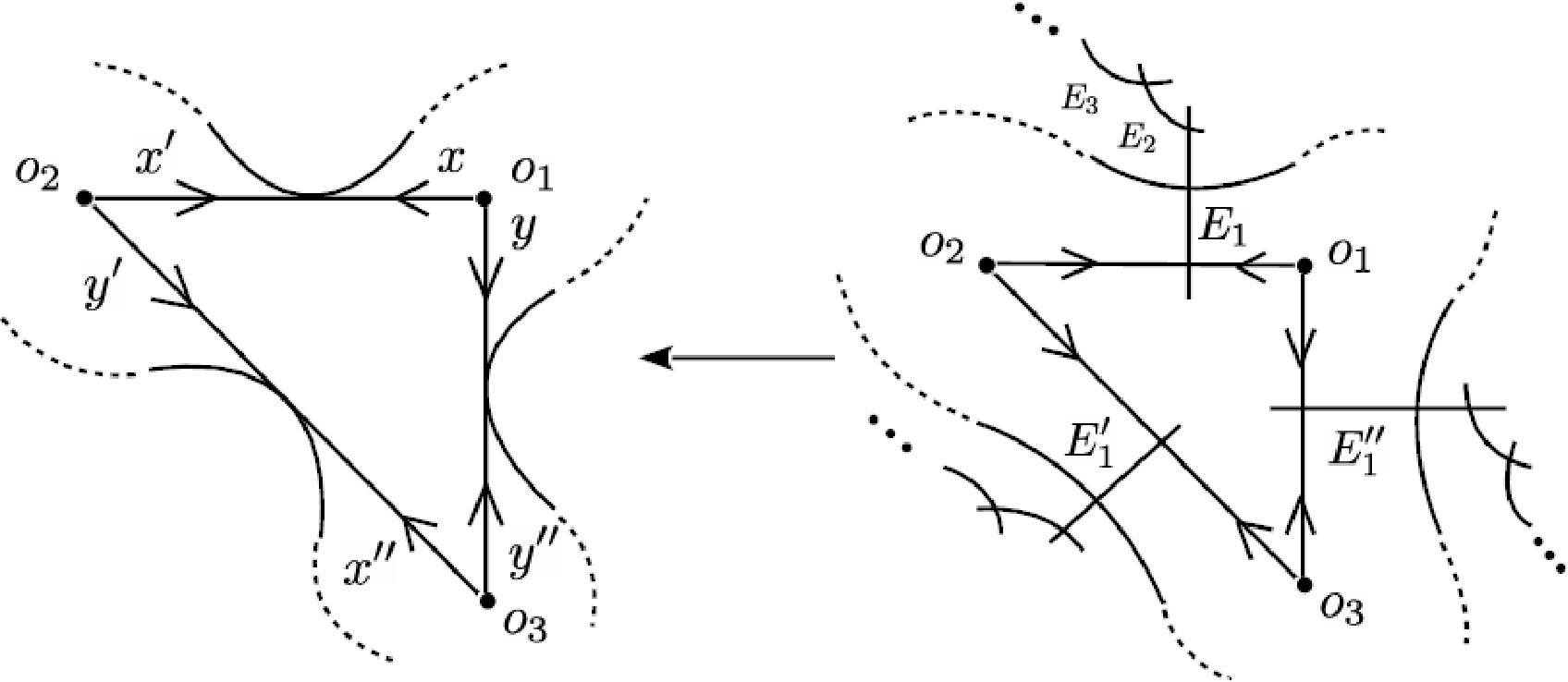}
\par\end{centering}
\caption{Fig.3. \protect\label{fig:fig.3}Blowing-up the singularities in the
discriminant loci in $\mathcal{M}_{\check{X}}$. By the blow-up coordinate
$(s_{1},s_{2})=(x-1,\frac{y}{(1-x)^{5}})$, the exceptional divisor
$E_{1}$ is described by $\left\{ s_{1}=0\right\} $. }
\end{figure}
\par\end{center}
\begin{prop}
By the connection matrix $\varphi_{b_{o_{2}}b_{o_{1}}}$, we transform
the local monodromy matrices to matrices with respect to a common
local data $(b_{o_{1}},\mathcal{B}_{o_{1}})$, 
\[
T_{x'}:=\varphi_{b_{o_{2}}b_{o_{1}}}^{-1}T_{x'}'\varphi_{b_{o_{2}}b_{o_{1}}},\,\,\,\,T_{y'}:=\varphi_{b_{o_{2}}b_{o_{1}}}^{-1}T_{y'}'\varphi_{b_{o_{2}}b_{o_{1}}}.
\]
Similarly, we transform the local monodromy $T_{E_{1}}^{L}$ to $T_{E_{1}}:=\varphi_{p_{12}b_{o_{1}}}^{-1}T_{E_{1}}^{L}\varphi_{p_{12}b_{o_{1}}}$
with the connection matrix $\varphi_{p_{12}b_{o_{1}}}$ along $\overline{b_{o_{1}}p_{12}}$
($p_{12}:=E_{1}\cap\overline{o_{1}o_{2}}$). Then the following relation
holds:
\begin{equation}
T_{x'}=T_{E_{1}}^{-1}T_{x}^{-1}T_{y}^{4}\,,\,\,\,\,\,\,T_{y'}=T_{y}.\label{eq:monod-rel}
\end{equation}
\end{prop}

\begin{proof}
A direct proof was given by calculating local monodromy matrices and
connection matrix explicitly in \cite[Prop.4.4,4.6,4.8]{HTmov}. Here
we present a more conceptional proof: We first note that the local
monodromy matrix $T_{x'}'$ is defined by the monodromy around $\left\{ x'=\frac{1}{x}=0\right\} $
with the value $y'=-\frac{y}{x}$ being fixed. Similarly, the matrix
$T_{E_{1}}^{L}$ is defined by the monodromy around $\left\{ s_{1}=x-1=0\right\} $
with the value $s_{2}=\frac{y}{(1-x)^{5}}$ being fixed. On the other
hand, there is a monodromy relation, $\mathcal{T}_{x}\mathcal{T}_{E_{1}}\mathcal{T}_{x'}=\mathrm{id}$,
associated to the line $\overline{o_{1}o_{2}}$ where the value of
$y$ is fixed to be $\varepsilon\,(|\varepsilon|\ll1)$. We note that
the additional conditions for the local monodromy, e.g., the condition
$s_{2}=\frac{y}{(1-x)^{5}}$ being fixed, can be represented by suitable
twists by $T_{y}$. Then, from the relation $\mathcal{T}_{x}\mathcal{T}_{E_{1}}\mathcal{T}_{x'}=\mathrm{id},$we
obtain a relation
\[
T_{x}\big(T_{E_{1}}(T_{y}^{-1})^{5}\big)(T_{x'}T_{y})=1.
\]
Since the divisors are normal crossing, $T_{y}$ commutes with $T_{x},T_{x'}$
and $T_{E_{1}}$, and we obtain the first relation in (\ref{eq:monod-rel}).
The second relation is easy.
\end{proof}
From the line $\overline{o_{1}o_{3}}$, we obtain similar relations,
\begin{equation}
T_{x''}=T_{x}\,,\,\,\,\,\,T_{y''}=T_{E_{1}''}^{-1}\,T_{y}^{-1}\,T_{x}{}^{4}\,.\label{eq:monod-rel2}
\end{equation}
In \cite[Prop.4.6, Rem.4.7]{HTmov}, the monodromy matrix $T_{E_{1}}$
was calculated to be 
\[
T_{E_{1}}=\left(\begin{smallmatrix}1 & 0 &  & \cdots &  & 0\\
0 & 1\\
 & 0 & 1 &  &  & :\\
: & 0 & 0 & 1\\
 & \text{-}50 & 0 & 0 & 1 & 0\\
0 &  & \cdots &  & 0 & 1
\end{smallmatrix}\right)
\]
in terms of the symplectic basis (\ref{eq:canonial-from-Pi}). The
number $50$ was identified with the flopping curve in $X_{1}\dashrightarrow X_{2}$,
and reading $T_{E_{1}}$ as the action $\begin{cases}
A_{1}\rightarrow A_{1}\\
B_{1}\rightarrow B_{1}-50\,A_{1}
\end{cases}$ on the homology cycles $\left\{ A_{0},A_{1},A_{2},B_{2},B_{1},B_{0}\right\} $,
the monodromy $T_{E_{1}}$ was called ``Picard-Lefschetz formula
for flopping curves''. 
\begin{prop}
Taking logarithms of the relations (\ref{eq:monod-rel}) and (\ref{eq:monod-rel2}),
we obtain
\[
\begin{cases}
N_{1}'=4N_{2}-N_{1}+\Delta_{1,0}'\\
N_{2}'=N_{2}
\end{cases},\,\,\,\,\begin{cases}
N_{1}''=N_{1}\\
N_{2}''=4N_{1}-N_{2}+\Delta_{2,0}''
\end{cases}.
\]
\end{prop}

We identify the above relations with the relations (\ref{eq:vphi21})
and (\ref{eq:vphi31}) for the generators of the K\"ahler cones.
The extra terms $\Delta_{1,0}'$ and $\Delta_{2,0}''$ can be understood
as ``quantum corrections''. In fact, these corrections come from
$T_{E_{1}}$ and $T_{E_{1}''}$ representing the flopping curves (see
\cite[Sect.4.3]{HTmov}). Finally, using the connection matrix $\check{\rho}=\check{\varphi}_{b_{o_{1}}b_{o_{3}}}\circ\check{\varphi}_{b_{o_{3}}b_{o_{2}}}\circ\check{\varphi}_{b_{o_{2}}b_{o_{1}}}$,
the monodromy nilpotent cones are glued infinitely to make a larger
cone. This is the mirror picture of the movable cone $\mathrm{Mov}(X_{1})$.
See \cite{HTmov} for more details. There, another example having
nontrivial movable cone may be found. 

Birational geometry of complete intersection Calabi-Yau manifolds
has been studied recently in physics literature, see \cite{BCLuRue,BCLuRue2},\cite{BeHe}
and references therein. 

~

\subsection{\protect\label{subsec:Abel-CY}Calabi-Yau manifolds fibered by abelian
surfaces}

\noindent Homological mirror symmetry describes mirror symmetry as
an equivalence between the derived categories of coherent sheaves
and the derived Fukaya category of a mirror manifold. Therefore if
two Calabi-Yau manifolds are derived equivalent, then these two naturally
share a mirror Calabi-Yau manifold. As for the derived equivalence,
the following result is known due to Bridgeland \cite{Bridge}: 
\begin{thm}
In dimension three, birational Calabi-Yau manifolds are derived equivalent. 
\end{thm}

Combined with homological mirror symmetry, it is natural that we have
encountered the birational models of $X$ from the study of mirror
family $\check{\mathfrak{X}}\rightarrow\mathcal{M}_{\check{X}}$.
At the same time, one may expect that Calabi-Yau manifolds which are
not birational to a Calabi-Yau manifold but derived equivalent to
it arise from mirror symmetry. Such Calabi-Yau manifolds are called
Fourier-Mukai partners of a Calabi-Yau manifold, and the so-called
Grassmannian and Phaffian duality provides an interesting example
of Fourier-Muaki partners \cite{Ro,BCKvS,BoCa,HoriT}. (For K3 surfaces,
we have a general result for the numbers of Fourier-Mukai partners,
see \cite{FMpart} and reference therein.) The double cover of $Z_{3}$
in the preceding subsection (Remark \ref{rem:Reye-cong}), provides
another interesting example of Fourier-Mukai partners \cite{HTreye,HTreyeDx}.

Our example below turns out to be more interesting because we see
that both birational models and Fourier-Mukai partners arise from
its mirror family. To describe it briefly, let us consider an abelian
surface $A$ with its polarization $\mathcal{L}$ of $(1,d)$ type.
Such an abelian surface $A$ can be embedded into $\mathbb{P}^{d-1}$
by the linear system $|\mathcal{L}|$. Gross and Popescu have studied
the ideal $\mathcal{I}(A)$ of the image of this embedding \cite{GP}.
In particular, when the polarization is of type $(1,8)$, it was found
that the ideal contains four quadrics of the form; 
\[
f_{1}=\frac{w_{0}}{2}(x_{0}^{2}+x_{4}^{2})+w_{1}(x_{1}x_{7}+x_{3}x_{5})+w_{2}x_{2}x_{6},\,\,\,f_{i+1}=\sigma^{i}f_{1}\,\,\,\,(i=1,2,3),
\]
where $\sigma:x_{i}\mapsto x_{i+1},\tau:x_{i}\mapsto\xi^{-i}x_{i}\,(\xi^{8}=1)$
represent the actions of the Heisenberg group $\mathcal{H}_{8}=\langle\sigma,\tau\rangle$
on the homogeneous coordinates $x_{i}$ of $\mathbb{P}^{7}$. The
parameters $w_{i}=w_{i}(A)$ are determined by $A$ and represent
a point $[w_{0},w_{1},w_{2}]\in\mathbb{P}^{2}$. Thus the correspondence
$\text{A}\mapsto w(A)$ defines a (rational) map from an open set
of the moduli space $\mathcal{A}^{(1,8)}$ of $(1,8)$-polarized abelian
surfaces to $\mathbb{P}^{2}$. The following theorem is due to Gross
and Popescu \cite{GP}: 
\begin{thm}
$\mathcal{A}^{(1,8)}$is birational to a conic bundle over $\mathbb{P}^{2}$. 
\end{thm}

While the ideal $\mathcal{I}(A)$ is related to the moduli space $\mathcal{A}^{(1,8)}$
as above, the four quadrics contained in it define a complete intersection
in $\mathbb{P}^{7}$ which is a singular Calabi-Yau threefold. We
denote this Calabi-Yau variety by $V_{w}$. The following properties
of $V_{w}$ has been studied in \cite{GP}. 
\begin{prop}
The Calabi-Yau variety $V_{w}=V(f_{1},\cdots,f_{4})\subset\mathbb{P}^{4}$
has the following properties:

\begin{myitem}

\item[$(1)$] $V_{w}$ is a pencil of $(1,8)$-polarized abelian surfaces
which has 64 base points.

\item[$(2)$] Only at the 64 base points, $V_{w}$ is singular with
ordinary double points (ODPs). These singularities are resolved by
a blow-up along an abelian surface $A$, giving a small resolution
$X'\rightarrow V_{w}$.

\item[$(3)$] By flopping exceptional curves of the small resolution,
$X'\rightarrow V_{w}\leftarrow X$, we have a Calabi-Yau threefold
$X$ with fibers $(1,8)$-polarized abelian surfaces over $\mathbb{P}^{1}$.

\item[$(4)$] Both Calabi-Yau manifolds $X$ and $X'$ admit free
actions of the Heisenberg group $\mathcal{H}_{8}$ $($which actually
acts as $\mathbb{Z}_{8}\times\mathbb{Z}_{8}$ on $X$ and $X'$$)$.

\end{myitem} 
\end{prop}

The Hodge numbers of $X$ are determined as follows \cite{GP}: First
we obtain $h^{2,1}(X')=h^{2,1}(X)=2$ by studying the deformation
spaces. Since the singular fibers of $X$ turns out to be translation
scrolls of elliptic curves, we see that the topological Euler number
of $X$ is zero. This entails that $h^{1,1}(X)=h^{1,1}(X')=2$.

Using free actions of the Heisenberg group $\mathcal{H}_{8}$ (which
acts on $X$ as $\mathbb{Z}_{8}\times\mathbb{Z}_{8}$), we have quotient
Calabi-Yau manifolds 
\[
\check{X}:=X/\mathbb{Z}_{8},\,\,\,Y:=X/\mathbb{Z}_{8}\times\mathbb{Z}_{8}
\]
(and also $\check{X}':=X'/\mathbb{Z}_{8}$, $Y':=X'/\mathbb{Z}_{8}\times\mathbb{Z}_{8}$).
All these are Calabi-Yau manifolds with $h^{1,1}=h^{2,1}=2$. The
following interesting properties are shown in \cite{Sch,Bak}. 
\begin{prop}
The fiber-wise dual of the abelian-fibered Calabi-Yau manifold $X\rightarrow\mathbb{P}^{1}$
is isomorphic to $Y=X/\mathbb{Z}_{8}\times\mathbb{Z}_{8}$, and $X$
and $Y$ are derived equivalent. 
\end{prop}

These properties naturally motivate us studying mirror symmetry of
$X$. Regarding this, Gross and Pavanelli \cite{GPavanelli} made
an interesting conjecture: 
\begin{conjecture}
Mirror of $X$ is $\check{X}=X/\mathbb{Z}_{8}$, and mirror of $\check{X}$
is $Y=X/\mathbb{Z}_{8}\times\mathbb{Z}_{8}$. 
\end{conjecture}

Recently, by constructing families of $\check{X}$ and $Y$ and studying
LCSL boundary points carefully, the conjecture was verified affirmatively
\cite{HTAbCY} as follows: 
\begin{prop}
If we take a subgroup $\mathbb{Z}_{8}\simeq\langle\tau\rangle\subset\mathcal{H}_{8}$
to define $\check{X}=X/\mathbb{Z}_{8}$, then there exists a family
$\check{\mathfrak{X}}\rightarrow\mathcal{M}_{\check{X}}^{0}$ which
describes a mirror family of $X$, i.e., the B-structure coming from
a LCSL boundary point $o\in\mathcal{M}_{\check{X}}$ is isomorphic
to the A-structure of $X$. Also, there is another LCSL boundary point
$\tilde{o}$ in $\mathcal{M}_{\check{X}}$, and the B-structure from
$\tilde{o}$ is isomorphic to the A-structure of the Fourier-Mukai
partner $Y$ of $X$. 
\end{prop}

The parameter space of the family $\check{\mathfrak{X}}\rightarrow\mathcal{M}_{\check{X}}^{0}$
is given by a toric variety. Studying the boundary points in $\mathcal{M}_{\check{X}}\setminus\mathcal{M}_{\check{X}}^{0}$
in more detail, we actually find more LCSL boundary points other than
$o,\tilde{o}$ above, and find that the B-structures from them are
isomorphic to the birational models $X'$ and $Y'$ \cite{HTAbCY}.
Namely, we observe that the birational models as well as Fourier-Mukai
partners of $X$ emerge from the LCSL boundary points in the parameter
space of the family $\check{\mathfrak{X}}\rightarrow\mathcal{M}_{\check{X}}$.
We refer a recent work \cite{HTAbCY} for more details.

~

~

\section{\protect\label{sec:Summary}Summary~}

We have surveyed mirror symmetry of Calabi-Yau manifolds from the
perspective of period integrals of families of Calabi-Yau manifolds.
For elliptic curves or K3 surfaces, there is a well-studied theory
of period domains, based on which we can study the properties of period
maps canonically. For Calabi-Yau manifolds of dimension three, the
corresponding domain to study period integrals (or the image of period
maps) is not known. However, for families of Calabi-Yau hypersurfaces
or complete intersections defined by reflexive polytopes, we have
explicit expressions of period maps and beautiful relations to mirror
symmetry that arise from special boundary points. To determine the
image of period maps is still a difficult problem even for these families
of Calabi-Yau manifolds; however, these examples of Calabi-Yau manifolds
provide us with a good working field to deduce the properties of period
maps and the geometry of the moduli spaces of Calabi-Yau manifolds
in general. 

As we see in the second example in Sect.~\ref{sec:Birat-and-FM},
mirror symmetry is not restricted to Calabi-Yau hypersurfaces or complete
intersections in toric varieties but holds for much wider classes
of Calabi-Yau manifolds. A systematic method to generate such Calabi-Yau
manifolds is not known. However, when we restrict our attention to
Calabi-Yau manifolds which have one-dimensional deformation spaces,
there is an interesting approach initiated in \cite{AZ}. It was noted
that period integrals of one-dimensional families of Calabi-Yau manifolds
satisfy fourth-order differential equations with certain distinguished
properties, which are called Calabi-Yau differential equations/operators,
see \cite{vS} for a survey. In the table given in \cite{AvEvSZ},
possible forms of Calabi-Yau differential operators are listed by
generating operators with required properties. Though some of them
are identified with the Picard-Fuchs differential equations of the
corresponding families of Calabi-Yau manifolds, most of them are not.
Recently, it has been found \cite{KKSE,Schim} that many of them describes
special (one-dimensional) families of Calabi-Yau manifolds which arise
from families of Calabi-Yau complete intersections over $\mathcal{M}_{f_{1}\cdots f_{r}}$
of $\dim\mathcal{M}_{f_{1}\cdots f_{r}}\geq2$. It has been argued
that these one-dimensional loci in $\mathcal{M}_{f_{1}\cdots f_{r}}$
correspond to certain contractions (followed by partial smoothing)
of mirror Calabi-Yau manifolds of Picard numbers $\rho\geq2$. Since
these contractions are singular and only admit non-K\"ahler small
resolutions, they are called almost Calabi-Yau manifolds in \cite{Schim}. 

~

In this survey article, we focused mostly on the problem of constructing
families of Calabi-Yau manifolds from the viewpoint of mirror symmetry.
Before closing this survey, we address two related subjects. 

Mirror symmetry was discovered by studying quintic hypersurfaces and
their mirror family in 1991. Soon after that, in the late 90s, Candelas
et al \cite{CdelOV} started $p$-adic analysis of the period integrals
due to Dwork \cite{Dw1,Dw2} for the family of the mirror quintic
threefold to calculate the local Zeta functions. The basic object
to study is the Frobenius matrix for a $p$-adic family of Calabi-Yau
manifolds, which expresses the Frobenius action on the $p$-adic cohomology.
Recently, the general form of the Frobenius matrix was formulated
near the degeneration (LCSL) points \cite{Cand-Att}. Using this form,
the local Zeta functions were calculated for many interesting examples
of one-dimensional families of Calabi-Yau complete intersections.
From the calculations, interesting examples of the so-called rank
two attractor points, which correspond to the loci where certain splitting
of Hodge structure occurs, have been found \cite{Cand-Att,Cand-Att2}.
Further developments are expected in studying arithmetic aspects of
the moduli spaces of Calabi-Yau manifolds. 

The other important aspect of mirror symmetry, which we haven't addressed,
is the so-called geometric mirror symmetry, where two approaches are
known; one is the SYZ mirror construction due to Strominger, Yau,
and Zaslow \cite{SYZ}, and the other is called the Gross-Siebert
program \cite{GS1,GS2}. These approaches consider a certain degeneration
of Calabi-Yau manifolds and consider a certain duality operation in
the geometry of degenerations to make pairs of Calabi-Yau manifolds.
Such degenerations occur at the fibers over the LCSL points, which
we characterized by the monodromy property. In geometric mirror symmetry,
we consider a deformation family of Calabi-Yau manifolds only near
a degeneration point. Our global construction of the moduli spaces,
in particular, two examples given in Sect.~\ref{sec:Birat-and-FM},
naturally leads us to a question that how birational geometry or Fourier-Mukai
partners appear in the description of the geometric mirror symmetry. 

\newpage

\appendix
\renewcommand{\themyparagraph}{{\Alph{section}.\arabic{subsection}.\alph{myparagraph}}}

\section{\protect\label{sec:AppendixA}Useful formulas for the expansion of
$w_{0}^{T_{0}}(x,\frac{J}{2\pi i})$ }

We summarize useful formulas which express the expansion (\ref{eq:w0-J-symplectic})
of $w_{0}^{T_{0}}(x,\frac{J}{2\pi i})$. To simplify notation, we
write $w_{0}^{T_{0}}(x,\frac{J}{2\pi i})$ by $w_{0}(x,\hat{J})$
omitting the superscript $T_{o}$ and setting $\hat{J}:=\frac{J}{2\pi i}$.
Following \cite{HosJ}, we introduce slightly different cohomology
valued hypergeometric series; 
\[
w_{0}(x,\hat{J})=\sum c(n+\hat{J})x^{n+\hat{J}},\,\,\,\,\,\tilde{w}_{0}(x,\hat{J})=\sum c(n+\hat{J})x^{n}
\]
 and 
\[
w_{\s}(x,\hat{J})=\sum\frac{c(n+\hat{J})}{c(\hat{J})}x^{n+\hat{J}},\,\,\,\,\,\,\tilde{w}_{\s}(x,\hat{J})=\sum\frac{c(n+\hat{J})}{c(\hat{J})}x^{n}.
\]
As explained in the text (Subsect.~\ref{subsec:Coh-v-hypg-ser}),
the Taylor expansion with respect to $\hat{J}$ makes sense for these.
We write the expansion as 
\[
\begin{aligned} &  & w_{0}(x,\hat{J}) & =w_{0}^{(0)}(x)+\frac{1}{1!}w_{0}^{(1)}(x,\hat{J})+\frac{1}{2!}w_{0}^{(2)}(x,\hat{J})+\frac{1}{3!}w_{0}^{(3)}(x,\hat{J}),\\
 &  & \tilde{w_{0}}(x,\hat{J}) & =\tilde{w}_{0}^{(0)}(x)+\frac{1}{1!}\tilde{w}_{0}^{(1)}(x,\hat{J})+\frac{1}{2!}\tilde{w}_{0}^{(2)}(x,\hat{J})+\frac{1}{3!}\tilde{w}_{0}^{(3)}(x,\hat{J}),
\end{aligned}
\]
and the same definition for $w_{\s}(x,\hat{J}$) and $\tilde{w}_{\s}(x,\hat{J})$,
where by the superscript $(k)$ we indicate the degree $k$ part of
the Taylor expansion. By the definition of $\tilde{w}_{0}(x,\hat{J})$
and $\tilde{w}_{\s}(x,\hat{J})$, we note that no term containing
$\log x$ appear in the expansions for these two. It is easy to deduce
the following relations:
\begin{equation}
\begin{aligned}w_{0}^{(0)} & =\tilde{w_{0}^{(0)}},\\
w_{0}^{(1)} & =(\log x\cdot\hat{J})\tilde{w}_{0}^{(0)}+\tilde{w}_{0}^{(1)},\\
w_{0}^{(2)} & =(\log x\cdot\hat{J})^{2}\tilde{w}_{0}^{(0)}+2(\log x\cdot\hat{J})\tilde{w}_{0}^{(1)}+\tilde{w}_{0}^{(3)},\\
w_{0}^{(3)} & =(\log\cdot\hat{J})^{3}\tilde{w}_{0}^{(0)}+3(\log x\cdot\hat{J})^{2}\tilde{w}_{0}^{(1)}+3(\log x\cdot\hat{J})\tilde{w}_{0}^{(2)}+\tilde{w}_{0}^{(3)},
\end{aligned}
\label{eq:w-by-tilde-w}
\end{equation}
where $\log x\cdot\hat{J}:=\sum_{k}\log x_{k}\,\frac{J_{k}}{2\pi i}$.
We have relations in the same form for $w_{\s}^{(k)}$ and $\tilde{w}_{\s}^{(k)}$.
The following Lemma is shown in \cite[Lemma 6.6]{HosJ}\footnote{There is a mistake in the sign of the third term of Eq.(55) of \cite{HosJ}.
It should read as $\Psi^{(3)}(0)=+\frac{6\zeta(3)}{(2\pi i)^{3}}c_{3}(X_{\Delta_{s}})$.
Accordingly, the sings of $\zeta(3)$ in (58) and (59) of \cite{HosJ}
should be changed.}
\begin{lem}
\label{lem:A-1}We have the following relations: 
\begin{equation}
\begin{aligned}\tilde{w}_{0}^{(0)} & =\tilde{w}_{\s}^{(0)},\,\,\,\,\,\,\tilde{w}_{0}^{(1)}=\tilde{w}_{\s}^{(1)},\,\\
\tilde{w}_{0}^{(2)} & =\tilde{w}_{\s}^{(2)}-\frac{c_{2}}{12}\,\tilde{w}_{\s}^{(0)},\\
\tilde{w}_{0}^{(3)} & =\tilde{w}_{\s}^{(3)}-\frac{c_{2}}{4}\,\tilde{w}_{\s}^{(1)}+\frac{6\zeta(3)}{(2\pi i)^{3}}c_{3}\,\tilde{w}_{\s}^{(0)},
\end{aligned}
\label{eq:tw-tw-reg}
\end{equation}
where $c_{2}$ and $c_{3}$, respectively, represents the second and
third Chern class of $\hat{Z}_{f_{\Delta}}$. 
\end{lem}

The term $\zeta(3)$ (and also $\zeta(2)=\frac{\pi^{2}}{6}$) in the
right hand sides of (\ref{eq:tw-tw-reg}) appears when evaluating
derivatives of Gamma functions through the special values $\psi(1)=-\gamma$
(Euler's constant),$\psi'(1)=\zeta(2)$ and $\psi''(1)=-2\zeta(3)$
for $\psi(z)=\frac{d\Gamma(z)}{\Gamma(z)}$. These special values
of $\psi$ function do not appear for the expansion of $\tilde{w}_{\s}(x,\hat{J})$
since the ratio $\frac{c(n+\hat{J})}{c(\hat{J})}$ is expressed by
ratios of Gamma functions and contributions from the denominators
remove these special values, e.g., $\frac{d\;}{d\varepsilon}\frac{\Gamma(n+\varepsilon+1)}{\Gamma(\varepsilon+1)}=\frac{\Gamma(n+\varepsilon+1)}{\Gamma(\varepsilon+1)}(\psi(n+\varepsilon+1)-\psi(\varepsilon+1))$.
The subscript $\s$ in $\tilde{w}_{\s}$ and $w_{\s}$ indicates this
systematic subtraction. 
\begin{lem}
\label{lem:w-by-wreg}Corresponding to Lemma \ref{lem:A-1}, we have
\begin{equation}
\begin{aligned}w_{0}^{(0)} & =w_{\s}^{(0)},\,\,\,\,\,\,w_{0}^{(1)}=w_{\s}^{(1)},\,\\
w_{0}^{(2)} & =w_{\s}^{(2)}-\frac{c_{2}}{12}\,w_{\s}^{(0)},\\
w_{0}^{(3)} & =w_{\s}^{(3)}-\frac{c_{2}}{4}\,w_{\s}^{(1)}+\frac{6\zeta(3)}{(2\pi i)^{3}}c_{3}\,w_{\s}^{(0)},
\end{aligned}
\label{eq:w-w-reg}
\end{equation}
\end{lem}

\begin{proof}
When we substitute the relations (\ref{eq:tw-tw-reg}) into the r.h.s
of (\ref{eq:w-by-tilde-w}), the resulting terms are summarized into
the claimed forms. 
\end{proof}
We can solve the relation (\ref{eq:w0-J-symplectic}) for the period
integrals $\Pi_{0},\Pi_{i}^{(1)},\Pi_{j}^{(2)},\Pi^{(3)}$ as $\Pi^{(0)}=w_{0}^{(0)}$
and 
\[
\begin{aligned} & \sum_{i}\Pi_{i}^{(1)}\,J_{i}=\frac{1}{1!}w_{0}^{(1)}\,,\;\;\;\sum_{j}\Pi_{j}^{(2)}b_{j}^{(2)}=\frac{1}{2!}w_{0}^{(2)}+\sum_{k,i}a_{ki}\Pi_{i}^{(1)}b_{k}^{(2)}\\
 & \;\;\;\;\;\;\Pi^{(3)}\mathrm{vol}=-\frac{1}{3!}w_{0}^{(3)}-\frac{c_{2}}{12}\sum_{i}\Pi_{i}^{(1)}\,J_{i},
\end{aligned}
\]
where $b_{j}^{(2)}$ and $\mathrm{vol}$ are defined in (\ref{eq:basis-b}),
and satisfy $\int J_{i}b_{j}^{(2)}=\delta_{ij}$ and $\int\mathrm{vol}=1$.
Now using Lemma \ref{lem:w-by-wreg}, we can express period integrals
in terms of $\left\{ w_{\s}^{(k)}\right\} $ as follows: 
\begin{prop}
The period integrals are expressed by 
\[
\begin{aligned}\Pi_{0}(x) & =w_{\s}^{(0)}(x)\,,\,\,\,\,\,\,\,\sum_{i}\Pi_{i}^{(1)}\,J_{i}=\frac{1}{1!}w_{\s}^{(1)},\\
\sum_{j}\Pi_{j}^{(2)}b_{j}^{(2)} & =\frac{1}{2!}w_{\s}^{(2)}-\frac{c_{2}}{24}w_{\s}^{(0)}+\sum_{k,i}a_{ki}\Pi_{i}^{(1)}b_{k}^{(2)},\\
\Pi^{(3)}\mathrm{vol} & =-\frac{1}{3!}w_{\s}^{(3)}-\frac{c_{2}}{12}w_{\s}^{(1)}-\frac{\zeta(3)}{(2\pi i)^{3}}c_{3}\,w_{\s}^{(0)}.
\end{aligned}
\]
\end{prop}

If we use the following definition in terms of $J$ (not $\hat{J}$),
\[
w_{\s}^{(1)}=\sum_{i}w_{\s,i}^{(1)}J_{i},\,\,\,w_{\s}^{(2)}=\sum_{k,l}w_{\s,kl}^{(2)}J_{k}J_{l},\,\,\,w_{\s}^{(3)}=\sum_{k,l,m}w_{\s,klm}^{(3)}J_{k}J_{l}J_{m},
\]
after integration, we obtain more familiar form of the integral basis,
\begin{equation}
\begin{aligned}\Pi_{0} & =w_{\s}^{(0)},\,\,\,\,\,\,\Pi_{i}^{(1)}=w_{\s,i}^{(1)},\\
\Pi_{k}^{(2)} & =\frac{1}{2}\sum_{l,m}K_{klm}w_{\s,lm}^{(2)}-\frac{c_{2}\cdot J_{k}}{24}w_{\s}^{(0)}+\sum_{i}a_{ki}w_{\s,i}^{(1)},\\
\Pi^{(3)} & =-\frac{1}{6}\sum_{k,l,m}K_{klm}w_{\s,klm}^{(3)}-\sum_{k}\frac{c_{2}\cdot J_{k}}{24}w_{\s,k}^{(1)}-\frac{\zeta(3)}{(2\pi i)^{3}}\chi\,w_{\s}^{(0)}
\end{aligned}
\label{eq:canonial-from-Pi}
\end{equation}
with $K_{klm}:=\int J_{k}J_{l}J_{m}$, $c_{2}\cdot J_{k}=\int c_{2}J_{k}$
and $\chi=\int c_{3}$. This form of integral and symplectic basis
of period integrals is always assumed when we study mirror symmetry. 
\begin{rem}
It is widely observed that the final form (\ref{eq:canonial-from-Pi})
applies not only to hypergeometric series $w_{0}^{T_{o}}(x)$ but
more general power series solutions arising the boundary points (LCSLs).
For example, the formula (C.2) of \cite{HTAbCY} can be recognized
as the same as (\ref{eq:canonial-from-Pi}) but the regular power
series solution $\Pi_{0}(x)=w_{\s}(=w_{0})$ for the family is not
a hypergeometric series. Nevertheless, the general form $(\ref{eq:canonial-from-Pi})$
gives us an integral symplectic basis for the family. 
\end{rem}

~

~

{\footnotesize Department of Mathematics, Gakushuin University,}{\footnotesize\par}

{\footnotesize Mejiro, Toshima-ku, Tokyo 171-8588, Japan}{\footnotesize\par}

{\footnotesize e-mail: hosono@math.gakushuin.ac.jp}{\footnotesize\par}
\end{document}